\documentclass[reqno]{amsart}

\usepackage{graphicx}
 \usepackage{bbm}
  \usepackage{bm}
\usepackage{amsfonts,dsfont,amsmath,enumerate}
\usepackage{amsbsy}
\usepackage{graphicx}
\usepackage{oldgerm}
\usepackage{yhmath}
\usepackage{wrapfig}
 \usepackage{bm,amsthm,amsopn,amssymb,dsfont, oldgerm,yhmath}
\usepackage{color}
\definecolor{greencite}{rgb}{0.2,0.6,0.2}
\definecolor{bluformula}{rgb}{0.1,0.2,0.6}
\definecolor{rosso}{rgb}{1,0,0}

 \usepackage[colorlinks=true,linkcolor=bluformula,urlcolor=blue,citecolor=greencite]{hyperref}

\theoremstyle{plain}
\newtheorem{theorem}{Theorem}
\newtheorem{theorem*}{Theorem}
\newtheorem{lemma}{Lemma}
\newtheorem{proposition}{Proposition}
\newtheorem{corollary}{Corollary}

\theoremstyle{definition}

\theoremstyle{plain}
\newtheorem{remark}{Remark}
\newtheorem*{remark*}{Remark}
\numberwithin{equation}{section}

\def\e{{\epsilon}}

\def\ni{\noindent}
\def\OT{{\O\!\times\!(0,T)}}
\def\OTY{{\O\!\times\!(0,T)\!\times \!Y}}

\def\re{\rho_\e }

\def\ln{\left|\left|}
\def\rn{\right|\right|}
\def\lb{\left|}
\def\rb{\right|}

\def\ov{\overline}
\def \tr {\mathop {\rm tr}}
\def\le{\leq}

\def\lc{\left[}
\def\rc{\right]}

\def\si3{\sum_{i=1}^3}

\def\OYs{{\O \times Y\setminus B}}

\def\rpt{\right.}

\def\O{\Omega}
\def\ov{\overline}
\def\e{\varepsilon}

\def\a{\alpha}
\def\b{\beta}

\def\l{\lambda}
\def\a{\alpha}
\def\b{\beta}

\def\OY{{\O\time1s Y}}
 \def\OYs{{\O\!\times \!(Y\setminus B)}}

 \def\xe{{x  \over \e}}

\def\OT{{\O\!\times\!(0,T)}}

\def\sdto{{\longrightarrow \!\!\!\!\!\!\!\!\longrightarrow}}
\def\dto{\rightharpoonup \!\!\!\!\!\!\!\!\rightharpoonup}

\def\bfdiv{{\rm \bf div}} 
\def\dist{{\rm dist}}
\def\lp{\left(}
\def\rp{\right)}
\def\la{\left\{}
\def\ra{\right\}}

\def\build#1_#2^#3{\mathrel{\mathop{\kern 0pt #1}\limits_{#2}^{#3}}}

\def\ru{\ov\rho_1}

\def\lime{\lim_{\e\to 0}} 

\def\bfa{\pmb{a}} 
\def\bfb{\pmb{b}}

\def\bfe{\pmb{e}}  
\def\bff{\pmb{f}}  
\def\bfg{\pmb{g}}  
\def\bfh{\pmb{h}}  
\def\bfi{\pmb{i}}

\def\bfu{\pmb{u}}  
\def\bfv{\pmb{v}}  
\def\bfw{\pmb{w}}  
\def\bfx{\pmb{x}}

\def\bfA{\pmb{A}} 
\def\bfB{\pmb{B}}

\def\bfH{\pmb{H}} 
\def\bfI{\pmb{I}}

\def\B{{\mathcal B}}

\def\D{{\mathcal D}}

\def\F{{\mathcal F}}

\def\H{{\mathcal H}}
\def\I{{\mathcal I}}

\def\L{{\mathcal L}}
\def\M{{\mathcal M}}

\def\P{{\mathcal P}}

\def\NN{\mathbb{N}}

\def\RR{\mathbb{R}}  
\def\SS{\mathbb{S}}

\def\ZZ{\mathbb{Z}}
\def\rr{\mathbb{R}}

\def\bfnu{{\pmb{ \nu}}} 
\def\bfxi{{\pmb{ \xi}}}

\def\bfsigma{{\pmb{ \sigma}}} 
  
\def\bfphi{{\pmb{  \phi}}} 
 
\def\bfpsi{{\pmb{  \psi}}}

\def\bfvarphi{{\pmb{ \varphi}}}

\def\bfXi{\pmb{ \Xi}}

\def\bfPsi{\pmb{  \Psi}}

\def\gotN{{{\mathfrak{N}}}}
\def\gotO{{{\mathfrak{O}}}}

\def\gotO{{{\mathfrak{O}}}}

  \def\bfnabla{\pmb{  \nabla}}

\def\?{{\bf ???}}
\def\OT{{\O\times (0,t_1)}}
\def\mdto{{\,\buildrel m_\e \over \dto\,}}

\def\implique{\quad  \Longrightarrow  \quad}

\def\lpt{\left.}
\def\rpt{\right.}
\def\f{\varphi}

\def\intb{{\int\!\!\!\!\!\!-}}

\def\.{\cdot}
  \def\OTI{\OT\times I}
\def\OTB{{\O\!\times\!(0,T)\!\times\! B}}

\def\implique{\quad  \Longrightarrow  \quad}

\def\lpt{\left.}
\def\rpt{\right.}
\def\f{\varphi}

\def\intb{{\int\!\!\!\!\!\!-}}

\def\xpe{{x'\over\e}}

\begin{document}

\bibliographystyle{plain}

\title[]{HOMOGENIZATION OF  STRATIFIED ELASTIC media
   WITH HIGH CONTRAST }

\author{Michel Bellieud}
 \address{ Universit\'e  de Montpellier, Case courier
048, Place Eug\`ene Bataillon, 34095 Montpellier Cedex 5, France}
 \email{michel.bellieud@univ-montp2.fr}

\date{}
 

\begin{abstract}
  We determine the asymptotic behavior of the solutions to the linear  elastodynamic equations
in a stratified medium   comprising an alternation of possibly very  stiff layers  with much softer ones, 
when  the thickness of the layers tends to zero. The limit equations may depend on higher order terms, characterizing bending effects. 
A part  of this work is   set in the context of non-periodic homogenization and an extension to  stochastic homogenization is presented.
\end{abstract}

\maketitle

\centerline{Preliminary version}


 \section{Introduction}
\label{intro}
{\color{black}
In this paper, 
we  analyze the asymptotic behavior
of the solution    to the linear elastodynamic equations    
in  a 
 composite material 
 wherein, at a microscopic scale,    
  possibly very "stiff"  
  layers  
alternate   with  a  much "softer" medium.}
%
 Stratified   composite media 
  have been intensively investigated over the last decades, {\color{black} especially}      in the context  of diffusion equations \cite{BoPi,GuMo,GuMoPi,HeMo,HeMoPi,Hn,MuTa,TaR,TaP}. 
As regards   linear elasticity, 
 layered  elastic composites 
{\color{black}  have}
  been studied in  {\color{black}\cite{FrMu, GuMo2, JaLiLu, McCo}}  under  assumptions
of uniform boundedness and uniform definite positiveness of 
 the elasticity  tensor 
{\color{black}   guaranteeing  that the effective equation 
 is   a standart   linear elasticity equation.   
When  these assumptions break down, as for instance  in the so-called   "high contrast case",
 the limit  equilibrium equation   may be of a quite different type: it may   correspond,
in theory,   to  the Euler equation associated to the minimization of 
 any lower semi-continuous quadratic form  on $L^2$ vanishing on rigid motions \cite{CaSe}.
In particular,  it  may be non-local and   depend on higher order derivatives of the displacement.
Elastic media with  high contrast have been studied under  various geometrical assumptions.
Composites with stiff grain-like  inclusions  have been investigated in  \cite{BeArma,  BeSiam, Pa2}, 
  stiff fibered structures in 
\cite{BeSiam, BeBo2, BeGr,PiSe, Sili}, and  stiff media with holes filled with a soft material in \cite{ChCh, Co, Sa1}.
Our aim is to complement this body of work in 
the context  of stratified  media. 
Our approach is  based on the two-scale convergence  method \cite{Al,ArDoHo,BraBri,
ChSmZh,Ne,Ng},   which yields    the convergence  to an effective  solution. It also yields     a first order  corrector result in $L^2$ (see Remark \ref{remcorrector}), but not 
the    rigorous  error estimates of higher order with respect to small parameters 
 provided by the  asymptotic expansions method   \cite{AbKoPaSm,AmKoPaSm,BaKaSm,BiSu,BiSu2,Ch,PaSu,Pa,Pa2, 
Sa,Sa07}.

For a given     bounded smooth    open subset $\O$ of $\RR^3$,  we  consider a  linear elastodynamic problem like \eqref{Pe}.
We  assume  that    the  Lam\'e  coefficients 
  take  
 possibly large  values   in a 
   subset ${B_\e}$  of $\O$  and much smaller values elsewhere. 
 The set $B_\e$   consists 
 of a non-periodic distribution of   parallel disjoint  homothetic  layers   of thickness  $r_\e$,
 whose median planes are orthogonal to $\bfe_3$  and  separated by a minimal distance $\e$, where   $\e, r_\e$ are 
 positive reals converging to zero (see fig. 3.1).
 The effective volume fraction of the stiff phase is characterized  by the parameter $\vartheta$ defined by \eqref{defvartheta}. 
 Both cases  $\vartheta=0$ and  $0<\vartheta<1$ are investigated.
The  order of magnitude  of the Lam\'e  coefficients in the 
stiff phase is determined  by the 
 parameters $k$ and $\kappa$ defined by (\ref{kkappa}).  

When the    elasticity  coefficients in the soft phase  are of order    $1$
and  the effective  volume fraction of the stiff phase vanishes
   the limit behavior of the composite is governed, 
if  $0<k<+\infty$,  by the equation 
 
\begin{equation}
\label{Pvintro2}
\begin{aligned}
  ( \rho +n{\overline \rho_1})  {\partial^2\bfu\over \partial t^2}   
 -\bfdiv   \bfsigma ( \bfu)
 - nk   \bfdiv\bfsigma_{x'}(\bfu')
 =( \rho +n\ru ) \bff   \hskip 0,5cm    
\hbox{ in } \OT,
  \end{aligned}
  \end{equation}

\ni where $\rho$ denotes the mass density in the softer phase, and   $\bfu'$,   $\bfsigma_{x'}$, $\bfsigma$ and 
 $\ru$  are defined, respectively,  by  
  (\ref{exprim}),  (\ref{defsigma}),  and (\ref{defrhoe}). 
The function $n$  characterizes the rescalled  effective  
 number of  sections of stiff layers  per unit length in the $\bfe_3$ direction and  is obtained as 
 the weak* limit in $L^\infty(\Omega)$ of the sequence $(n_\e)$ defined by \eqref{defne}.
When    the order of magnitude of the elasticity coefficients in the stiff layers  is larger,  that is when $k=+\infty$, 
   the  functions 
$u_1$  and   $u_2$  vanish on  the set $\{n>0\}$ and 
the behavior of $u_3$
   is       governed  by the  equation  
 (\ref{infty0f}),
(\ref{inftykappaf}) or (\ref{inftyinftyf}), depending on   the order  of magnitude of 
$\kappa$.
In the 
case     $0\!<\!\kappa\!<\!+\infty$,    this equation     involves    the $4^{th}$  partial derivatives  of   
  $u_3$ with respect to $x_1, x_2$:

\begin{equation}
\label{bendingintro2}
\begin{aligned}
  ( \rho +n{\overline \rho_1})  {\partial^2
u_3
\over \partial t^2}  -(\bfdiv\bfsigma(\bfu))_3  + n{ \kappa  \over 3}{l+1\over l+2}&\sum_{\a,\b=1}^2 {\partial^4  u_3 \over \partial x_\a^2\partial x_\b^2}  
 \\& =  ( \rho +n{\overline \rho_1})  f_3   \ \hbox{ in } \OT,
    \end{aligned}
  \end{equation}

 \ni  
 revealing  bending    effects. 
The effective behavior on the set $\{n=0\}$ is that of a homogeneous material without stiff layers.
%
In Theorem \ref{thstoch}, we extend these results to the stochastic case. 
The set $B_\e(\omega)$ then depends
  on
  a random element $\omega$ of some sample space $\gotO\subset 2^\RR$
equiped with a probability   $P$ satisfying \eqref{P}. 
The limit  problem as $\e\to0$  is deduced  from the above equations, $P$-almost surely, 
by substituting for 
  $n$    the conditional expectation $ {E}_{{P}}^{\F} n_0(\omega)$  with respect to $P$ 
given the   $\sigma$-algebra  $\F$ 
  of  the periodic sets,   
 of the random variable  $n_0$ defined by \eqref{defn0}.

If  the  
 order of magnitude of  the elasticity 
 coefficients  in  the soft interlayers  is strictly smaller than $1$ and strictly  larger than $\e^2$, 
 the effective equations
are  deduced from (\ref{Pvintro2}),  (\ref{bendingintro2}), formally, 
by removing the term $\bfdiv \bfsigma (\bfu)$ (see   Theorem \ref{thinter}).  
%
%

 When 
  the elastic moduli in the soft phase 
are of order $\e^2$, 
the effective behavior of the composite turns sensitive  to    the  slightest  geometrical   perturbation  (see Remark \ref{remper}).
The effective equation can not be expressed   simply in terms of the function $n$ as in the other cases. This characteristic renders  the study of non-periodic homogenization a very difficult task: we only treat   the  case of an $\e$-periodic distribution of stiff layers.
The  homogenized problem   then takes the form of   a system  of equations  
coupling   some  field $\bfv$, characterizing the effective displacement in the stiff layers, 
with  the two-scale limit 
 $\bfu_0:\OT\times (-{1\over 2}, {1\over 2})^3\to \RR^3$  of the solution  $(\bfu_\e)$ to \eqref{Pe}  (see \cite{Al,Ng}).
This  field  $\bfv$ is obtained as  the limit of the sequence $(\bfu_\e m_\e)$, where $m_\e$ is the measure    supported by the stiff layers defined by \eqref{defme}.  
If $0<k<+\infty$, the  effective  behavior of the displacement in the stiff 
medium  
  is governed   by the      equation   

\begin{equation} \begin{aligned}
 {\overline \rho_1}\! {\partial^2\bfv\over \partial t^2}   
-k   \bfdiv\bfsigma_{x'}(\bfv')
 =\!\ru  \bff +    \bfg(\bfu_0)\qquad \hbox{  in }\   \Omega\times(0,T), 
 \end{aligned}
\label{Pvtintro}
\end{equation}

\ni associated with  the  boundary and initial  conditions given in (\ref{k0}).
This
 equation      displays     stretching    vibrations with regard to the 
transversal components $v_1$, $v_2$  of $\bfv$. 
It is coupled with the soft phase through 
 the   field   $\bfg(\bfu_0)$ which  represents  the sum of the  surface forces   applied on each stiff   layer   by the adjacent  soft medium. This field  is  defined by (\ref{eysigmayg}),   in terms of   the restriction  
of    $\bfu_0$ to ${\Omega\times(0,T)\times  Y\setminus A}$, which  characterizes  the effective displacement in  the soft interstitial layers. The letters      $Y$   and     $A$  
symbolize, respectively,     the unit cell  and   the rescaled  stiff layer (see \eqref{defYBSigma}, \eqref{defA}).
The effective displacement in the soft  phase   is governed   by the   equation  

\begin{equation}
\rho {\partial^2\bfu_0\over \partial t^2}-\bfdiv_y (\bfsigma_{0y} (\bfu_0 ))= \rho \bff\qquad \hbox{  in }\  {\Omega\times(0,T)\times  Y\setminus A} , 
 \nonumber
\end{equation}

\ni where $ \bfsigma_{0y}$ is defined by (\ref{eysigmayg}). This equation is    coupled with the  variable  $\bfv $  by  the 
relation (\ref{coupling})  on $ \OT\times A$.
The weak limit    of    $(\bfu_\e)$ in     $L^2(\Omega\times(0,T))$ 
is given by 
$\bfu (x,t)=\int_Y
\bfu_0 (x,t,y)dy$. 

When    the order of magnitude of the elasticity coefficients in the stiff layers  is larger,
   the  functions 
$v_1$  and   $v_2$  vanish  and 
the effective displacement  in the stiff phase   is       governed  by the  equation  of   $v_3$ 
given by (\ref{infty0}),
(\ref{inftykappa}) or (\ref{inftyinfty}), depending on   the order  of magnitude of 
$\kappa$.
In the 
case     $0\!<\!\kappa\!<\!+\infty$,    this equation,   
\begin{equation}
   {\overline \rho_1} {\partial^2
v_3
\over \partial t^2} \!  +\!{\kappa  \over3} {l\!+\!1\over l\!+\!2}\!\!\sum_{\a,\b=1}^2\!\! {\partial^4  v_3 \over \partial x_\a^2\partial x_\b^2}  \!
 =  \ru\!  f_3+ (\bfg(\bfu_0))_3    \quad \hbox{ in } \OT,
 \nonumber
\end{equation}

\ni       involves    the $4^{th}$  partial derivatives  of   
  $v_3$ with respect to $x_1, x_2$, 
 characterizing   bending   vibrations.
   Otherwise, the stiff layers   display the behavior of a collection of 
           unstretchable membranes  if $(k,\kappa)=(+\infty,0)$   and that   of  fixed   bodies  if $\kappa=\infty$.

Our  results apply as well to the case of equilibrium equations 
and    to multiphase composites (see  remarks \ref{remmultiphase}, \ref{remequilibrium}, \ref{remmultistiffinter},  \ref{remequilibriumstiffinter}). 
  The paper is organised as follows: in Section \ref{secnotations}   we specify the notations  and 
in  Section \ref{secresults}   we state our main results.  
In Section \ref{sectwoscale}, we recall some classical  results 
and introduce 
a non-periodic variant of the two-scale convergence for which we establish a compactness result.  
The effective equations are derived in Section \ref{secprooftheoremth} by employing 
  apriori estimates  demonstrated in Section \ref{secapriori}, and a technical lemma  proved in the appendix. 
}

 \section{Notations}\label{secnotations} 
  In this article, $\{ \bfe_1, .., \bfe_N\}$ stands for the canonical basis of   $\RR^N$.
Points in $\RR^N$ or in $\ZZ^N$ and real-valued functions  are represented by symbols beginning by a  lightface lowercase  (example $x, i, \det \bfA...$) and vectors and vector-valued functions   by symbols
beginning by a boldface lowercase (examples: $\bfx$, 
 $\bfi$,   
$\bfu$,
$\bff$, $\bfg$,
$\bfdiv  \bfsigma_\e $,...). Matrices and matrix-valued functions are represented  by symbols beginning by a boldface uppercase
with the following exceptions:
$\bfnabla\bfu$ (displacement gradient), $\bfe(\bfu)$ (linearized strain tensor).
We denote by $u_i$ or
$(\bfu)_i$  the  components of a vector $\bfu$ and   by $A_{ij}$ or $(\bfA)_{ij}$ those of a matrix $\bfA$ (that is $\bfu=
\sum_{i=1}^N u_i \bfe_i =\sum_{i=1}^N (\bfu)_i
\bfe_i$; $\bfA=\sum_{i,j=1}^N A_{ij} \bfe_i\otimes\bfe_j =\sum_{i,j=1}^N (\bfA)_{ij} \bfe_i\otimes\bfe_j $). We do not employ  the usual repeated index
convention for summation.    We denote 
by $\bfA\!:\!\bfB=\sum_{i,j=1}^N A_{ij}B_{ij}$ the   inner
product of two matrices, by $ \e_{ijk} $   the three-dimensional   alternator, by 
$\bfu\wedge \bfv=\sum_{i,j,k=1}^3\e_{ijk}u_jv_k\bfe_i$ the exterior product in $\RR^3$,   by $\SS^M$ ($M\in \NN$)
   the set of all real symmetric matrices of order
$M$, by $\bfI_M$   the $M\times M$ identity matrix.  
{\color{black} The symbol   $\sharp D$ denotes   the cardinality of  a finite set $D$. }
The letter
$C$ stands for  different  constants  whose precise values  may vary.
For any weakly differentiable vector field   $\bfpsi:\Omega\subset  \RR^3\to \RR^3$, we set 

\begin{equation}
\begin{aligned} 
 &  \bfpsi':= \psi_1\bfe_1+\psi_2\bfe_2; \quad 
%
%
 \bfe_{x'}(\bfpsi):= \sum_{\a,\b=1}^2 {1\over2}\lp {\partial \psi_\a\over\partial x_\b}+ {\partial \psi_\b\over\partial x_\a}\rp \bfe_\a\otimes\bfe_\b\  (=\bfe_{x'}(\bfpsi'));
 \\  &  
\bfsigma_{x'}(\bfpsi ):= \begin{pmatrix}2 { \partial  \psi_{   1} \over \partial x_1}+ {2l \over l +2}\lp { \partial \psi_{   1} \over \partial x_1}+{ \partial  \psi_{ 2} \over \partial x_2}\rp&  { \partial \psi_{  1} \over \partial x_2}+{ \partial  \psi_{   2} \over \partial x_1} & 0
\cr   { \partial \psi_{  1} \over \partial x_2}+{ \partial  \psi_{  2} \over \partial x_1} &  2{ \partial \psi_{ 2} \over \partial x_2}+{2l\over l +2}\lp { \partial  \psi_{   1} \over \partial x_1}+{ \partial  \psi_{  2} \over \partial x_2}\rp&0
\cr 0 &  0& 0
\end{pmatrix}.
\end{aligned} 
\label{exprim}  
   \end{equation}

  \ni  {\color{black} We reproduce and modify here some notations from \cite{BeSiam}:} we denote by $C_\sharp^\infty(Y)$ (resp. $C_\sharp(Y)$)  the set of $Y$-periodic functions of $C^\infty\lp \RR^3\rp$ (resp. $C(\RR^3)$), by $C^\infty_\sharp(Y\setminus {B})$ the set of the
restrictions of the elements of 
$C_\sharp^\infty(Y)$ to $Y\setminus {B}$,   by
$H^1_\sharp(Y)$ (resp. $H^1_\sharp(Y\setminus {B})$) the completion of $C_\sharp^\infty(Y)$ (resp. $C^\infty_\sharp(Y\setminus {B})$)  with respect to the norm $w\to \lp \int_Y(|w|^2+|\bfnabla w|^2)dy
\rp^{1\over2}$ 
( resp. $w\to( \int_{Y\setminus {B}}(|w|^2   + |\bfnabla w|^2)dy
)^{1\over2}$  ).   
For any  subset $Q$ of the unit cell $Y$,  the symbol $Q_\sharp$ stands for  the periodization on all $\RR^3$ of $Q$,
 that is 

\begin{equation}
\begin{aligned}
Q_\sharp:= \bigcup_{z\in \ZZ^3} \{z\}+Q.
\end{aligned}
  	\label{Qsharp}
\end{equation}

 \section{Setting  of the problem and results}\label{secresults} 
 We consider a cylindrical domain  $\Omega :=\Omega'\times  (0,L) $, where 
  $\Omega'$ is 
a  bounded smooth  domain of $ \RR^2$.
Given a small positive real $\e$, the non-periodic distribution $B_\e$ of  disjoint  homothetical  stiff layers $B_\e^{j}$   will be described in terms of   a finite 
  subset $ \omega_\e$  of   $\RR$
 
  \begin{equation} \label{defomegae}
 \omega_\e := \la \omega_\e^{j}, j\in   J_\e\ra ,
 \end{equation}
 
\ni  satisfying 
 
  \begin{equation}  \label{condomegae}
  \omega_\e\subset (0,L), \quad 
\min_{j,j'\in J_\e, j\not= j'} |\omega_\e^{j}-\omega_\e^{j'}|=\e, 
\quad \dist(\omega_\e, \{0,L\}) > \frac{\e}{2}, 
 \end{equation}
 
\ni  and of a small parameters  $r_\e$ verifying  

\begin{equation}
 \begin{aligned}
 &
\e >r_\e(1+\delta) \qquad \hbox{for some $\delta>0$},
   \end{aligned}
  	\label{disjoint}
\end{equation}

\ni by setting    (see Fig. 3.1)

  \begin{equation}
 \begin{aligned}
 &{B_\e}:=   
  \bigcup_{j\in J_\e} B^{j}_\e; \quad B^{j}_\e :=
  \Omega'\times    
   \lp \omega_\e^{j}+ r_\e I\rp; \quad I:=  \lp -\frac{1}{2},  \frac{1}{2}\rp.
 \end{aligned}
  	\label{defBe}
\end{equation}


 \begin{figure}[h]\label{fig}
 \centering
 \includegraphics[height=5.5cm]{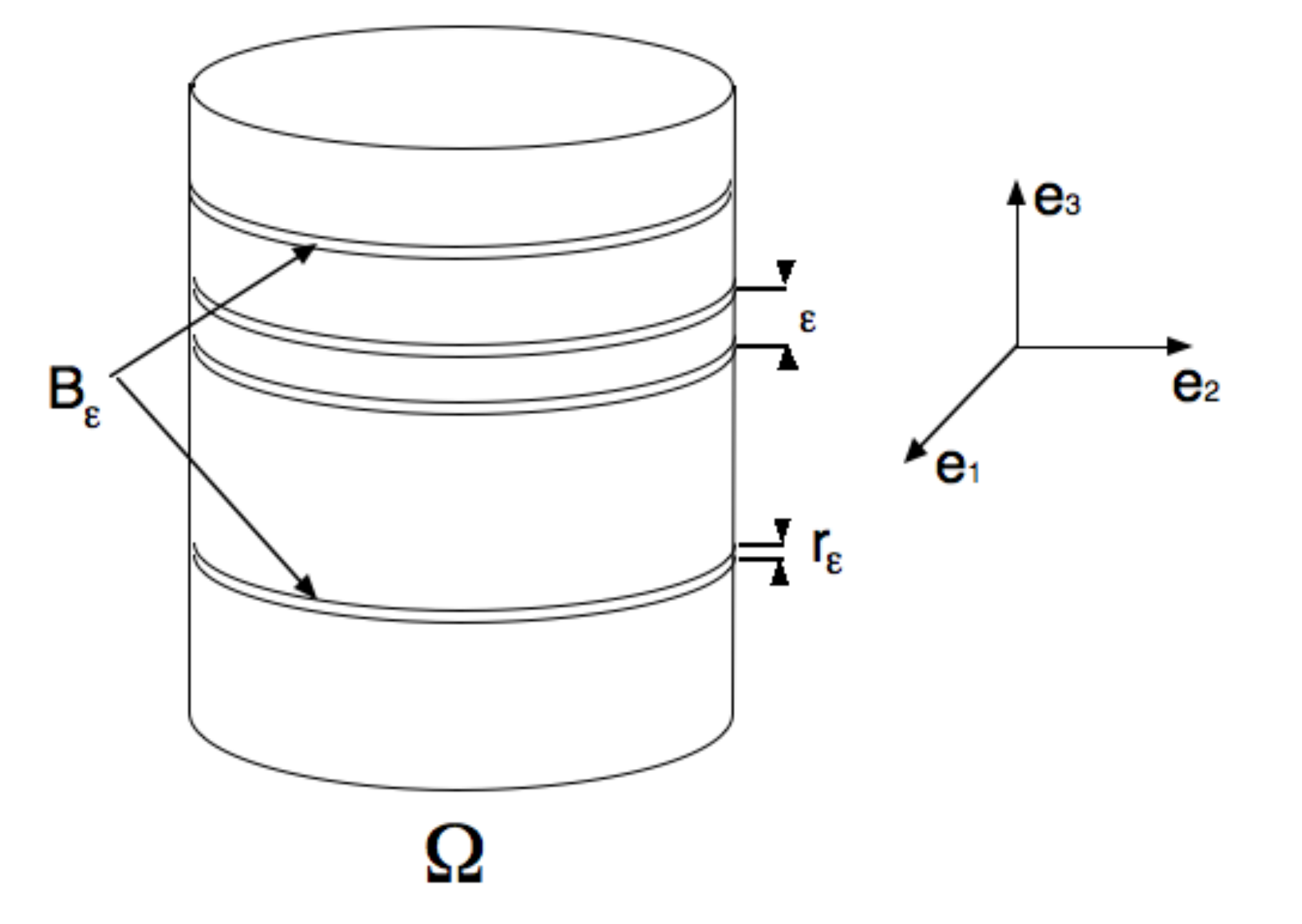} 
 \caption{} 
  \end{figure}
 
\ni 

  {\color{black}

\ni As in \cite{BeSiam}, we  consider the vibration problem 

{\color{black}
\begin{equation}
(\P_\e):\la\begin{aligned}
 &\rho_\e {\partial^2 \bfu_\e\over \partial t^2} - \bfdiv(\bfsigma_\e( \bfu_\e)) =\rho_\e \bff  \qquad  \hbox{ in }    \OT,
  \\ & \bfsigma_\e( \bfu_\e)=\l_\e \tr(\bfe(\bfu_\e))\bfI + 2 \mu_\e \bfe(\bfu_\e), \quad  \bfe(\bfu_\e)= {1\over 2} (\bfnabla \bfu_\e +
 \bfnabla^T
\bfu_\e),
\\ & \bfu_\e\in C ([ 0,\,T] ;\,H_0^1(\O, \RR^3))\cap
C^1( [0,\,T] ;\,L^2(\O, \RR^3)),  
\\ & \bfu_\e( 0) = \bfa_0 ,\quad {\partial \bfu_\e \over \partial t}( 0) =
\bfb_0  ,
\quad  \bff\in C(\ov{\OT}; \RR^3), 
 \\&(\bfa_0,\bfb_0)\in  \lp H^1_0(\O,\RR^3) \times
L^2( \O, \RR^3)\rp  \cap C(\ov\O,\RR^3)^2.
\end{aligned}\rpt
\label{Pe}
\end{equation}
\ni}
 
\ni The  Lam\'e coefficients  $\mu_\e$, $\l_\e$ and the mass density $\rho_\e$   are assumed to take  constant values of possibly different orders of magnitude in the set of  layers $B_\e$ and
in the set of  interlayers $\Omega\setminus B_\e$. More precisely, we suppose  that

   \begin{equation}
\begin{aligned}
 &  \mu_\e(x)=\mu_{1\e}\mathds{1}_{{B_\e}}(x)+ \mu_{0\e} \mathds{1}_{\O\setminus {B_\e}}(x) ,\quad  \l_\e(x)=\l_{1\e}\mathds{1}_{{B_\e}}(x)+ \l_{0\e} \mathds{1}_{\O\setminus {B_\e}}(x) ,   
 \\ &  \mu_{1\e}\ge c>0,  \qquad l_\e:= {\l_{1\e}\over \mu_{1\e}}, \qquad    \lim_{\e\to 0} l_\e=  l \in  [0,+\infty ), \quad  0<\mu_{0\e}<\!\!<\mu_{\e1}, 
\end{aligned}
\label{lmu}\end{equation}

 \ni  and 
 
 \begin{equation}
  \rho_\e(x)= \rho\mathds{1}_{\Omega\setminus B_\e} + \frac{\e}{r_\e} \ov\rho_1 \mathds{1}_{ B_\e}, 
  \qquad \rho, \ov\rho_1\in (0,+\infty).
  	\label{defrhoe}
\end{equation}


\ni We  assume and set

 \begin{equation}
 k:= \lim_{\e \to 0} \frac{r_\e}{\e} \mu_{1\e}\in (0,+\infty],\qquad\qquad    \kappa:= \lim_{\e \to 0}  \frac{r_\e^3}{\e} \mu_{1\e}  \in [0,+\infty].
  	\label{kkappa}
\end{equation}

 \ni The weak* relative compactness in   $L^\infty(0,T; L^2(\O;\RR^3))$  of the sequence of the solutions to (\ref{Pe}) is ensured by  the following hypothesis:

   \begin{equation}
\begin{aligned} & \sup_{\e>0}   \int_\O \! \!\lp\re\!\left| \bfb_0 \right|^2  \!\! +\! \bfsigma_\e ( \bfa_0 )\!:\!\bfe(\bfa_0  )\!\rp \!dx+\int_\OT  \! \re 
\left|
\bff\right|^2\!\!dxdt<+\infty.
\end{aligned}
\label{hyp2}\end{equation}

}
 

 \subsection{ Case of   interlayers  with Lam\'e coefficients of order 1
   }\label{subsecstiff}


\ni We consider the case of extremely thin layers of extremely large stiffness alternating with interlayers of elastic moduli  of order 1.   The effective volume fraction of the stiff layers is characterized by:
 
 \begin{equation}
\vartheta := \lim_{\e\to 0} \frac{r_\e}{\e}.
\label{defvartheta}
\end{equation}

\ni We 
  assume in this subsection  that 
  
     \begin{align}
  	& \vartheta=0,
  \label{repetit}
	 \\ & \mu_{0\e}=\mu>0; \qquad \quad \l_{0\e}=\l\ge0.\label{mulambda}
  	\end{align}

%

\ni We introduce  the operator 
$   \bfsigma \!:H^1(\Omega;\RR^3)\to\! L^2(\Omega;\SS^3)$ and  $n_\e\in L^\infty(\O)$
defined  by 

   \begin{align}
  	&  \bfsigma(\bfvarphi) := \l\tr(\bfe(\bfvarphi))\bfI +2\mu \bfe(\bfvarphi)\quad \forall \bfvarphi \in H^1(\Omega;\RR^3), \label{defsigma}
 	 \\ &n_\e(  x) :=  \sum_{i\in Z_\e}\sharp   \lp  \omega_\e  \cap \lp \e i-\frac{ \e}{2}, \e i+\frac{ \e}{2}\rc  \rp\  \mathds{1}_{ \lp \e i-\frac{ \e}{2}, \e i+\frac{ \e}{2}\rp}(x_3),\label{defne}
	 \\&  Z_\e:= \la i\in \ZZ, \    \lp \e i-\frac{ \e}{2}, \e i+\frac{ \e}{2}\rc\subset (0,L)\ra \label{defZe}.
  	\end{align}

%
%
%

%
%

\ni   Assumption  (\ref{condomegae})   implies that $|n_\e|_{L^\infty(\O)}\le 1$, therefore, up to a subsequence,

\begin{equation}
\label{neton}
\begin{aligned}
&
 n_\e \buildrel \star \over \rightharpoonup n  \ \hbox{ weakly*  in }   \ L^\infty(\O)\quad \hbox{for some $n\in L^\infty(\O)$}.
\end{aligned}
\end{equation}
 
\ni  The scalar    $\frac{1}{\e} n_\e(x )$ is an approximation   at    $x$   of the   local number   of stiff layers  per unit length   in the $\bfe_3$ direction.
\ni Under these assumptions, we prove that   the solution       to  (\ref{Pe}) weakly* converges   in $L^\infty(0,T; $ $H^1_0(\O;\RR^3))$ 
 to the unique
solution 
to
  $ ( \P^{hom}_{( n,  k, \kappa)})$ defined,  in terms of  
  $k$,  $\kappa$,  $n$  given  by  (\ref{kkappa}),  (\ref{neton}),  as follows:
if $0<k<+\infty$, we get (see (\ref{exprim}))

 \begin{equation}
\label{k0f}
\begin{aligned}&
(\P^{hom}_{ (n,   k, 0) } ): 
\\&  {  (0 <k<+\infty )}
\end{aligned}
\la \begin{aligned}
&
 ( \rho +n{\overline \rho_1})  {\partial^2\bfu\over \partial t^2}   
 -\bfdiv   \bfsigma ( \bfu)
\\& \hskip 1,5cm - nk   \bfdiv\bfsigma_{x'}(\bfu')
 =( \rho +n\ru ) \bff   \hskip 0,5cm    
\hbox{ in } \OT,
\\& \bfu   \in   C([0,T]; H^1_0(\O;\RR^3))  \cap    C^1\!([0,T]; L^2(\O; \RR^3 )), 
\\& \bfu(0)= \bfa_0, \ {\partial \bfu \over \partial t}(0)=
\bfb_0.
\end{aligned}\rpt
\end{equation}

\ni   If $k=+\infty$ and $\kappa=0$,  
 the limit problem is   deduced from 
 (\ref{k0f}), formally,  by substituting $(0,0)$ for $(u_1(x), u_2(x))$ when $n(x)>0$:

 \begin{equation}
\label{infty0f}
(\P^{hom}_{( n, +\infty, 0) }) :  \la
\begin{aligned}
&
 ( \rho +n{\overline \rho_1}) \! {\partial^2 u_3\over \partial t^2}   
 -(\bfdiv\bfsigma(\bfu))_3
 = ( \rho +n{\overline \rho_1})  f_3   \hskip 0,5cm 
\hbox{ in } \OT,
\\&
{\color{black}   \rho   \! {\partial^2 u_\a\over \partial t^2}   
 -(\bfdiv\bfsigma(\bfu))_\a
 =   \rho    f_\a   \hskip 0,5cm 
\hbox{ in } \{n=0\} \times(0,T),}
\\&{\color{black}  
nu_1\!=\!nu_2\!\!=0,\  
 \bfu   \! \in\! C([0,T]; \!H^1_0(\O;\RR^3 )\!) \!\cap \!  C^1([0,T]; L^2(\O;\RR^3  )\!), }
  \\&{\color{black} \bfu(0) = (a_{01} \mathds{1}_{\{n=0\}}, a_{02}\mathds{1}_{\{n=0\}}, a_{03}), \  }
\\&{\color{black}{\partial \bfu \over \partial t}(0)= (b_{01} \mathds{1}_{\{n=0\}},b_{02}\mathds{1}_{\{n=0\}}, b_{03}).}
\end{aligned}\rpt
\end{equation}

\ni The case $0<\kappa<+\infty$ is
characterized by  the emergence  of   fourth order derivatives with respect to $x_1, x_2$ in the limit equations, revealing bending effects: 

  \begin{equation}
\label{inftykappaf}
\begin{aligned}&
(\P^{hom}_{( n, +\infty, \kappa)}): \! \!\!
\\& (0\!<\!\kappa\!<\!\!+\infty) 
\end{aligned}
\la\begin{aligned}&
 ( \rho +n{\overline \rho_1})  {\partial^2
u_3
\over \partial t^2}  -(\bfdiv\bfsigma(\bfu))_3
 + n{ \kappa  \over 3}{l+1\over l+2}\sum_{\a,\b=1}^2 {\partial^4  u_3 \over \partial x_\a^2\partial x_\b^2}  
   \\& \hskip 4cm =  ( \rho +n{\overline \rho_1})  f_3   \ \hbox{ in } \OT,
   \\&
{\color{black}   \rho   \! {\partial^2 u_\a\over \partial t^2}   
 -(\bfdiv\bfsigma(\bfu))_\a
 =   \rho    f_\a   \hskip 0,5cm 
\hbox{ in } \{n=0\} \times(0,T),}
\\&{\color{black} nu_1\!=\!nu_2\!\!=0,\  
 \bfu   \! \in\! C([0,T]; \!H^1_0(\O;\RR^3 )\!) \!\cap \!  C^1([0,T]; L^2(\O;\RR^3  )\!), }
  \\&{\color{black}  u_3   \in C([0,T]; L^2_n(0,L; H^2_0(\O') )) , \quad}
  \\&{\color{black}  \bfu(0) = (a_{01} \mathds{1}_{\{n=0\}}, a_{02}\mathds{1}_{\{n=0\}}, a_{03}), \  }
\\&{\color{black}{\partial \bfu \over \partial t}(0)= (b_{01} \mathds{1}_{\{n=0\}},b_{02}\mathds{1}_{\{n=0\}}, b_{03}).}
\end{aligned}\rpt
\end{equation}

\ni If $\kappa=+\infty$, 
we     get:

 \begin{equation}\label{inftyinftyf}(\P^{hom}_{( n, +\infty,+\infty)}): 
 \la \begin{aligned}& \rho  {\partial^2\bfu\over \partial t^2}    -\bfdiv   \bfsigma ( \bfu)  
 =  \rho   \bff   \hskip 0,5cm  \hbox{ in } \{n=0\}\times (0,T),
 \\& n\bfu=0,  
\quad \bfu   \in   C([0,T]; H^1_0(\O;\RR^3))  \cap  \hskip-0 cm C^1\!([0,T]; L^2(\O; \RR^3 )), 
\\& \bfu(0)= \bfa_0\mathds{1}_{\{n=0\}}, \ {\partial \bfu \over \partial t}(0)=\bfb_0\mathds{1}_{\{n=0\}}  .\end{aligned}\rpt\end{equation}


 
  \begin{theorem} 
\label{thstiff} Assume  (\ref{repetit}),    (\ref{mulambda}), (\ref{neton}),    then    the sequence $( \bfu_\e)$ of the
solutions to (\ref{Pe})
 weakly*    converges  in $L^\infty(0,T; $ $H^1_0(\O;\RR^3))$ to the unique solution of the problem $(\P^{hom}_{(n,   k,\kappa)})$ given
 by  (\ref{k0f})-(\ref{inftyinftyf}).

  \end{theorem}
 

%
%
%
   %

{\color{black}  
  \begin{remark}
(i) When 
stiff fibers \cite{BeVisco, BeGr}  (resp.  grain-like inclusions \cite{BeArma}) 
 embedded in a matrix of stiffness 
 of order 1 are considered,   the fibers (resp. the  inclusions)  disappear from  the limit problem 
 if $r_\e \ll \exp-\frac{1}{\e^2}$ (resp. $r_\e\ll \e^3$), where $r_\e$ denotes  the diameter of the sections 
 of the fibers (resp. of the inclusions). 
 This never  occurs in the stratified case, 
 whatever the choice of
  $(r_\e)$.
This is related to the fact that the harmonic capacity of a surface in $\Omega$ is always positive, whereas that of a line or a point are equal to zero. 
%
%
%

\ni (ii) Under (\ref{mulambda}), the case $\vartheta>0$,  $k<+\infty$ 
has been studied in   \cite{GuMo}, \cite{McCo}. In the case $\vartheta>0$,
$k=+\infty$, we  came up against   technical complications (see Remark \ref{remI22}).


\end{remark}

%


 \subsection{Stochastic case}\label{subsecstoch}

 \ni 
Fixing  $d>0$ and set 

   \begin{align}
  	&\gotO:=\la \omega
 \in 2^{\RR }, \ \forall (\omega_1,\omega_2)\in \omega^2, \  \omega_1\not= \omega_2 \ \Rightarrow \ 
\vert \omega_1-\omega_2\vert \geq d \ra,\nonumber
	 \\ &\omega_\e (\omega):= \e \omega \cap (\e,L-\e)\qquad \forall \omega\in \gotO.
\label{Ceta}
  	\end{align}

%
%
%
}


 \ni  Let  $\B_\gotO$  be  the Borel $\sigma$-algebra   generated by the Hausdorff distance on $\gotO$ (see Remark \ref{remHaus}),
and   $ {P} $ be  a probability 
on   $(\gotO, \B_\gotO)$  satisfying 

\begin{equation}
{P}(A+z)= {P}(A) \quad  \forall \  z\in \ZZ , \  \forall \,A\in \B_\gotO.
\label{P}
\end{equation}

\ni   We consider the random distribution of stiff homothetical layers $B_\e(\omega_\e (\omega))$ and the problem $(\P_{\e}(\omega))$ obtained by substituting $\omega_\e (\omega)$ for $\omega_\e$ in 
 \eqref{defBe}, \eqref{Pe}.
In what follows,
  $\F$  represents   the $\sigma$-algebra of the  $Y$-periodic  elements of $\B_\gotO$,
   ${E}_{{P}}^{\F} X$ 
    the   conditional expectation  of a random variable  $X$ given $\F$ with respect to ${P}$,
    $n_\e(\omega)$  the element of $L^\infty(\O)$  defined by    substituting 
 $\omega_\e (\omega)$ for $\omega_\e $ in  (\ref{defne}), 
 and  $n_0:\gotO\to \NN$ the random variable given  by  {\color{black}
 \begin{equation}
 n_0(\omega):= \sharp \lp \omega\cap \left[ -\frac{1}{2}, \frac{1}{2}\right[\rp \qquad \forall \omega\in \gotO.
 \label{defn0}
 \end{equation}
 }
 

\ni  The following theorem  is proved in \cite{BeHDR}:


\begin{theorem}\label{thcvne}
Under the assumptions stated  above, there exists a    sequence of reals  $(\e_k)$ converging to $0$ and  a ${P}$-negligible subset $\gotN$ of $\gotO$, 
   such that for all $\omega\in \gotO\setminus \gotN$, 
 
  \begin{equation}
\begin{aligned}
 &
 n_{\e_k}(\omega)\buildrel \star\over \rightharpoonup 
  {E}_{{P}}^{\F} n_0(\omega) \quad \hbox{weakly* in } \ L^\infty(\O).
\end{aligned}
\label{neta}
\end{equation}
 
 \end{theorem}

\ni The following result    straightforwardly  follows from  theorems \ref{thstiff}, 
  \ref{thcvne}: 


\begin{theorem}\label{thstoch}
Assume  (\ref{repetit}),    (\ref{mulambda}),  and let $(\e_k)$ and $\gotN$ be  the sequence and the $P$-negligible set given   by  Theorem \ref{thcvne}. Then, for all $\omega\in \gotO\setminus \gotN$, the solution  
to 
 $(\P_{\e_k}(\omega))$,   
 weakly*  converges in  $L^\infty(0,T; H^1_0(\Omega;\RR^3))$
 to the unique solution to the problem  
$ ( \P^{hom}_{ ({E}_{{P}}^{\F} n_0(\omega),    k, \kappa)})$ defined by (\ref{k0f}-\ref{inftyinftyf}).
 \end{theorem}

  \begin{remark}\label{remHaus}
The  restriction   of  the Hausdorff distance   $d_\H$  to  $\gotO$ is an extended metric on $\gotO$, and    the mapping   $d_\gotO: \gotO^2\to [0,1]$  defined  by  
$d_\gotO( \omega, \omega'):= \min \{1, d_\H(\omega,\omega')\}$ is a finite metric on $\gotO$ which 
 turns   $\gotO$ into a complete metric space.
 \end{remark}
  


 \subsection{Intermediate case }\label{subsecintermediate}

Under the assumptions 

   \begin{align}
  	&\e^2<\!\!<\mu_{0\e}<\!\!<1, \quad  {\color{black} 0\le \lambda_{0\e}\le C\mu_{0\e},}  \label{mulambdaintermediaire}
	 \\ &n_\e \to n \quad \hbox{ strongly in } \quad L^2(\Omega),\label{netonstrong}
  	\end{align}

%
%
%
%

\ni 
 we show   that  the solution   to (\ref{Pe}) weakly*  converges  in $L^\infty(0,T; L^2(\O;\RR^3))$
to the unique solution to  $ ( \P^{hom}_{(n,   k,\kappa)})$   defined
 by

 \begin{equation}
\label{k01}
\begin{aligned}&
(\P^{hom}_{ ( n, k, 0) } ): 
\\&  {  (0\! <\!k\!<\!+\infty )}
\end{aligned}
\!\!\la \begin{aligned}
&
 ( \rho(1-\vartheta n)+n{\overline \rho_1})  {\partial^2\bfu\over \partial t^2}   
-nk   \bfdiv\bfsigma_{x'}(\bfu')
  \\& \hskip 2,5cm =(\rho(1-\vartheta n)+n\ov \rho_1 ) \bff \    
\hbox{ in } \OT,
  \\ &   u_1, u_2 \in C([0,T]; L^2_n(0,L; H^1_0(\Omega'))), 
  \\ & {\color{black}   \bfu\in   C^1([0,T]; L^2(\O; \RR^3 )), 
\ 
  \bfu(0)= \bfa_0, \ {\partial \bfu \over \partial t}(0)=
\bfb_0,}
\end{aligned}\rpt
\end{equation}


 \begin{equation}
\label{infty01}
(\P^{hom}_{( n, +\infty, 0) }) : \hskip 0 cm \la
\begin{aligned}
&
(\rho(1\!-\!\vartheta n)\!+\!n\ov \rho_1 ) \! {\partial^2 u_3\over \partial t^2}   
 =(\rho(1-\vartheta n)\!+n\ov \rho_1 )  f_3 
\quad \hbox{ in } \OT,  
  \\ &{\color{black} \rho  {\partial^2
u_\a
\over \partial t^2}    =  \rho  f_\a  \ \hbox{ in } \{n=0\}\times (0,T), \ (\a\in \{1,2\}),}
\\& {\color{black} n}u_1={\color{black} n}u_2=0,\quad 
  \bfu  \in C^1([0,T]; L^2(\O;\RR^3 )),
\\&{\color{black} \bfu(0) = (a_{01} \mathds{1}_{\{n=0\}}, a_{02}\mathds{1}_{\{n=0\}}, a_{03}), \  }
\\&{\color{black}{\partial \bfu \over \partial t}(0)= (b_{01} \mathds{1}_{\{n=0\}},b_{02}\mathds{1}_{\{n=0\}}, b_{03}), }
\end{aligned}\rpt
\end{equation}


 \begin{equation}
\label{inftykappa1}
\begin{aligned}&
(\P^{hom}_{(n, +\infty,\kappa)}):  
\\& (0<\kappa<+\infty) 
\end{aligned}
\la\begin{aligned}&
(\rho(1-\vartheta n)+n\ov \rho_1 )  {\partial^2
u_3
\over \partial t^2}  
 + n{\kappa  \over 3} {l+1\over l+2}\sum_{\a,\b=1}^2 {\partial^4  u_3 \over \partial x_\a^2\partial x_\b^2}  
  \\& \hskip 2,5cm =  (\rho(1-\vartheta n)+n\ov \rho_1 )  f_3   \ \hbox{ in } \OT,
  \\ &{\color{black}\rho  {\partial^2
u_\a
\over \partial t^2}    =  \rho  f_\a  \ \hbox{ in } \{n=0\}\times (0,T), \ (\a\in \{1,2\}),}
  \\& {\color{black}n}u_1 ={\color{black}n}u_2=0,\quad 
\\&{\color{black}   u_3  \in C ([0,T] ; L^2_n (0,L; H^2_0(\Omega'))), \  \bfu\in C^1 ([0,T] ; L^2(\O;\RR^3)), }
\\&{\color{black} \bfu(0) = (a_{01} \mathds{1}_{\{n=0\}}, a_{02}\mathds{1}_{\{n=0\}}, a_{03}), \  }
\\&{\color{black}{\partial \bfu \over \partial t}(0)= (b_{01} \mathds{1}_{\{n=0\}},b_{02}\mathds{1}_{\{n=0\}}, b_{03}),}
\end{aligned}\rpt
\end{equation}


\begin{equation}\label{inftyinfty1}(\P^{hom}_{( n, +\infty,+\infty)}): 
 \la \begin{aligned}& \rho  {\partial^2\bfu\over \partial t^2}    
 =  \rho   \bff   \hskip 0,5cm  \hbox{ in } \{n=0\}\times (0,T),
 \\& n\bfu=0,  
\quad \bfu   \in     C^1\!([0,T]; L^2(\O; \RR^3 )), 
\\& \bfu(0)= \bfa_0\mathds{1}_{\{n=0\}}, \ {\partial \bfu \over \partial t}(0)=\bfb_0\mathds{1}_{\{n=0\}}  .\end{aligned}\rpt\end{equation}

%

 
  \begin{theorem} 
\label{thinter}  
Under
     (\ref{mulambdaintermediaire}), (\ref{netonstrong}), 
    the solution   to  (\ref{Pe}) weakly*  converges  in $L^\infty(0,T; L^2(\O;\RR^3))$ to the unique solution to    $(\P^{hom}_{(n,k,\kappa)})$ given
by  (\ref{k01})-(\ref{inftyinfty1}). 
    \end{theorem}

  
{\color{black}
  \begin{remark}
 (i)  Problems  (\ref{k01})-(\ref{inftyinfty1}) are formally deduced from  (\ref{k0f})-(\ref{inftyinftyf}) by removing the term ``$\bfdiv\bfsigma(\bfu)$" (see (\ref{repetit})).
 This indicates  that no    strain energy  is stored in the softer phase. 
%
%
%
 
 \ni (ii) Assumption (\ref{netonstrong}), stronger than  (\ref{neton}), precludes the application of   Theorem \ref{thcvne}  and the extension of 
   Theorem \ref{thinter} to the setting of stochastic homogenization.
   
    \end{remark}
    }



{\color{black}
 \subsection{ Case of  soft interlayers  with Lam\'e coefficients of order $\e^2$
   }\label{subseccritical}


We}  
  assume  that 

 \begin{equation}
\begin{aligned}
 &   \mu_{0\e}=\e^2\mu_0, \quad \lambda_{0\e}=\e^2\lambda_0, \quad \mu_0>0, \quad \lambda_0\ge 0,
\end{aligned}
\label{mu0lambda0}
\end{equation}

  \ni  
  and 
   that the stiff layers are periodically distributed (see \eqref{defZe}):
  
%
%
%
%
%
%
  
   \begin{equation}
 \begin{aligned}
 &{B_\e}:=   
  \bigcup_{i\in Z_\e} B^{i}_\e; \quad B^{i}_\e :=
  \Omega'\times    
   \lp \e i+ r_\e I\rp.
 \end{aligned}
  	\label{defBeper}
\end{equation}

%
%

 \ni  Under these hypotheses, setting  

   \begin{align}
  	& Y := \lp -{1\over2},  {1\over 2} \rp^3 ; 
\   {B}:= \lp -{1\over2},  {1\over 2} \rp^2\times \lp -{\vartheta\over2},  {\vartheta\over 2} \rp; 
\  \Sigma:=  \lp -{1\over2},  {1\over 2} \rp^2\times\{0\},
\label{defYBSigma}
	 \\ &A:=   B \quad \hbox{if } \ \  \vartheta>0,
 \qquad A:=   \Sigma  \quad \hbox{if } \ \ 
  \vartheta=0, 
\label{defA}
  	\end{align}

%
%
%
%
%

\ni 
 we show that the  solution $\bfu_\e$  to
(\ref{Pe}) two-scale converges to  $\bfu_0\in  C([0,T]; L^2(\O, $ $ H^1_\sharp(Y;\RR^3) ) )$ (see  Section \ref{sectwoscale} for the definition of this   convergence), 
and  the sequence $(\bfu_\e m_\e)$, where $m_\e$
 is  the measure
  defined by 

 \begin{equation}
\label{defme}
\begin{aligned}
m_\e:= \frac{\e}{r_\e}
\mathds{1}_{B_\e}(x) \L^3_{\lfloor \Omega},
\end{aligned}
\end{equation}

 \ni weakly* converges in  $L^\infty(0,T; \M (\ov \Omega;\RR^3))$
to   $\bfv\in C^1([0,T]; L^2(\O; \RR^3 ))$,
where $(\bfu_0, \bfv )$
  is the unique solution to the 
  coupled 
  system  of equations (comparable in certain respects with \cite[(2.17)]{BeSiam})
 
   \begin{equation}
\la \begin{aligned}
&{\color{black} ( \P^{hom}_{soft}(\kappa, \vartheta) ) ,  }
\\ &( \P^{hom}_{stif\!f}(  k,\kappa)),
\end{aligned}
\rpt
\label{Phom1}
\end{equation}

\ni  defined below in terms of  $k$, $\kappa$, and  $\vartheta$   given, respectively,  by (\ref{kkappa}) and  (\ref{defvartheta}).
The fields $\bfu_0$ and $\bfv$ are linked  by  the following  relation on $\Omega\times (0,T)\times A$:
 
\begin{equation} 
\begin{aligned}
&\bfv(x,t)=   \bfu_0(x,t,y)   \ \  \hbox{   in }  \ \OT \times  A \quad & &\hbox{if } \quad
\vartheta>0
   \quad  \hbox{or } \quad  \kappa>0,
\\&
\hskip-0,13cm \lpt \begin{aligned}
&\bfv' =  {\bfu'_0}    \ \  \hbox{  on }  \ \OT \times  A
 \quad 
 \\&  
v_3  = \int_{A} u_{03}(.,y) d\H^2(y)\quad 
\end{aligned}\ra & &\hbox{if } \quad 
\vartheta=0
 \quad  \hbox{ and} \quad  \kappa=0.
\end{aligned}   
\label{coupling}
\end{equation}


\ni We introduce  the operators 
$  \bfe_y, \bfsigma_{0y}\!:\! H^1(Y;\RR^3)\! \to\! L^2(Y;\SS^3)$, $\bfg \!:\! \H \!\to \!\RR^3$
defined  by {\color{black} 
 \begin{equation}
 \begin{aligned}
  &  (\bfe_y(\bfw))_{ij}={1\over 2} \lp{\partial w_i\over\partial y_j} +{\partial w_j \over \partial y_i}\rp, \ 
 \bfsigma_{0y} (\bfw) := \l_0 \tr(\bfe_y(\bfw ))\bfI +2\mu_0 \bfe_y(\bfw)  ,
 \\ &  \bfg(\bfw) := 
 \la\begin{aligned}
 &-\int_{\partial (Y\setminus B)\cap \ov B} \bfsigma_{0y}(\bfw)  \bfnu_{Y\setminus B} d\H^2(y),
 & & \hbox{if } \ A=B,
 \\& \int_{\Sigma} (\bfsigma_{0y}(\bfw^+) -\bfsigma_{0y}(\bfw^-))\cdot \bfe_3 d\H^2(y)
 & & \hbox{if } \ A=\Sigma,
 \end{aligned} \rpt
 \end{aligned}
\label{eysigmayg}
\end{equation}

\ni where   $\bfnu_{Y\setminus B}$ stands for  the outward  normal to $\partial ({Y\setminus B})$
and

\begin{equation}
\H:= \la \bfw\in  H^1(Y\setminus A;\RR^3),\ \  \bfdiv
(\bfsigma_{0y}(\bfw))\in (H^1(Y\setminus A;\RR^3))^* \ra,
  	\label{Hdiv}
\end{equation}

\ni denoting  by  $E^*$  the topological  dual of a   Banach space $E$,
%
%
and  by} $\bfw^+$ (resp. $\bfw^-$)  the restriction of $\bfw$ to $\lp\frac{-1}{2},\frac{1}{2}\rp^2\times \lp0,\frac{1}{2}\rp$
(resp. $\lp\frac{-1}{2},\frac{1}{2}\rp^2\times \lp \frac{-1}{2},0\rp$).
Problem   $( \P^{hom}_{soft} (\kappa,\vartheta))$ in  (\ref{Phom1})  is the equation of $\bfu_0$ 
 in $\Omega\times (0,T)\times  (Y\setminus A)$ 
coupled with $\bfv$
 (\ref{coupling}) and given by 
(denoting by $\bfnu $   the outward normal to $\partial Y$):

   \begin{equation}
( \P^{hom}_{soft} (\kappa,\vartheta)\!)\!: \! \!  \left\{\begin{aligned}
& \rho {\partial^2 \bfu_0 \over \partial t^2}-\bfdiv_y (\bfsigma_{0y} (\bfu_0 ))= \rho \bff  \hskip0,2 cm \hbox{ in } \ {\Omega\times(0,T)\times  Y\setminus A},\\
  & (\bfu_0,\bfv) \quad \hbox{satisfies} \quad (\ref{coupling}),
  \\  & \bfsigma_{0y}(\bfu_0 )\bfnu(y)= -\bfsigma_{0y}(\bfu_0 )\bfnu( -y)   \hskip0 cm \hbox{ on } \ \OT\times \partial Y, \\
    & \bfu_0\! \in \!C([0,T]; L^2(\O, H^1_\sharp(Y;\RR^3)\!)\!)\!\cap \!C^1([0,T]; L^2(\OY;\RR^3)\!),   \\
    &  \bfu_0(0)\mathds{1}_{Y\setminus A}=\bfa_0 \mathds{1}_{Y\setminus A}, \quad   {\partial \bfu_0\over
\partial t}(0)\mathds{1}_{Y\setminus A} =\bfb_0 \mathds{1}_{{Y\setminus A}} .  
\end{aligned}
\right.
\label{Phommatrix}\end{equation}

 \ni  Equation \eqref{Phommatrix}  governs  the effective  behavior of  the displacement in the soft phase.
Problem  $( \P^{hom}_{stif\!f}(  k,\kappa))$  in  (\ref{Phom1})  is   an equation of  $\bfv$ in $\Omega\times(0,T)$
coupled  with $( \P^{hom}_{soft} (\kappa,\vartheta))$  through  the source term 
$\bfg(\bfu_0)$ defined 
  by (\ref{eysigmayg}). This equation  
 rules the effective behavior of the  displacement in the stiff layers. 
Its form is determined by   the order of magnitude of the coefficients $k,\kappa$.
If $0<k<+\infty$, we get  (see (\ref{exprim}))
   
   \begin{equation}
   \begin{aligned}
   &
   \begin{aligned}
    & (\P^{hom}_{stif\!f}(  k, 0) ):  \\ &(0<k<+\infty) \end{aligned}
 \!  \left\{\begin{aligned} 
 &  
 {\overline \rho_1} {\partial^2
\bfv
\over \partial t^2}   -  k  \bfdiv \bfsigma_{x'}(\bfv')
   =\ru  \bff 
+    \bfg(\bfu_0)  
\    \hbox{ in } \OT,
  \\ &   v_1, v_2 \in C([0,T]; L^2(0,L; H^1_0(\Omega')))\cap C^1([0,T]; L^2(\O )) , 
  \\ &   \bfv\in   C^1([0,T]; L^2(\O; \RR^3 )), 
\quad 
  \bfv(0)= \bfa_0, \ {\partial \bfv \over \partial t}(0)=
\bfb_0.
\end{aligned}
\right.
\\&
\end{aligned}
 \label{k0}
\end{equation}

\ni If $(k,\kappa)=(+\infty,0)$,  we obtain
 
\begin{equation} 
(\P^{hom}_{stif\!f}(  +\infty,0)):
\la
\begin{aligned}
  & {\overline \rho_1} {\partial^2
v_3
\over \partial t^2}   
  = \ru    f_3+   (\bfg (\bfu_0))_3 \quad   \hbox{ in }\  \OT,
  \\&  v_1=v_2 =0, 
  \\ &      v_3 \in   C^1([0,T], L^2(\O )),  \  
   v_3(0)=  a_{03}, \ {\partial v_3 \over \partial t}(0)= b_{03}. 
\end{aligned}
 \rpt
 \label{infty0}
\end{equation} 

\ni If $0<\kappa<+\infty$, the 
emergence of  fourth derivatives of $v_3$ 
 reveal  bending effects:

\begin{equation}
\begin{aligned}
&(\P^{hom}_{stif\!f}( +\infty,\kappa)) : \\&(0<\kappa<+\infty)\end{aligned}
\la\begin{aligned}
&
 {\overline \rho_1} {\partial^2
v_3
\over \partial t^2} \!  +\!{\kappa  \over3} {l\!+\!1\over l\!+\!2}\!\!\sum_{\a,\b=1}^2\!\! {\partial^4  v_3 \over \partial x_\a^2\partial x_\b^2}  \!
\\&\hskip 2,2cm=  \ru\!  f_3+ (\bfg(\bfu_0))_3    \quad \hbox{ in } \OT,
\\& v_1=  v_2 =0 , 
 \\& v_3 \in C ([0,T] ;L^2(0,L;H^2_0(\Omega' )))\cap  C^1 ([0,T] ; L^2(\O  )),
\\&  v_3(0)= a_{03}, \ {\partial v_3 \over \partial t}(0)= b_{03}. 
\end{aligned}\rpt
 \label{inftykappa}
\end{equation} 

\ni If  $\kappa=+\infty$,  the displacement in the   stiff layers asymptotically  vanishes:
   \begin{equation}
 (\P^{hom}_{stif\!f}(+\infty,+\infty)): 
 \quad\bfv=0 
 .\hskip5,5 cm
	\label{inftyinfty}
\end{equation}


    \begin{theorem} 
\label{th} Under  
   (\ref{mu0lambda0}), (\ref{defBeper}), 
the solution  $\bfu_\e$  to  (\ref{Pe}) 
  two-scale converges to  $\bfu_0$ with respect to $x$ and  weakly*   converges      in  
$L^\infty(0,T; L^2(\O, \RR^3))$  to   $\bfu= \int_Y \bfu_0(.,y) dy$,
 and  the sequence
 $(\bfu_\e m_\e)$, where $m_\e$ is   defined by (\ref{defme}),  weakly*  converges        in  
$L^\infty(0,T; \M (\ov\O, \RR^3))$  to   $\bfv\L^3_{\lfloor\O}$, where   $( \bfu_0,\bfv
  )$ is the unique solution   to (\ref{Phom1}).
 Moreover, $\bfu_\e(\tau)$ two-scale converges to  $\bfu_0(\tau)$ with respect to $x$, for each
$\tau \in  [0,T] $.   

  \end{theorem}


  \begin{remark}\label{remcorrector} {\it
  One can show (see \cite[p.2548]{BeSiam} for more details), 
that if $\bfa_0=0 $ and if the fields  $\bfb_0$,  $ \bff$   are sufficiently regular,
%
%
%
  the following corrector result holds
 \begin{equation}
\lim_{\e\to 0} \ln \bfu_\e- \bfu_0\lp x,t, \xe\rp \rn_{L^2(\OT;\RR^3)}=0.
	\label{corrector}
\end{equation}

  }
  \end{remark}
 

\begin{remark}
\ni The effective problem (\ref{Phom1}) is non-local   in space and time.
Non-local effects    \cite{Al}, \cite{ArDoHo}-\cite{BraBri},   \cite{CaSe,Ch,ChSmZh,Kh,Pa},  \cite{Sa1}-\cite{Sili},  
and memory effects \cite{AbKoPaSm,AmHaZi,Ma,Ta} are typical of   composite media   with high contrast.   

\end{remark}


 
 \begin{remark}[Multiphase stratified elastic  media ]\label{remmultiphase} {\color{black} In the same way as in  \cite[Section 4]{BeSiam}, we}  can    extend Theorem \ref{th}   to the case of a multiphase   medium whereby $m$ $\e$-periodic disconnected  families $B_\e^{ [1] }$, ...,  $B_\e^{ [m] }$ of  parallel layers 
are 
embedded in a soft matrix.
The limit problem then takes the form

\begin{equation} 
\la \begin{aligned}
  &( \P^{hom}_{soft} ) ,  
   \\   
     &   ( \P^{hom\  [j] }_{stif\!f}) ,   \  j\in \{1,...,m\}, 
\end{aligned}\rpt
\nonumber
\end{equation}

\ni   and can be written under the variational  form  (\ref{Pvar2}) for some  suitable choice of data $H,V,a, h, \xi_0, \xi_1$.
Each family  $B_\e^{ [m] }$  is associated to some  subset $A^{[j]}$ of $Y$  like in (\ref{defA}).
The   system $( \P^{hom}_{soft} ) $  governs  the effective displacement  in the soft phase, and  only 
differs from (\ref{Phommatrix})   by  the relation  (\ref{coupling})  which is  replaced by    a series of relations on  each set  $\Omega\times (0,T)\times A^{[j]}$ between $\bfu_0$ and some auxiliary variable $\bfv^{[j]}$ characterizing the effective displacement in $B_\e^{[j]}$.
Each   problem   $( \P^{hom\ [j]}_{stif\!f}) $  
consists  of  an equation  of   $ \bfv^{[j]} $   of the same form  as
$( \P^{hom }_{stif\!f}) $  in (\ref{Phom1}), coupled with $\bfu_0$ through the   operator  $\bfg^{[j]}$ deduced from 
(\ref{eysigmayg}) by replacing $A$ by $A^{[j]}$. 
Multiphase composites comprising, besides    stiff layers, periodic distributions of  fibers  and  grain-like inclusions,  
can  be considered. 
Multiphase homogenized models have   been studied   in  \cite{BeSiam,Pa,Pa2,Sa1,Sa,Sa07}.

\end{remark}


\begin{remark}[Equilibrium equations]\label{remequilibrium} 
One can check that  if the solution   to
 
 \begin{equation} 
  - \bfdiv(\bfsigma_\e( \bfu_\e)) =
   \bff \quad \hbox{ in } \quad \O ,
\ \quad \bfu_\e\in   H_0^1(\O, \RR^3) ,  \quad  \bff \in  L^2(\O, \RR^3). \label {Pelliptique} 
\end{equation}

\ni  
 two-scale converges  to  $\bfu_0\in L^2(\OY;\RR^3)$, then  

 \begin{equation} 
 \bfu_0 \in V \quad \hbox{ and }\quad  a(\bfu_0, \bfw_0)= (\bff, \bfw_0)_H, \quad \forall \bfw_0\in V, \label {Phomell} 
\end{equation}

\ni  where     the Hilbert spaces $V$ and $H$ and    the non-negative  symmetric bilinear form  $a(.,.)$ are  those mentioned in Remark \ref{remmultiphase}. 
The    form $a(.,.)$    may 
fail  to be coercive on $V$. 
 One can prove (the proof is similar to that  sketched in Remark \ref{remequilibriumstiffinter} in  the  context of Theorem \ref{thinter})
 that 
this  coercivity and the convergence of the solution to \eqref{Pelliptique} 
 are  guaranteed 
provided that
a multiphase stratified composite is considered   whereby the set of stiff layers  comprises  
a  family $B_\e^{[j]}$  of parallel  layers of    thickness $r_\e^{[j]}$
such that $\kappa^{[j]}>0$, that is     with elastic moduli of order  larger than $  {\e}{ \lp r_\e^{[j]}\rp^{-3}} $. 
Similar results were obtained in 
  \cite[Corollary 5.1, Proposition 5.2]{BeSiam} for fibers and grain-like inclusions. Note in passing that 
  one should substitute $1$ for $\rho_\e$ in \cite[Formula 5.1]{BeSiam}, otherwise the proof of "$(iii)\Rightarrow (iv)$" in \cite[p. 2552]{BeSiam} is false. 
 

%

\end{remark}

\begin{remark} \label{remper}
 Under Assumption (\ref{mu0lambda0}),
 the  slightest   perturbation  of  periodicity  leads to a complete change of the form of the effective problem.
 For instance, if 
$m\in \NN$ and 
$\omega_\e^j=\e \lp j + \frac{1}{2}\rp $  if $j$ is a multiple of $m$, and  $\omega_\e^j=\e  j $ otherwise  in (\ref{defBeper}), 
then the limit problem 
is  a   system of equations coupling $\bfu_0$ with $m$  auxiliary variables $\bfv^{[1]}, ..., \bfv^{[m]}$ as described   in  Remark \ref{remmultiphase}, which
 can not be  expressed, as in theorems \ref{thstiff}, \ref{thinter}, simply
 in terms of the function $n$ defined by (\ref{neton}).
The extension of Theorem \ref{th} to  the  non-periodic
 case is far beyond the scope of this paper.
 
\end{remark}

 
  \section{Two-scale convergence and other analysis tools}\label{sectwoscale}  
     {\color {black}
{\color{black} In Section \ref{subsectwoscale}, we recall some properties 
of the two-scale convergence  of G. Allaire   \cite{Al} and G. Nguetseng   \cite{Ng}
and  reproduce   some statements of   \cite{BeSiam} in a suitable form for the present context. 
In Section \ref{subsectwoscaleme}, we introduce a 
non-periodic notion  of two-scale convergence with respect to 
a sequence of measures 
and  establish  a compactness result (Lemma \ref{lemtwoscalewrtme}).  
Two  classical  analysis results are  recalled in Section \ref{subsecabstract}. }
 \subsection{Two-scale convergence}\label{subsectwoscale} 
A sequence $(f_\e)$ in $ L^2(0,T;L^2(\O ))$ {\color{black} is said to two-scale converge  }
  to   $f_0  \in   L^2(0,T  ; $ $L^2(\OY))$  with respect to    $ x $   
if, for {\color{black} all }
$\f_0 
\in \D(\OT ,  C^\infty_\sharp (Y)) $,  

\begin{equation}\begin{aligned}
&\lim_{\e\to 0} \int_\OT  f_\e (x,t) \f_0\lp x,t, {x \over \e} \rp  dxdt =  \int_\OTY f_0\f_0 d xdtdy,
\\& \hbox{(notation: $f_\e \ \dto \ f_0$) } .
  \end{aligned} 	\label{defdto}
\end{equation}

\ni  A sequence    $(\f_\e)\subset  L^2(0,T  ; L^2(\O ))$ 
{\color{black} strongly two-scale converges }
 to   $\f_0  \in   L^2(0,T  ; L^2(\O$ $\times Y))$ with respect to $x$ 
 if

\begin{equation}
\begin{aligned}
&\f_\e \ \dto \ \f_0 \quad \hbox{ and } \quad \lim_{\e\to 0} ||\f_\e||_{ L^2(0,T  ; L^2(\O ))} =|| \f_0||_{ L^2(0,T  ; L^2(\OY))},
\\& \hbox{(notation: $\f_\e \ \sdto \ \f_0$)}.
 \end{aligned}
  	\label{sdto}
\end{equation}

\ni The  symbols   $\ \dto\,$ and  $\ \sdto\ $ 
{\color{black} will also}
denote   
  the   two-scale convergence  and the strong two-scale convergence 
of sequences  $(f_\e)$ in $  L^2(\O ) $   independent of $t$,
 defined  by formally 
 {\color{black} considering them}
   as constant in $t$.
{\color{black} 
 Any   bounded sequence in $L^2(0,T;L^2(\O ))$ has  a two-scale convergent subsequence \cite{Ng}.   
{\color{black} An admissible sequence with respect to two-scale convergence is a
}
  sequence  $(\f_\e)\!\subset \!L^2(0,T;L^2(\O ))$ 
that    two-scale converges to some  $\f_0\in  L^2(0,T  ;  L^2(\OY))$ and 
{\color{black} such that,}
 for   every two-scale convergent sequence $(f_\e)$, 
\begin{equation}
\lim_{\e\to 0} \int_\OT  f_\e  \f_\e dxdt =  \int_\OTY f_0\f_0 d xdtdy .
\label {dto}
\end{equation}

\ni 
{\color{black} A sequence $(\f_\e)$ 
 is admissible if and only if it strongly two-scale converges to  some 
$\bfvarphi_0$ (see \cite[p.2528]{BeSiam}).}
{\color{black} For all  $ \psi_0 \in L^2(0,T; L^2(\O , 
C^\infty_\sharp (Y)))\cup L^2_\sharp (Y, C(\ov\OT))$,
the sequence $( \psi_0( x,t ,{x \over \e}))_{\e>0}$
 strongly two-scale converges to $\psi_0$ (see \cite{Al}, Lemma 5.2, Corollary 5.4).}
In particular, if $Q$ is a Borel subset of $Y$, and $Q_\sharp$ is  
its periodization on $\RR^3$ 
 defined by (\ref{Qsharp}), then the sequence 
$\lp\mathds{1}_{Q_\sharp}\lp \frac{x}{\e}\rp\rp$ strongly two-scale converges to $\mathds{1}_Q(y)$. 
Under (\ref{defBeper}), if $\vartheta>0$, then 
$|\mathds{1}_{\Omega\setminus B_\e}-\mathds{1}_{(Y\setminus B)_\sharp}\lp \frac{x}{\e}\rp|_{L^2(\Omega)} \to 0$, therefore  $\mathds{1}_{\Omega\setminus B_\e} \ \sdto \ \mathds{1}_{Y\setminus B}$. If  $\vartheta=0$,  $(\mathds{1}_{\Omega\setminus B_\e}) $ 
strongly converges to $1$ in $L^2(\Omega)$, hence strongly two-scale converges to $1$.
We deduce that (see \eqref{defA})

\begin{equation}
\label{1Besdto}
\begin{aligned}
\mathds{1}_{\Omega\setminus B_\e} \ \sdto \ \mathds{1}_{Y\setminus A}, 
\quad \mathds{1}_{ B_\e} \ \sdto \ \mathds{1}_A.
\end{aligned}
\end{equation}



 {\color{black}
\ni The next Lemma is a straightforward variant of  \cite[Lemma 6.1]{BeSiam}.
  }

 
 \begin{lemma}\label{lemtwoscale} 
 \ni  (i) 
  Let $(h_\e)$ be a bounded  sequence in $L^\infty(\OTY)$  such that $h_\e \sdto h_0$.   Then,  for every sequence 
$(\chi_\e)\subset L^2(0,T;L^2(\O))$, 
the  following implications hold:

    \begin{align}
  	&\chi_\e \ \sdto \   \chi_0  \implique  \chi_\e h_\e \ \sdto \ \chi_0  h_0,\label{lem1}
	 \\ & \chi_\e \ \ \dto \ \, \chi_0   \implique  \chi_\e h_\e \ \ \dto \ \, \chi_0  h_0.\label{lem2}
  	\end{align}


\ni  (ii) If     $(f_\e)$ is  
bounded       in $L^\infty(0,T; L^2(\O))$, {\color{black} then $(f_\e)$   two-scale converges, up to a subsequence,  to    some    
  $f_0\in L^\infty(0,T; L^2(\OY))$.}
If in addition $(f_\e)$ is   bounded   in $W^{1,\infty}(0,T; L^2(\O))$,
then  $f_0\in W^{1,\infty}(0,T; L^2(\OY))$ and 
 $\lp \partial f_\e\over \partial t\rp$ two-scale converges to  
$ {\partial f_0\over \partial t}$.  Besides, if    $f_\e(0)\dto  a_0$, then $a_0=f_0(0)$ and    $f_\e( \tau)\dto f_0( \tau ),\ \forall \tau \in [0,T]$. {\color{black} Furthermore,}
if  $\lp \partial f_\e\over \partial t\rp\sdto  {\partial f_0\over \partial t}$ and  $f_\e(0)\sdto  a_0$, then $f_\e( \tau)\sdto f_0( \tau )$, $\forall \tau \in [0,T]$.


\end{lemma}

{\color{black}
 \subsection{Two-scale convergence with respect to $(m_\e)$}\label{subsectwoscaleme}
 One can  easily check  that   the sequence  $(m_\e)$ defined by \eqref{defme}  is bounded in 
  $\M (\ov\O)$    and satisfies 

   \begin{equation}
   \begin{aligned}
&m_\e \buildrel * \over \rightharpoonup  n \L^3_{\lfloor \O} \quad \hbox{weakly* in } \ \M (\ov\Omega),
\end{aligned}
\label{meton}
  \end{equation}

\ni where $n$ is defined by (\ref{neton}). Notice that 

\begin{equation}
   \begin{aligned}
   n=1 \quad 
   \hbox{ under} \quad \eqref{defBeper}.
\end{aligned}
\label{n=1per}
  \end{equation}
  
  \ni 
In what follows, the symbol  $L_n^2(\Omega;\RR^3)$  stands for  the set of all Borel fields $\bfw:\Omega\to \RR^3$ such that 
$\int_\Omega  |\bfw|^2 n dx<+\infty$. Similarly,  for any Hilbert space $H$,  we denote by   $L^2_n(0,L;H)$ 
 the set of all Borel fields $\bfw:(0,L)\to H$ such that 
$\int_0^L  |\bfw|_H^2 n dx<+\infty$.
We set (see (\ref{disjoint}), (\ref{defBe}))
 
 \begin{equation}
\label{yepsnonper}
\begin{aligned}
y_\e(z):= \sum_{j\in J_\e} ( z-\omega_\e^j) \mathds{1}_{ \omega_\e^j+r_\e(1+\delta) I}(z)
.
\end{aligned}
\end{equation}

 \ni  
We say that a sequence $(f_\e)$ in $ L^2(0,T;L^2 (\O ))$     two-scale converges  to   $f_0  \in   L^2(0,T  ; $ $L^2_n(\O\times I))$    with respect to the sequence  of measures  $(m_\e)$  
if for each
$\psi
\in \D(\OT ; $ $ C^\infty_\sharp (I)) $,  the following  holds 
%
 \begin{equation}
\begin{aligned}
&\lim_{\e\to 0} \int_\OT  f_\e (x,t) \psi \lp x,t, {y_\e(x_3)\over r_\e} \rp dm_\e dt =  \int_{\OT\times I}   f_0\psi nd xdtdy_3.
\\& \hbox{(Notation: $f_\e \ \buildrel m_\e \over \dto \ f_0$)}.
\end{aligned}
\label{defmdto}
\end{equation}


\begin{lemma}\label{lemtwoscalewrtme} 
%
%
 Let    $(f_\e)$  be a sequence   in  $ L^2(0,T;L^2(\O ))$
satisfying 

\begin{equation}
\begin{aligned}
\sup_{\e>0} \int_{\O\times (0,T)}   | f_\e|^2  d m_\e dt<+\infty. 
\end{aligned}
\label{feme<infty}
\end{equation}

\ni Then $(f_\e)$
two-scale converges    with respect to $(m_\e)$,   up to a subsequence, to some $f_0\in L^2_n(0,T;L^2(\O\times I ))$.
  In addition, if 

\begin{equation}
\begin{aligned}
\sup_{\e>0, \tau>0} \int_{\O }   | f_\e|^2(\tau)  d m_\e  <+\infty, 
\end{aligned}
\label{feme<infty2}
\end{equation}

\ni then $f_0\in L^\infty(0,T; L^2_n(\Omega\times I))$.

\end{lemma}

\begin{proof} 
%
By    Cauchy-Schwarz Inequality  and by (\ref{feme<infty}), we have 

 \begin{equation}
\label{apdtm}
\begin{aligned}
&\lb  \int_\OT f_\e(x,t) \psi\lp x,t,{y_\e(x_3)\over r_\e}\rp dm_\e dt\rb
\\&\hskip 2,5cm  \le C \lp  \int_\OT |f_\e|^2 dm_\e dt\rp^{\frac{1}{2}} ||\psi||_{L^\infty(\OTI)}
\\&\hskip 2,5cm  \le C||\psi||_{L^\infty(\OTI)}
\quad \forall \psi \in C(\ov{\OTI}). 
\end{aligned}
\end{equation}

\ni Hence, by the Riesz representation theorem, for each $\e>0$   there exists a finite Radon measure $\theta_\e\in \M (\ov \OTI )$ such that 
 
\begin{equation}
\label{nonumber}
\begin{aligned}
&\int \psi d\theta_\e  = \int_\OT \!\!\!f_\e(x,t) \psi\lp x,t,{y_\e(x_3)\over r_\e}\rp dm_\e dt\quad \forall \psi \in C(\ov{\OTI}).
\end{aligned}
\end{equation}

\ni By (\ref{apdtm}) and (\ref{nonumber}),  the sequence   $(\theta_\e)$ is bounded in $\M (\ov{\OTI})$, thus  weakly* converges, up to a subsequence, 
to some $\theta \in \M (\ov\OTI)$.
By Cauchy-Schwarz Inequality,  we have 
 \begin{equation}
\label{inu}
\begin{aligned}
\lb \int \psi d\theta_\e \rb 
&\!\le \!\lp \int_\OT \!|f_\e|^2 dm_\e dt\rp^{1\over2} \!\!\! \lp \int_\OT  \lb \psi\lp x,t,{y_\e(x_3)\over r_\e}\rp\rb^2 \!\!dm_\e dt\rp^{1\over2}
\!\!\! .\end{aligned}
\end{equation}

\ni 
The proof of the next statement is similar to that of 
    \cite[Lemma 1.3]{Al}:
   
    \begin{equation}
\label{inu2}
\begin{aligned}
\lime  \int_\O  \lb \varphi\lp x, {y_\e(x_3)\over r_\e}\rp\rb^2 dm_\e  = \int_{\O\times I}  |\varphi|^2 n dx  dy_3
\quad   \forall \varphi\in C(\ov {\O\times I }).
\end{aligned}
\end{equation}

\ni  We deduce from (\ref{feme<infty}), 
(\ref{nonumber}),  (\ref{inu}),  (\ref{inu2}), and from  the weak* convergence in  $\M (\ov{\OTI})$ of $(\theta_\e)$ to $\theta$, that 

 \begin{equation}
\nonumber
\begin{aligned}
\lb \int \psi d\theta\rb =  \lime \lb \int \psi d\theta_\e \rb  \le   C ||\psi||_{L^2_n(\OTI)} \quad \forall \psi \in C(\ov{\OTI}). 
\end{aligned}
\end{equation}

\ni  Thus, the linear form $\psi\to \int \psi d\theta$  is continuous on $C(\ov{\OTI})$ with respect to the strong topology of $L^2_n(\OTI)$.
By a density argument, 
this linear form can be extended to a continuous linear form on $L^2_n(\OTI)$ which, 
by the Riesz representation theorem, takes the form $\psi\to \int_{\OTI} \psi f_0 n dxdtdy$  for some  
$f_0\in L^2_n(\OT\times I)$.
We infer  that 
 $\theta = n f_0 \L^5_{\lfloor_{\OTI} }$, and then,  taking (\ref{nonumber})  and  the weak* convergence of $(\theta_\e)$ to $\theta $ into account, deduce  (\ref{defmdto}). 
%
%
%
%
%
%
%
%
\ni  Under  (\ref{feme<infty2}),  by Fubini's Theorem and Cauchy-Schwarz inequality, we have

 \begin{equation}
\begin{aligned}
    \int_\OT  \hskip-0,8cm f_\e (x,t) \psi  \lp x,t, {y_\e(x_3)\over r_\e} \rp &dm_\e dt  
\le  C \!\int_{(0,T)}  \lp \int  \lb \psi\lp x,\tau,{y_\e(x_3)\over r_\e}\rp\rb^2 \!\!dm_\e  \rp^{1\over2} d\tau.
\end{aligned}
\nonumber
\end{equation}

\ni By passing to the limit as $\e\to 0$ in the last inequality, thanks to (\ref{inu2})    and to the Dominated Convergence Theorem, we get
$ \int_{\OT\times I}   f_0\psi   nd xdtdy_3 \!\le C |\psi|_{L^1(0,T; L^2_n (\Omega\times I))} $
%
%
and deduce, by the arbitrary choice of $\psi$,  that $f_0\in L^\infty(0,T; L^2_n (\Omega\times I))$.
  \end{proof}
}


  \subsection{{\color{black} Two classical  results}}\label{subsecabstract}
  {\color{black} 
 For the reader's convenience, we reproduce below   Lemma A2 of \cite{BeBo} 
 (see  also \cite{Bo}   for  a more general version) and 
Theorem 6.2 of \cite{BeSiam},  which collects some abstract results proved in  \cite{DaLi,Li,LiMa}.
The lemma   will be employed to identify the limit of the sequence 
$(\bfu_\e m_\e)$, where $\bfu_\e$ is the solution to \eqref{Pe}.
The theorem  
  will be applied  to check  the existence,  the  uniqueness, and some regularity properties  of the solution 
to Problem  (\ref{Pe}) and of the  associated limit problems.}

\begin{lemma} \label{feps} Let $K$ be a compact subset of $\RR^N$ and 
 $(\theta_\e)$   a  bounded   sequence of   positive Radon measures on  $K $,  weakly* 
converging in $\M(K)$    to some   $\theta\in \M(K)$. 
 Let $(f_\e)$  be   a sequence of
$\theta_\e$-measurable functions such
that $\sup_\e \int |f_\e|^2 d\theta_\e<+\infty$. Then the sequence  
$(f_\e \theta_\e)$
is sequentially  relatively compact in the weak* topology $\sigma(\M (K), C(K))$ and every
cluster point  is of the form $ f\theta$, with $f \in L^2_\theta$. 
 Moreover, if $f_\e\theta_\e\buildrel\star \over \rightharpoonup f\theta$, then 
 
 \begin{equation}
\begin{aligned}
& \liminf_{\e\to 0} \int |f_\e|^2d\theta_\e \ge \int |f|^2d\theta.
\end{aligned}
\label{lb}
\end{equation}
 
 \end{lemma}

%

 \begin{theorem}\label{thDautrayLions}    Let 
 $V $ and  $ H $ be separable  Hilbert spaces  such that   
$ 
V\subset H=H'\subset V',
 $
with  continuous and dense imbeddings. Let $||.||_V,\ |.|_H,\ ((.,.))_{ V}, \
(.,.)_{ H}$
denote their respective    norm  and inner product.  
Let $ a: V\times V \to \RR $ be  a continuous bilinear symmetric form on  $V$. Let  $A\in \L(V,V')$ be  defined by $a(\xi,\widetilde\xi)= (A\xi, \widetilde\xi)_{(V',V)}, \ \forall \
(\xi,\widetilde\xi)\in V^2$. Assume that 

\begin{equation}
\exists (\l,\a)\in \RR_+ \times \RR_+^*, \ \ a( \xi,\xi)+ \l  |\xi |_H^2 \ge \a ||\xi||_V^2, \ \forall \ \xi\in V.
\label{hypC}
\end{equation}

\ni Let  $h\in L^2(0,T; H)$, $ \xi_0 \in V$, $\xi_1 \in H $.
Then there exists a unique solution $\xi$ to
\begin{equation}
\begin{aligned}
 & A\xi(t)+{\partial^2 \xi\over \partial t^2}(t)= h(t), \quad  \xi\in L^2(0,T; V) , 
 \\ &     {\partial \xi\over \partial t}\in L^2(0,T; H),   \quad \xi(0) = \xi_0,  \quad {\partial \xi\over \partial t}(0)= \xi_1.
\end{aligned}\label{Pvar}
\end{equation}

\ni  
Furthermore, we have 
\begin{equation}
   \xi\in C([0,T]; V)\cap   C^1([0,T]; H) ,  \quad {\partial \xi\over \partial t}\in L^2(0,T; V ), \quad {\partial^2 \xi\over \partial t^2}\in L^2(0,T; V').
\label{xireg}
\end{equation}

\ni Besides, setting   

\begin{equation}{e}(\tau):= {1\over 2} \lc
\lp  {\partial \xi\over \partial t}(\tau),{\partial \xi\over \partial t}(\tau)\rp_H\!+ a(\xi(\tau),\xi(\tau))
 \rc,\ \forall \ \tau\in [0,T],
\label{defenergy}
\end{equation}
 
 \ni  the following holds  

\begin{equation} 
{e}(\tau)={e}(0)+ \int_0^\tau \lp h,{\partial \xi\over \partial t}\rp_H dt,  \qquad \forall \ \tau\in [0,T]  .
\label{energy}
\end{equation}

\ni  Problem (\ref{Pvar})  is equivalent to   

\begin{equation}
\begin{aligned}
&\int_{0}^T  \!\lp a(\xi(t),\widetilde\xi)\eta(t)+ (\xi(t),\widetilde\xi)_H{\partial^2 \eta\over \partial t^2}(t)\rp \! dt
  +  (\xi_0,\widetilde\xi)_H {\partial \eta\over \partial t}(0)
   \\
&   \hskip3,7cm - (\xi_1,\widetilde\xi)_H \eta(0)=\!\!\int_{0}^T\!\!(h,\widetilde\xi)_H\eta(t) dt ,  \\
&   \hskip1cm   \forall \ \widetilde\xi\in V, \quad \forall \ \eta\in \D(]-\infty,T[); \quad \xi\!\in \!L^2(0,T;V) , \  {\partial \xi\over \partial t}\! \in \!L^2(0,T;H).
\end{aligned}
\label{Pvar2}
\end{equation}

\end{theorem}
 
 {\color{black}
\section{ Asymptotic    behavior of the solution to \eqref{Pe}}\label{secapriori}  {\color{black} In this section, we establish a series of   estimates satisfied by the solution $\bfu_\e$ to (\ref{Pe}) (see Proposition \ref{propapriori}), and 
investigate  in  lemmas \ref{lemident1}, \ref{lemident2} , and \ref{lemidentu0v}, the asymptotic behavior of sequences satisfying such estimates. 
These results are synthetized in 
 Corollary   \ref{corapriori}. We start with a key inequality. }
    
  \begin{lemma}\label{lemkey} 
We have 
 
  \begin{equation}
\label{key}
\begin{aligned}
& \int   \lp  \lb {\varphi_1 \over  r_\e}\rb^2 +\lb  {\varphi_2\over  r_\e}\rb^2 +
\left|\varphi_3\right|^2 
 \rp dm_\e 
  \le {C\over r_\e^2}\int  \lb  
\bfe(\bfvarphi )\rb^2dm_\e  \quad \forall \bfvarphi\in H^1_0(\O;\RR^3) . 
\end{aligned}
\end{equation}

\end{lemma}

  \begin{proof}  By (\ref{defBe}) and (\ref{defme}), it is sufficient  to show that 
 for all $j\in J_\e$, and all $\bfvarphi\in H^1(B_\e^j;\RR^3)$ 
 such that 
 $\bfvarphi=0$ on  $\partial B_\e^j\cap \partial \Omega$, 

 \begin{equation}
\label{keyBj}
\begin{aligned}
& \int_{B_\e^j}   \lp  \lb {\varphi_1 \over  r_\e}\rb^2 +\lb  {\varphi_2\over  r_\e}\rb^2 +
\left|\varphi_3\right|^2 
  \rp dx  
 \le {C\over r_\e^2}\int_{B_\e^j}  \lb  
\bfe(\bfvarphi )\rb^2dx .
\end{aligned}
\end{equation}

 \ni   By    Korn's inequality, we have 

 \begin{equation}
\nonumber
\begin{aligned}
&\int_{\Omega'\times \lp-{1\over 2}; {1\over2} \rp} |\bfpsi|^2 
dy \le C \int_{\Omega'\times \lp-{1\over 2}; {1\over2} \rp} |\bfe(\bfpsi)|^2 dy\qquad  \forall \bfpsi \in W,
\\&
W:=\la\bfpsi\in H^1\lp \Omega'\times \lp-{1\over 2}; {1\over2} \rp;\RR^3\rp , \quad \bfpsi =0 \hbox{ on }
\partial \Omega'
 \times \lp-{1\over 2}; {1\over2}\rp\ra.
 \end{aligned}
\end{equation}

 \ni By making the change of variable $y=(x_1,x_2,\frac{x_3-\omega_\e^j}{r_\e})$, we get, for all $\bfpsi \in W $, 

 \begin{equation}
\label{pW}
\begin{aligned}
\int_{B_\e^{j}}& \lb \bfpsi\rb^2\lp x_1,x_2,\frac{x_3-\omega_\e^j}{r_\e}\rp
 \! dx
  \le \!C\int_{B_\e^{j}} \lb\bfe(\bfpsi)\rb^2\! \lp x_1,x_2,\frac{x_3-\e i}{r_\e}\rp  \!dx.
  \end{aligned}
\end{equation}

\ni 
Setting 
$
\bfvarphi(x)=\begin{pmatrix}   \psi_1 
\\ \psi_2  
\\  \frac{1}{r_\e}\psi_3  
\end{pmatrix}\lp x_1,x_2,\frac{x_3-\omega_\e^j}{r_\e}\rp,
$ a straightforward computation yields 
 
 \begin{equation}
\nonumber
 \begin{aligned}
&\int_{{B_\e^{j}}}\lp   \lb  \varphi_1 \rb^2 +\lb   \varphi_2  \rb^2 +
r_\e^2 \left| \varphi_3   \right|^2 
 \rp(x)  dx 
= \int_{B_\e^{j}}     |\bfpsi|^2 
   \lp x_1,x_2,\frac{x_3-\omega_\e^j }{r_\e}\rp dx
 \\&\hskip 2 cm \le  C \int_{B_\e^{j}} |\bfe(\bfpsi)|^2 \lp x_1,x_2,\frac{x_3-\omega_\e^j }{r_\e}\rp dx
 \le C  \int_{{B_\e^{j}}} \lb  \bfe(\bfvarphi)\rb^2(x) dx.
 \end{aligned}
\end{equation}

\ni  The inequality (\ref{keyBj}) is proved. 
 \end{proof}

%
%
 

\begin{proposition}\label{propapriori} There exists a unique solution $\bfu_\e$ to 
 (\ref{Pe}). Moreover, 

 \begin{equation}
\label{uereg}
\begin{aligned}
 {\partial  \bfu_\e  \over \partial  t } \in    L^2(0,T;H^1_0(\O; \RR^3)) ,
\  {\partial^2 \bfu_\e  \over \partial  t^2}\in L^2(0,T; H^{-1}(\O; \RR^3)).\end{aligned}
\end{equation}

\ni 
Under  (\ref{lmu}),      (\ref{hyp2}),    
there exists a    constant $C>0$ such that

\begin{equation}
\begin{aligned}
& \int_{\O }\!\!   \mu_{0\e}|\bfe( \bfu_\e)|^2( \tau)dx +\int_\O  \lp  \re\lb {\partial \bfu_\e\over \partial t} \rb^2\! +   |\bfu_\e |^2  
\rp( \tau)
dx  
\!\le\! C  
 \quad & &  \forall \tau \in [0,T] ,
 \\  & 
  \int  \lb \bfe \lp \bfu_\e   \rp  \rb^2( \tau)  
  +  \lb   \bfu_\e'     \rb^2( \tau)  dm_\e   \le  C \lp \frac{r_\e}{\e} \mu_{1\e}\rp^{-1}   & & \forall \tau \in [0,T]  ,
   \\  & 
  \int  \lb  u_{\e3}    \rb^2( \tau)  dm_\e   \le  C \lp \frac{r_\e^3}{\e} \mu_{1\e}\rp^{-1} , \qquad   \int  \lb  \bfu_{\e}    \rb^2( \tau)  dm_\e   \le  C   & & \forall \tau \in [0,T].
  \end{aligned}
\label{apriori}
\end{equation}

%
\end{proposition}

 \begin{proof} 
Problem (\ref{Pe}) is equivalent to   (\ref{Pvar2}), where  $H:=L^2(\O;\RR^3)$, $(\bfxi,\widetilde\bfxi)_H:=\int_\O \rho_\e \bfxi\cdot \widetilde\bfxi dx $, $V:=H^1_0(\O;\RR^3)$ ($V'=H^{-1}(\O;\RR^3)$),  $a(\bfxi,\widetilde\bfxi):= \int_\O \bfsigma_\e(\bfxi):\bfe(\widetilde{\bfxi})dx$, and 
$(\bfxi_0,\bfxi_1,\bfh)=(\bfa_0,\bfb_0,   \bff)$. By (\ref{defrhoe}), 
 $(H, (.,.)_H)$ is a Hilbert space  and  
the assumptions of   Theorem \ref{thDautrayLions} are satisfied. Therefore, Problem   (\ref{Pe})   has a unique solution. Assertion  (\ref{uereg}) follows  from (\ref{xireg}). 
 By (\ref{energy}) we have, for all $  \tau\in [0,T]$,

\begin{equation}
\begin{aligned}
  {1\over 2}\int_\O   &      \lp\re\left|{\partial \bfu_\e\over \partial t}\right|^2   +  \bfsigma_\e(
\bfu_\e):\bfe(\bfu_\e) \rp
(
\tau) dx  
 \\
 &  \hskip0,5cm =  {1\over 2}\int_\O  \lp\re\left| \bfb_0 \right|^2   + \bfsigma_\e ( \bfa_0 ):\bfe(\bfa_0  )\rp dx +  \int_{\O\times (0,\tau)}
\re
\bff\cdot {\partial 
\bfu_\e
\over\partial t}  dx dt.
\end{aligned}
\label{side}
\end{equation}

\ni We deduce from  Cauchy-Schwarz Inequality that 

$$
 \lb\int_{\O\times (0,\tau)}
\re
\bff\cdot {\partial 
\bfu_\e
\over\partial t}  dx dt\rb \le \sqrt{\int_\OT  \re 
\left|
\bff\right|^2dxdt} \sqrt{\int_\OT  \re 
\left|
{\partial  \bfu_\e  \over \partial t}\right|^2dxdt}.
$$

\ni Taking (\ref{hyp2}) into account, we infer

\begin{equation}
\begin{aligned}
\int_\O     &  \lp
 \re\left|{\partial \bfu_\e\over \partial t}\right|^2 dx +  
\bfsigma_\e( \bfu_\e):\bfe(\bfu_\e) \rp( \tau) dx  
  \\
&  \hskip 1,5cm \le  C\lp 
1 +    \sqrt{\int_\OT  \re 
\left|
{\partial  \bfu_\e  \over \partial t}\right|^2dxdt} 
\rp \qquad \forall  \tau \in [0,T]
.
\end{aligned}
\label{A}
\end{equation}

\ni
By integrating  (\ref{A}) with respect to $\tau$ over $(0,T)$, we deduce   that  
 $
  \int_\OT  
\re
\lb{\partial 
\bfu_\e
\over
\partial t}\rb^2$ $dxdt  \le C  
$
 and then, coming back to  (\ref{A}),  that
\begin{equation}
\int_\O     \re   \left|{\partial \bfu_\e\over
\partial t}\right|^2 ( \tau)   dx +    
\int_\O    
\bfsigma_\e( \bfu_\e ):\bfe(\bfu_\e )( \tau)dx 
 \le C \qquad \forall \tau \in [0,T].
\label{B}
\end{equation}

\ni
We infer from    (\ref{Pe}),    (\ref{lmu}), (\ref{defme}),   and (\ref{B}),  that 

\begin{equation}
\begin{aligned}
 & \int_{\O } \re    \left|{\partial \bfu_\e\over
\partial t} \right|^2\!\! ( \tau)   +  \mu_{0\e}|\bfe (\bfu_\e)|^2 ( \tau)  dx \le  C\quad & & \forall \tau \in [0,T],
  \\
&  
\int   |\bfe(\bfu_\e)|^2( \tau) dm_\e  \le  C\lp\frac{r_\e}{\e}  \mu_{1\e}\rp^{-1}\quad & &  \forall \tau \in [0,T]
.
\end{aligned}
\label{B2} 
\end{equation}

 \ni By  (\ref{defrhoe}),   (\ref{defme}),   and (\ref{B2}),  and by the continuity of $\bfa_0$ (see (\ref{Pe})), we have

\begin{equation}
\begin{aligned}
\int_\O      |\bfu_\e |^2& ( \tau)   dx +\int    |\bfu_\e |^2 ( \tau)   dm_\e 
 =  
\int_\O          \left|\bfa_0 + \int_0^\tau
{\partial \bfu_\e\over
\partial t} (t)dt \right|^2    d(\L^3+m_\e)(x)
\\&\le C\lp 1+  \int_{\Omega\times(0,T)} 
 \rho_\e \lb {\partial \bfu_\e\over
\partial t}\rb^2(t)  dxdt
\rp
  \le
 C \hskip 0,5cm  \forall   \tau \in   [0,T].
  \end{aligned}
\label{uborne}
\end{equation}

 \ni By 
 (\ref{key}) and \eqref{B2}, we have 

\begin{equation}
\begin{aligned}
   \int    \lb\frac{u_{\e1}}{\e} \rb^2 + 
\lb\frac{u_{\e2}}{\e} \rb^2 +   \left | u_{ \e3} \right|^2( \tau) dm_\e  
 &\le  {C \over
r_\e^2   }\int  |\bfe(\bfu_\e)|^2(\tau) dm_\e
   \le {C\over \e^2\mu_{1\e}}\hskip 1cm \forall \tau \in [0,T],
\end{aligned}
\nonumber
\end{equation}
which, 
 joined with
 (\ref{B2}),  (\ref{uborne})  completes the proof of     (\ref{apriori}). 
    \end{proof}


%
}


\begin{lemma}\label{lemident1}
Let $(\bfu_\e)$ be a sequence in $L^\infty(0,T; H^1_0(\Omega;\RR^3))$ satisfying
 
 \begin{equation}
   \begin{aligned}
\sup_{\e>0, \tau\in (0,T)} \int |\bfu_\e|^2+ \lb \bfe(\bfu_\e)\rb^2 (\tau ) d m_\e <+\infty.
\end{aligned}
\label{euborne22}
  \end{equation}
  Then there exists $\bfv\in L^\infty(0,T; L_n^2(\Omega;\RR^3))$ such that, up to a subsequence, 

\begin{equation}
\begin{aligned}
 &   
\bfu_\e m_\e  \buildrel \star \over \rightharpoonup  n\bfv   
& &  \hbox{weakly* in } \ L^\infty(0,T; \M (\ov\Omega;\RR^3) ),
\\&v_1, v_2 \in L^\infty(0,T; L_n^2(0,L; H^1_0(\Omega'))),
 \\& \bfe_{x'}(\bfu_\e')m_\e  \buildrel \star \over \rightharpoonup 
n \bfe_{x'} (\bfv')
\ 
& & \hbox{weakly* in } \ L^\infty(0,T; \M (\ov\Omega; \SS^3) ).
\end{aligned}
 \label{uemetov}
 \end{equation}
 
 {\color{black}
 \ni Furthermore,
 
  \begin{equation}
   \begin{aligned}
& n v_1=n v_2=0\quad & &  \hbox{if} \qquad \liminf_{\e\to0} \sup_{\tau\in (0,T)} \int   \lb \bfe(\bfu_\e)\rb^2 (\tau ) d m_\e  =0,
\\& n\bfv=0 & &  \hbox{if} \qquad \liminf_{\e\to0} \sup_{\tau\in (0,T)} \int   \lb \frac{1}{r_\e} \bfe(\bfu_\e)\rb^2 (\tau ) d m_\e  =0.
\end{aligned}
\label{v=0}
  \end{equation}

 }
 
 \ni  Moreover, if 
   
  \begin{equation}
   \begin{aligned}
&\sup_{\e>0, \tau\in (0,T)} \int_\Omega \mu_{0\e}\lb \bfe(\bfu_\e)\rb^2 (\tau ) d x <+\infty ;\quad \e^2<\!\!<\mu_{\e0},
\\& n_\e \to n   \hbox{ strongly in }   L^2(\Omega)  \hbox{ and }   \bfu_\e \buildrel \star\over \rightharpoonup  \bfu    \hbox{ weakly* in }  L^\infty(0,T; L^2(\Omega;\RR^3)),
\end{aligned}
\label{H1}
  \end{equation}
  
\ni or   if  
 
  \begin{equation}
   \begin{aligned}
&\sup_{\e>0, \tau\in (0,T)} \int_\Omega  \lb \bfe(\bfu_\e)\rb^2 (\tau ) d x <+\infty ,
\\&   \bfu_\e \to  \bfu \quad \hbox{ strongly  in }  L^2(0,T; L^2(\Omega;\RR^3)),
\end{aligned}
\label{H2}
  \end{equation}
  
  \ni then $  n\bfu=n\bfv.$
%
%
   
\end{lemma}

  
  \begin{proof} 
  By  applying Lemma \ref{feps}   to $\theta_\e:= \L^1_{\lfloor(0,T)}\otimes m_\e $, $K:= \ov{(0,T)\times \Omega}$, $f_\e \in \{\bfu_\e,  \bfe_{x'}(\bfu_\e')\}$, 
  taking 
 (\ref{meton})   and (\ref{euborne22}) into account,  we obtain
 the  convergences 
  
  \begin{equation}
\begin{aligned}
  &   
\bfu_\e m_\e  \buildrel \star \over \rightharpoonup  n \bfv   
& &  \hbox{weakly* in } \ \M (\ov{(0,T)\times \Omega};\RR^3) ),
 \\& \bfe_{x'}(\bfu_\e')m_\e  \buildrel \star \over \rightharpoonup 
n\bfXi
\ 
& & \hbox{weakly* in }  \ \M (\ov{(0,T)\times \Omega};\SS^3) ),
\end{aligned}
\label{cvcv}
 \end{equation}
 
 \ni  up to a subsequence,  for some suitable 
  $\bfv\in L^2(0,T;L^2_n(\Omega;\RR^3))$,  $\bfXi \in L^2(0,T;L^2_n(\Omega; \SS^3))$ such that $ \Xi_{ij}=0$ if $3\in \{i,j\}$.
 By (\ref{euborne22}), the sequences $(\bfu_\e m_\e) $ and $(\bfe_{x'}(\bfu_\e')m_\e)$ are bounded in $L^\infty(0,T; \M (\ov\Omega ) )$, therefore the   convergences (\ref{cvcv}) also hold  with respect to the  weak* topology of  $L^\infty(0,T; \M (\ov\Omega ) )$.
Let us fix   $\bfPsi\in C^\infty(\ov{\OT};  \SS^2)$. By integration by parts, we have 

\begin{equation}
\begin{aligned}
\sum_{\alpha,\beta=1}^2  \int_\OT \hskip-0,9cm (\bfe_{x'}(\bfu_\e'))_{\a\b} &\Psi_{\a\b}\lp x,t \rp dm_\e dt 
 =-\int_\OT    \hskip-0,9cm  (u_{\e1}\bfe_1+u_{\e2}\bfe_2) \cdot \bfdiv_{x'}\bfPsi\lp x,t \rp dm_\e dt.
\end{aligned}
\nonumber
\end{equation}

 \ni  By  passing  to the limit as $\e\to 0$,  taking (\ref{cvcv}) into account,  we obtain
   
\begin{equation}
\begin{aligned}
\sum_{\alpha,\beta=1}^2  
\int_\OT \Xi_{\a\b}  \Psi_{\a\b} ndxdt & = - \int_\OT(v_1\bfe_1+v_2\bfe_2)  \cdot \bfdiv_{x'}\bfPsi ndxdt.
 \end{aligned}
\label{ajout}
\end{equation}

\ni 
By making  $\bfPsi$ vary in $\D(\OT; \SS^2)$, we deduce that $n\bfe_{x'}(\bfv')(= \bfe_{x'}(n\bfv'))=n\bfXi$ in 
the sense of distributions on $\OT$, and then infer from Korn inequality in $H^1 (\Omega')$ that  
$nv_1, nv_2 \in L^\infty(0,T; L^2(0,L; H^1 (\Omega')))$, that is  $ v_1,  v_2 \in L^\infty(0,T; L^2_n(0,L; H^1 (\Omega')))$.
By integrating (\ref{ajout})  by parts for  $\bfPsi\in C^\infty(\ov{\OT};  \SS^2)$, 
we get $ \int_0^T \int_0^L \int_{\partial\Omega'} n \bfv' \cdot \bfPsi\bfnu $ $ d\H^1 dx_3 dt =0$ and infer   
from   the arbitrariness  of $\bfPsi$    that  $v_1, v_2 \in L^\infty(0,T; L^2_n(0,L;$ $ H^1_0(\Omega')))$.   
{\color{black}
 Assertion \eqref{v=0} is a consequence of the next inequalities (holding for $\a\in \{1,2\}$), deduced from 
\eqref{lb},
  \eqref{key}, \eqref{cvcv}  
\begin{align*}
 & \int_{\OT}\hskip-0,5cm  |v_\a|^2 n dxdt  \le \liminf_{\e\to0} \int_{\OT} \hskip-0,5cm |u_{\e\a}|^2 dm_\e(x)dt 
  \le  \liminf_{\e\to0} \int_{\OT}  \lb  
\bfe(\bfu_\e )\rb^2dm_\e (x)dt,
\\&  \int_{\OT}\hskip-0,5cm  |\bfv|^2 n dxdt  \le \liminf_{\e\to0} \int_{\OT} \hskip-0,5cm |\bfu_\e|^2 dm_\e(x)dt 
  \le  \liminf_{\e\to0} \int_{\OT}  \lb  \frac{1}{r_\e}
\bfe(\bfu_\e )\rb^2dm_\e (x)dt .
\end{align*}
 }

\ni It remains to show that under  (\ref{H1})  or (\ref{H2}), $n\bfu=n\bfv$. To that aim, we set

\begin{equation}\label{defhatve}\begin{aligned}
&\hat\bfv_\e(x ,t):= \sum_{j\in   J_\e}     \bfu_\e(x_1,x_2,\omega_\e^j  ,t) \mathds{1}_{\lp\omega_\e^j -\frac{r_\e}{2}, \omega_\e^j +\frac{r_\e}{2}\rc }(x_3).
\end{aligned}\end{equation}

\ni  By Fubini's Theorem, Jensen's inequality and Korn's inequality in $H^1_0(\Omega;\RR^3)$, we have

 \begin{equation}
\label{estimhatve}
\begin{aligned}
\int \!\! |\bfu_\e-  \hat\bfv_\e |^2& (\tau)   dm_\e
\! =\!{\e\over r_\e}\!\sum_{j\in J_\e}\!
\int_{\Omega'}  \!dx' \!\!\int_{\omega_\e^j-{r_\e\over2}}^{\omega_\e^j+{r_\e\over2}} \!\lb \bfu_\e(x',x_3,\tau)\!-\!\bfu_\e(x',\omega_\e^j,\tau)\rb^2 \!\!dx_3  
\\&
\le{\e\over r_\e}\sum_{j\in  J_\e}
\int_{\Omega'} dx'  \int_{\omega_\e^j-{r_\e\over2}}^{\omega_\e^j+{r_\e\over2}} \lp\int_{\omega_\e^j-{r_\e\over2}}^{\omega_\e^j+{r_\e\over2}} \lb {\partial \bfu_\e\over \partial x_3}\rb (x',s_3,\tau) ds_3 \rp^2 dx_3 
\\&   \le \e r_\e  \int_\O |\bfnabla \bfu_\e|^2(\tau) dx\le C\e r_\e  \int_\O |\bfe( \bfu_\e)|^2 (\tau) dx.
\end{aligned}
\end{equation}

\ni Therefore, under  (\ref{H1})  or (\ref{H2}),      $\lime \int \!\! |\bfu_\e-  \hat\bfv_\e |^2  (\tau)   dm_\e=0$. Hence,  by   (\ref{cvcv}), 

\begin{equation}
\begin{aligned}
 &   
\hat\bfv_\e m_\e  \buildrel \star \over \rightharpoonup  n\bfv   
& &  \hbox{weakly* in } \ L^\infty(0,T; \M (\ov\Omega;\RR^3) ) .
\end{aligned}
 \label{hatvetov}
 \end{equation}
 
 \ni We define (see (\ref{defomegae}, \ref{defZe}))
 
  \begin{equation}\label{defovve}\begin{aligned}
&\ov\bfv_\e(x ,t):=  \sum_{i\in   Z_\e} \lp  \sum_{\la
\omega_\e^j  \in \omega_\e  \cap \lp \e i-\frac{\e}{2}, \e i+\frac{\e}{2} \rc\ra}   \bfu_\e(x',\omega_\e^j  ,t) \rp \mathds{1}_{\lp\e i -\frac{  \e}{2},\e i +\frac{  \e}{2}\rc }(x_3).
\end{aligned}\end{equation}

\ni Noticing that by  \eqref{condomegae}  and  \eqref{defZe}, 

 \begin{equation}
 \begin{aligned}
 &\omega_\e  \subset \bigcup_{i\in Z_\e }  \lp \e i-\frac{ \e}{2}, \e i+\frac{ \e}{2}\rc,
   \end{aligned}
  	\label{calCesubsetZe}
\end{equation}

\ni  we infer from \eqref{defme} and \eqref{defhatve} that 

\begin{equation}\label{ovvdx=hatvdm}\begin{aligned}
 \int_\Omega |\ov\bfv_\e|^2 (\tau)dx &\le C  \sum_{i\in   Z_\e}    \sum_{\la
\omega_\e^j  \in \omega_\e  \cap \lp \e i-\frac{\e}{2}, \e i+\frac{\e}{2} \rc\ra} \e \int_{\Omega'}| \bfu_\e(x',\omega_\e^j  ,\tau)|^2 dx'
\\&=C \frac{\e}{r_\e} \sum_{j\in J_\e} r_\e \int_{\Omega'} | \bfu_\e(x',\omega_\e^j  ,\tau)|^2 dx'
 =  C \int  | \hat\bfv_\e|^2  (\tau)dm_\e . 
\end{aligned}\end{equation}

\ni Therefore, by  (\ref{euborne22}) and (\ref{estimhatve}), under  (\ref{H1})   or (\ref{H2}),  the sequence $(\ov\bfv_\e)$ is bounded in $L^\infty(0,T; L^2(\Omega;\RR^3))$. Thus the following convergence holds, up to a subsequence

 \begin{equation}
 \begin{aligned}
 &\ov\bfv_\e \buildrel \star\over \rightharpoonup \bfw  \quad \hbox{ weakly* in } \ L^\infty(0,T; L^2(\Omega;\RR^3)).\end{aligned}
  	\label{ovvetow}
\end{equation}

\ni To identify $\bfw$, we  fix  $\bfvarphi\in \D(\OT ;\RR^3)$ and set 
$\breve\bfvarphi_\e(x ,t):= \sum_{i\in   Z_\e}$ $ \bfvarphi(x',\e i,t) $ $\mathds{1}_{\lp\e i -\frac{ \e}{2},\e i +\frac{ \e}{2}\rc }(x_3)$. Noticing that $|\breve\bfvarphi_\e-\bfvarphi|_{L^\infty(\OT;\RR^3)}\le C\e$, we infer

  \begin{equation}\label{wphi} \begin{aligned}
  \int_\OT \bfw\cdot  \bfvarphi dxdt &= \lime 
 \int_\OT \ov \bfv_\e \cdot  \breve\bfvarphi_\e dxdt. 
\end{aligned}\end{equation}

\ni On the other hand, by \eqref{defme},   (\ref{defhatve}) and (\ref{defovve}),  we have

\begin{equation}\label{418}\begin{aligned}
  \int_\O  \ov \bfv_\e \cdot  \breve\bfvarphi_\e (t) dx 
 &=  \sum_{i\in   Z_\e}    \sum_{\la
\omega_\e^j  \in \omega_\e  \cap \lp \e i-\frac{\e}{2}, \e i+\frac{\e}{2} \rc\ra}  \int_{\Omega'\times \lp \e i-\frac{\e}{2}, \e i+\frac{\e}{2} \rc}  \hskip-1,5cm  \bfu_\e(x',\omega_\e^j  ,t) \cdot  \bfvarphi(x',\e i,t) dx
\\& = \sum_{i\in   Z_\e}    \sum_{\la
\omega_\e^j  \in \omega_\e  \cap \lp \e i-\frac{\e}{2}, \e i+\frac{\e}{2} \rc\ra}  \int_{\Omega'\times \lp \omega_\e^j +\frac{r_\e}{2}, \omega_\e^j +\frac{r_\e}{2} \rp}    \hskip-1,5cm \hat \bfv_\e(x   ,t) \cdot  \bfvarphi(x',\e i,t) dm_\e.
\end{aligned}\end{equation}

\ni For all $ i\in   Z_\e$ and all   $\omega_\e^j  \in \omega_\e  \cap \lp \e i-\frac{\e}{2}, \e i+\frac{\e}{2} \rc$, we have 
$   |\bfvarphi(x_1,x_2,\e i,t)-\bfvarphi (x,t)|\le C\e$ in $\Omega'\times  \lp \omega_\e^j +\frac{r_\e}{2}, \omega_\e^j +\frac{r_\e}{2} \rp\times (0,T)$. 
Taking \eqref{calCesubsetZe} into account, we deduce  

\begin{equation}
\begin{aligned}
&\lb
\sum_{i\in   Z_\e}    \sum_{\la
\omega_\e^j  \in \omega_\e  \cap \lp \e i-\frac{\e}{2}, \e i+\frac{\e}{2} \rc\ra}  \int_{\Omega'\times \lp \omega_\e^j +\frac{r_\e}{2}, \omega_\e^j +\frac{r_\e}{2} \rp}    \hskip-1,5cm \hat \bfv_\e(x   ,t) \cdot  \bfvarphi(x',\e i,t) dm_\e
- \int \hat\bfv_\e(t) \cdot\bfvarphi (t)dm_\e\rb
\\& \hskip9cm  \le C \e \int  \lb \hat\bfv_\e\rb(t) dm_\e.
\end{aligned}
\nonumber
\end{equation}

\ni Therefore, by (\ref{hatvetov}) and (\ref{418}), the following  holds

\begin{equation}\label{nvphi} \begin{aligned}
 \lime  \int_\OT  \ov \bfv_\e \cdot  \breve\bfvarphi_\e   dx dt&=  \lime \int_\OT  \hat \bfv_\e  \cdot  \bfvarphi  dm_\e dt = \int_\OT  n\bfv\cdot  \bfvarphi dxdt .
\end{aligned}\end{equation}

\ni Joining (\ref{wphi}) and (\ref{nvphi}) we deduce,  by the arbitrary choice of $\bfvarphi$,  

\begin{equation}
 \bfw=n\bfv.
 \label{w=nv}\end{equation}

\ni On the other hand,   
%
%
by (\ref{defne}),   (\ref{defovve}), 
   (\ref{calCesubsetZe}),    Jensen's inequality and  Korn's inequality in $H^1_0(\Omega)$,  
we have

\begin{equation}
\label{neue-ovve} \begin{aligned}
 \int_\Omega & \lb n_\e \bfu_\e-\ov\bfv_\e \rb^2 (\tau) dx 
  \\&
= \sum_{i\in Z_\e}   \int_{\Omega'\times\lp \e i-\frac{\e}{2}, \e i+\frac{\e}{2} \rc} 
 \lb  \sum_{\la
\omega_\e^j  \in \omega_\e  \cap \lp \e i-\frac{\e}{2}, \e i+\frac{\e}{2} \rc\ra} \bfu_\e-\bfu_\e(x', \omega_\e^j,\tau) \rb^2 dx
 \\&
 \le C \sum_{i\in Z_\e}   \sum_{\la
\omega_\e^j  \in \omega_\e  \cap \lp \e i-\frac{\e}{2}, \e i+\frac{\e}{2} \rc\ra} \int_{\Omega'\times\lp \e i-\frac{\e}{2}, \e i+\frac{\e}{2} \rc} 
 \lb  \bfu_\e-\bfu_\e(x', \omega_\e^j,\tau) \rb^2 dx
 \\& \le  C \e^2 \sum_{i\in Z_\e}   n_\e(\e i)  \int_{\Omega'\times\lp \e i-\frac{\e}{2}, \e i+\frac{\e}{2} \rc} 
 \lb \frac{\partial \bfu_\e}{\partial x_3} \rb^2(\tau) dx\le C   \e^2    \int_\Omega
 \lb \bfe( \bfu_\e) \rb^2(\tau) dx .
\end{aligned}\end{equation}

\ni We infer that,  under either (\ref{H1})   or (\ref{H2}), 
the sequence  $
\big( \int_\OT   \lb n_\e \bfu_\e-\ov\bfv_\e \rb^2 (\tau) $ $dx dt\big)$ converges to $0$. 
Therefore, by \eqref{ovvetow} and \eqref{w=nv},

 \begin{equation}
 \begin{aligned}
 &n_\e \bfu_\e \buildrel \star\over \rightharpoonup n\bfv  \quad \hbox{ weakly* in } \ L^\infty(0,T; L^2(\Omega;\RR^3)).\end{aligned}
  	\label{neuetow}
\end{equation}
%
Under (\ref{H2}), it easily follows from  the strong convergence of   $(\bfu_\e)$  to $\bfu$ in $L^2(0,T; L^2(\Omega;$ $\RR^3))$ and  the weak* convergence of $(n_\e)$ to $n$ in $L^\infty(\Omega)$ (see (\ref{neton})),  
that    $(n_\e\bfu_\e)$  weakly  converges to $n\bfu$ in 
 $L^2(0,T; L^2(\O ;\RR^3))$,  therefore $n\bfu=n\bfv$.
The same conclusion holds under   (\ref{H1}), because  $(\bfu_\e)$  weakly*  converges to $\bfu$ in  $L^\infty(0,T; L^2(\O ;\RR^3))$ and $(n_\e)$ is bounded in $L^\infty(\Omega)$ and  strongly converges to $n$ in $L^2(\Omega)$.
The proof of Lemma \ref{lemident1}  is achieved.  
\end{proof}

   

\begin{lemma}\label{lemident2}
Let $(\bfu_\e)$ be a sequence in $L^\infty(0,T; H^1_0(\Omega;\RR^3))$ satisfying

 \begin{equation}
   \begin{aligned}
\sup_{\e>0, \tau\in (0,T)} \int 
 \lb \frac{1}{r_\e}\bfe(\bfu_\e)\rb^2 (\tau ) d m_\e <+\infty.
\end{aligned}
\label{eusurrborne}
  \end{equation}

\ni
Then, up to a subsequence,  the convergences (\ref{uemetov})
 take place. 
Moreover,

  \begin{equation}
\begin{aligned}
 &v_3\in L^\infty(0,T; L_n^2(0,L; H_0^2(\Omega' ))).
\end{aligned}
\label{flexion0}
\end{equation}

\ni Besides,  
up to a subsequence,
the following  convergences  hold  (see (\ref{defmdto})):

    \begin{equation}
\begin{aligned}
 &  {u_{\e \a}\over r_\e}  \ 
 \mdto \   \xi_{\a}(x,t) - {\partial v_3\over \partial x_\a}(x,t) y_3\qquad (\a\in \{1,2\}), 
 \\&   \lp{ 1 \over r_\e}\bfe(\bfu_\e) \rp_{\a\b}    \mdto \    {1\over 2}  \lp {\partial \xi_\a\over \partial x_\b}+ {\partial \xi_\b\over \partial x_\a} \rp(x,t)-  {\partial^2 v_3\over \partial x_\a \partial x_\b}(x,t)y_3\quad (\a,\beta \in \{1,2\}),
 \\& \xi_1,\xi_2\in  L^\infty(0,T; L_n^2(0,L; H_0^1(\Omega' ))).
\end{aligned}
\label{cvflexion}
\end{equation}

\end{lemma}


 \begin{proof}   By \eqref{key} and  (\ref{eusurrborne}), we have 
 
  \begin{equation}
\label{usurrborne}
\begin{aligned}
&\sup_{\e>0, \tau\in (0,T)} \int   \lp  \lb {u_{\e1} \over  r_\e}\rb^2 +\lb  {u_{\e2}\over  r_\e}\rb^2 +
\left|u_{\e3}\right|^2 
 \rp dm_\e 
 <+\infty. 
\end{aligned}
\end{equation}
 
\ni By (\ref{eusurrborne}) and (\ref{usurrborne}),   Assumption  (\ref{euborne22})  of   Lemma \ref{lemident1} is verified,  hence, up to a subsequence,  the convergences (\ref{uemetov}) take place. 
By  Lemma \ref{lemtwoscalewrtme}, (\ref{eusurrborne}) and (\ref{usurrborne}), there  exists  $\bfv\in L^\infty(0,T; L_n^2(\Omega;\RR^3))$,
 $\zeta_{01}, \zeta_{02}, \zeta_{03}  \in 
L^\infty(0,T;L^2_n(\Omega \times  I ))$,  and $ \bfXi^{b}  \in  L^\infty(0,T; $ $L_n^2(\O\times I ;\SS^3))$, such that 
 
   \begin{equation}
\begin{aligned}
 &\bfu m_\e \buildrel \star \over \rightharpoonup n \bfv\  \quad  \hbox{weakly* in } \ L^\infty(0,T; \M (\ov \Omega;\RR^3)); \qquad nv_1=nv_2=0;
\\& u_{\e3} \mdto \zeta_{03};
  \qquad  {u_{\e \a}\over r_\e}  \ 
 \mdto \   \zeta_{0\a} \quad (\a\in \{1,2\});
\qquad   \lp{ 1 \over r_\e}\bfe(\bfu_\e) \rp_{\a\b}  \ \mdto \     \bfXi^{b} .
\end{aligned}
\label{cv2}
\end{equation}

    \ni 
We establish  below that 

 \begin{equation}
\begin{aligned}
 &  n(x) \zeta_{03}(x,t,y_3)=n (x) v_3(x,t) & &  \hbox{ a.e. in } \ \OT\times I .  
\end{aligned}
\label{zetav3}
\end{equation}

\ni  Then, we   fix a matrix field $\bfPsi$ satisfying

 \begin{equation}
\label{Psi1}
\begin{aligned}
&\bfPsi \in C^\infty\lp \ov{ \Omega\times(0,T)}; \D_\sharp(I;\SS^3)\rp, \quad \Psi_{33}=0.
\end{aligned}
\end{equation}

\ni Noticing  that $x\to \bfPsi\lp x,t,{y_\e(x_3) \over r_\e }\rp $   vanishes on the complement of 
 the support of $m_\e$,  
by integration by parts we get

 \begin{equation}
\label{Eb1}
\begin{aligned}
 \int_{\Omega\times{(0,T)} }  \hskip-0,65cm  \bfe(\bfu_\e):& \bfPsi\lp x,t,{y_\e(x_3) \over r_\e }\rp   dm_\e  dt
 = 
 - 
  \int_{\Omega\times{(0,T)}} \hskip-0,65cm\bfu_\e\cdot  \bfdiv_x \bfPsi\lp x,t,{ y_\e(x_3)  \over r_\e }\rp  dm_\e  dt
\\&\hskip2,5cm-\sum_{\a=1}^2
 \int_{\Omega\times{(0,T)} }{u_{\e\a}\over  r_\e}  {\partial \Psi_{\a3}\over \partial y_3}\lp x,t,\frac{y_\e(x_3) }{r_\e }\rp dm_\e  dt. 
\end{aligned}
\end{equation}

\ni By  passing  to the limit as $\e\to 0$ in (\ref{Eb1}), taking   (\ref{eusurrborne}),   (\ref{cv2})  and (\ref{zetav3}) into account,     we infer 
 \begin{equation}
\label{vz}
\begin{aligned}
 0=- \int_{\OTI}  v_3  (\bfdiv_x \bfPsi)_3 n dxdtdy_3 
-\sum_{\a=1}^2\int_{\OTI} 
\zeta_{0\a} {\partial \Psi_{\a3}\over \partial y_3} n dxdtdy_3  . 
\end{aligned}
\end{equation}

\ni Fixing  $\a\in \{1,2\}$,  $\varphi \in C^\infty(\ov {\Omega\times(0,T)} )$,   $\psi \in  \D_\sharp (I) $, 
and selecting  in (\ref{vz}) a field   of the form  
$ 
\bfPsi(x,t,y_3):=  \varphi(x,t)\psi(y_3) ( \bfe_\alpha\otimes\bfe_3+\bfe_3\otimes\bfe_\alpha),
$ 
we get

\begin{equation}
\nonumber
\begin{aligned}
  0= - \int_{\OT }  v_3(x,t)&  {\partial \varphi \over \partial x_\a}(x,t)ndxdt  \lp \int_I \psi(y_3) dy_3  \rp
  \\&- \int_{\OT } \lp  \int_I \zeta_{0\a} (x,t, y_3)   {\partial \psi \over \partial y_3}(y_3) dy_3\rp \varphi(x,t) ndxdt.
\end{aligned}
\end{equation}

\ni Choosing  $\psi$ such that  $\lp \int_I \psi(y_3) dy_3  \rp\not=0$, 
and making   $\varphi$ vary in $C^\infty(\ov{\Omega\times(0,T)} )$, 
we deduce  that 

\begin{equation}
\label{v3H}
\begin{aligned}
v_3\in L^\infty(0,T; L^2_n(0,L;H^1_0(\Omega' ))),
\end{aligned}
\end{equation}

\ni  then, by integration by parts with respect to $x_\alpha$, infer 

\begin{equation}
\nonumber
\begin{aligned}
  0=   \int_{\OT \times I}   \hskip-1,3cm   \varphi  (x,t)\psi(y_3) \frac{\partial v_3}{\partial x_\alpha} (x,t)ndxdt    dy_3  
 - \int_{\OT\times I}   \hskip-1,3cm \zeta_{0\a} (x,t, y_3)   {\partial \psi \over \partial y_3}(y_3)  \varphi(x,t)n dxdt dy_3. 
\end{aligned}
\end{equation}



\ni We deduce,  from   the arbitrary    choice of $\varphi$ and $\psi$, that   

\begin{equation}
\nonumber
\begin{aligned}
&   \zeta_{0\a} \!  \in\! L^\infty\lp 0,T; L^2_n\!\lp \Omega; H^1\lp I\rp\!\rp\!\rp\!; \ 
 n{\partial   \zeta_{0\a} \over \partial y_3}(x,t, y_3) \!  =\!  -
n{\partial v_3\over \partial x_\a}(x,t)
\  \hbox{ in } \OT\times I ,
\end{aligned}
\end{equation}

\ni 
and then that 
 \begin{equation}
\label{ident}
\begin{aligned}
n  \zeta_{0\a}(x,t, y_3)    = n\xi_{\a}(x,t) - n{\partial v_3\over \partial x_\a}(x,t) y_3\  \ 
\ \hbox{ in } \OT\times I,
\end{aligned}
\end{equation}

\ni 
for some  suitable $\xi_\alpha\in L^\infty(0,T; L^2_n(\O ))$.   
Next, we choose  a matrix field $\bfPsi$ satisfying  (\ref{Psi1}) and  
$\Psi_{3k}=0\ \forall\ k\in \{1,2,3\}$.
By   multiplying (\ref{Eb1}) by $\frac{1}{r_\e }$, we get
%
 \begin{equation}
\nonumber
\begin{aligned}
\sum_{\a,\b=1}^2 \int_{\OT } \! \!{ e_{\a\b}(\bfu_\e)\over  r_\e } \Psi_{\a\b}&\lp x,t,\frac{y_\e(x_3)}{r_\e}\rp dm_\e dt=
\\& 
 \!-\!\sum_{\a,\b=1}^2  \int_{\OT } { u_{\e\a}\over  r_\e}   
\frac{\Psi_{\a\b}}{\partial x_\b} \lp x,t,\frac{y_\e(x_3)}{r_\e}\rp dm_\e  dt . 
\end{aligned}
\end{equation}

\ni
By passing to the limit  as $\e\to0$,  thanks to (\ref{cv2})  and (\ref{ident}), we find  

 \begin{equation}
\nonumber
\begin{aligned}
&\sum_{\a,\b=1}^2 \int_{\OT\times I}   \bfXi^b_{\a\b}
\Psi_{\a\b} ndxdtdy  
\\&\quad = -  \sum_{\a,\b=1}^2\int_{\OT\times I} \lp\xi_{\a}(x,t) - {\partial v_3\over \partial x_\a}(x,t) y_3\rp  \frac{\partial \Psi_{\a\b}}{\partial x_\b}(x,t, y_3)   ndxdtdy_3.
\end{aligned}
\end{equation}

\ni By the arbitrary choice of the functions $\Psi_{\a\b}(=\Psi_{\b\a})$ in $C^\infty\lp\ov{\Omega\times(0,T)}; \D (I)\rp$ and  (\ref{v3H}), we deduce that for $\alpha,\beta\in \{1,2\}$, the following holds

 \begin{equation}
\nonumber
\begin{aligned}
& \xi_\a \  \in
L^\infty(0,T;L^2_n(0,L;H^1_0 (\Omega'  ))), 
\quad  v_3   \in
L^\infty(0,T;L^2_n(0,L;H^2_0 (\Omega'  ))),
\\
 &n\Xi_{\a\b}^b(x,t, y_3) =
   {1\over 2} n \lp {\partial \xi_\a\over \partial x_\b}+ {\partial \xi_\b\over \partial x_\a} \rp(x,t)- n {\partial^2 v_3\over \partial x_\a \partial x_\b}(x,t)y_3 
  \hbox{ 
in }  \ 
\OT\times I.\end{aligned}
\end{equation}

\ni  
The proof of Lemma \ref{lemident2} is achieved.


\ni{\bf Proof of (\ref{zetav3}).} We set  
$\widetilde v_{\e3}(x ,t):= \sum_{j\in J_\e} \bfu_\e(x_1,x_2,\omega_\e^{j} ,t) \mathds{1}_{\lp \omega_\e^{j}-\frac{r_\e}{2}, \omega_\e^{j} +\frac{r_\e}{2}\rp }(x_3)$ (see (\ref{defBe})). By (\ref{eusurrborne}), we  have 

 \begin{equation}
\label{estimtildve1}
\begin{aligned}
&\int   \! |u_{\e3}  \!- \!\widetilde v_{\e3}|^2  (\tau)  dm_\e
 \! = \!{\e\over r_\e} \!\sum_{j\in J_\e} \! \!\int_{\Omega'}  \! \! dx' \! \! \int_{\omega_\e^{j}-{r_\e\over2}}^{\omega_\e^{j}+{r_\e\over2}}  \!\lb u_{\e3}(x,\tau) \!- \!u_{\e3}(x',\omega_\e^{j},\tau)  \rb^2  \! \!dx_3  
\\&
\quad\le  \!  \e r_\e   \!\sum_{j\in J_\e}  \!\int_{\Omega'}  \!  \!  \! dx'    \!  \! \int_{\omega_\e^{j}-{r_\e\over2}}^{\omega_\e^{j}+{r_\e\over2}}   \!\lb{\partial u_{\e3}\over \partial x_3}\rb^2   \!  \!(x',x_3,\tau) dx_3     \!
 \le  \! C r_\e^2  \!  \! \int   \!|\bfe(\bfu_\e)|^2(\tau) dm_\e  \! \le   \!C r_\e^4.
\end{aligned}
\end{equation}

\ni We easily deduce from (\ref{cv2}) and (\ref{estimtildve1}) that 

 \begin{equation}
\begin{aligned}
 &\widetilde v_{\e3} m_\e \buildrel \star \over \rightharpoonup n v_3 \  \quad  \hbox{weakly* in } \ L^\infty(0,T; \M (\ov \Omega));
\qquad \widetilde v_{\e3} \mdto \zeta_{03}.
\end{aligned}
\label{cv3}
\end{equation}

 \ni   Fixing  $\psi \in \D(\OT\times I)$, we set (see (\ref{yepsnonper}))
 
  \begin{equation}
\label{tildepsi}
\begin{aligned}
&
\widetilde \psi_{\e}\lp x ,t,  \frac{y_{\e}(x_3)}{r_\e}\rp:= \sum_{j\in J_\e}  \psi\lp x',\omega_\e^{j} ,t,  \frac{y_{\e}(x_3)}{r_\e}\rp
 \mathds{1}_{\lp \omega_\e^{j}-\frac{r_\e}{2}, \omega_\e^{j} +\frac{r_\e}{2}\rp }(x_3).
\end{aligned}
\end{equation}

\ni We have 

\begin{equation}
\label{esttildpsi}
\begin{aligned}
&
\lb \widetilde \psi_{\e}\lp x ,t,  \frac{y_{\e}(x_3)}{r_\e}\rp-\psi \lp x ,t,  \frac{y_{\e}(x_3)}{r_\e}\rp\rb\le C r_\e \quad \hbox{ in } \ B_\e. 
\end{aligned}
\end{equation}

\ni By making the  change of variables  $y=\frac{ x_3-\omega_\e^{j}}{r_\e}$, we get  
$
 \int_{\omega_\e^{j}-\frac{r_\e}{2}}^{ \omega_\e^{j} +\frac{r_\e}{2}}    
\psi\lp x',\omega_\e^{j} ,t,  \frac{y_{\e}(x_3)}{r_\e}\rp dx_3 
=r_\e  \int_I 
\psi\lp x',\omega_\e^{j} ,t, y_3\rp dy_3. 
$
 We infer 

 \begin{equation}
\begin{aligned}
&\hskip-0,5cm \int_{\OT} \widetilde  v_{\e3} \widetilde \psi_\e\lp x,t,\frac{y_{\e}(x_3)}{r_\e}\rp  dm_\e  dt
\\&
 =\frac{\e}{r_\e}   \sum_{j\in J_\e} \int_{\Omega'\times (0,T)} \hskip-0,6cm dx'dt 
 \int_{\omega_\e^{j}-\frac{r_\e}{2}}^{ \omega_\e^{j} +\frac{r_\e}{2}}    u_{\e 3}(x', \omega_\e^{j}, t)
\psi\lp x',\omega_\e^{j} ,t,  \frac{y_{\e}(x_3)}{r_\e}\rp dx_3
\\&
 =\frac{\e}{r_\e}   \sum_{j\in J_\e} \int_{\Omega'\times (0,T)} \hskip-0,6cm
 r_\e    u_{\e 3}(x', \omega_\e^{j}, t)\lp \int_I 
\psi\lp x',\omega_\e^{j} ,t, y_3\rp dy_3\rp  dx'dt 
\\&
 =\frac{\e}{r_\e}   \sum_{j\in J_\e} \int_{\Omega'\times (0,T)} \hskip-0,6cm dx'dt
 \int_{\omega_\e^{j}-\frac{r_\e}{2}}^{ \omega_\e^{j} +\frac{r_\e}{2}}    u_{\e 3}(x', \omega_\e^{j}, t)\lp \int_I 
\psi\lp x',\omega_\e^{j} ,t, y_3\rp dy_3\rp dx_3  
\\& = \int_{\OT} \widetilde  v_{\e3}(x,t) \widetriangle\psi_\e(x,t) dm_\e dt, 
\end{aligned}
\label{u031}
\end{equation}

\ni where 
$ \widetriangle\psi_\e(x,t) := \sum_{j\in J_\e}   \lp \int_I 
\psi\lp x',\omega_\e^{j} ,t, y_3\rp dy_3\rp   \mathds{1}_{\lp \omega_\e^{j}-\frac{r_\e}{2}, \omega_\e^{j} +\frac{r_\e}{2}\rp }(x_3).
$
Noticing that 
$ \lb   \widetriangle\psi_\e(x,t)-  \lp \int_I 
\psi\lp x ,t, y_3\rp dy_3\rp \rb \le C r_\e$  in $B_\e$, we deduce successively from (\ref{cv3}),   (\ref{u031}), (\ref{esttildpsi}), and again (\ref{cv3}) that 


\begin{equation}
\begin{aligned}
  &\int_{\OTI}  \hskip-1,3cm v_3(x,t)  \lp \int_I \hskip-0,1cm
\psi\lp x ,t, y_3\rp dy_3\hskip-0,1cm\rp\hskip-0,1cmn dx dt 
\! =\! \lime  \int_{\OT} \hskip-0,9 cm\widetilde  v_{\e3}(x,t) \hskip-0,1cm \lp \int_I 
\!\psi\lp x ,t, y_3\rp dy_3\rp \!dm_\e dt
\\ &= \lime  \int_{\OT} \hskip-0,8cm\widetilde  v_{\e3}(x,t) \widetriangle\psi_\e(x,t) dm_\e dt
 = \lime\int_{\OT} \hskip- 0,8cm\widetilde  v_{\e3} (x,t)\widetilde \psi_\e\lp x,t,\frac{y_{\e}(x_3)}{r_\e}\rp  dm_\e  dt
\\&= \lime\int_{\OT}  \hskip- 0,95cm\widetilde  v_{\e3} (x,t)  \psi \lp x,t,\frac{y_{\e}(x_3)}{r_\e}\rp  dm_\e  dt
 =  \int_{\OT\times I} \hskip- 1,3cm \zeta_{03} (x,t,y_3) \psi(x,t,y_3) n dxdtdy_3.
\end{aligned}
\label{u033}
\end{equation}

\ni By the arbitrary choice of $\psi$,  Assertion (\ref{zetav3}) is proved. 
 \end{proof}

  
  \ni 
The next Lemma is specific to  the periodic case.
Given a sequence  $(\bfu_\e)$ satisfying (\ref{euborne22}) and (\ref{hypu0v}) (and possibly  \eqref{eusurrborne}), we 
establish   some    relations satisfied by   its two-scale   limit $\bfu_0$   and  by the field $\bfv$ introduced in Lemma \ref{lemident1}.
   
 
 \ni  
 
\begin{lemma}\label{lemidentu0v} {\color{black}  Assume
that   $B_\e$ is the $\e$-periodic set  defined}  by (\ref{defBeper})   and let $(\bfu_\e)$ be a sequence in $L^\infty(0,T; H^1_0(\Omega;\RR^3))$  satisfying (\ref{euborne22}) and 
\begin{equation} 
\begin{aligned}
&\sup_{\tau\in [0,T], \ \e>0}  \int_\Omega  |\bfu_\e|^2+\e^2 |\bfe(\bfu_\e)|^2(\tau)dx  <+\infty.\end{aligned}   
\label{hypu0v}
\end{equation}

 \ni  Then, up to a subsequence,   the convergences  (\ref{uemetov})  take place with $n=1$. Moreover
 
 \begin{equation} 
\begin{aligned}
& \bfu_\e\,  \dto \, \bfu_0  \hbox{ and } \  \e\bfe(\bfu_\e ) \dto \,  \bfe_y(\bfu_0)   \hbox{ in accordance with (\ref{defdto})}, 
  \\& \bfu_0 \in L^\infty(0,T; L^2(\O; H^1_\sharp(Y; \RR^3))) .
\end{aligned}   
\label{uedtou0}
\end{equation}

\ni If, in addition, 

 \begin{equation} 
\begin{aligned}
&
 \sup_{\tau\in [0,T], \ \e>0} \int_\O   \re\lb {\partial \bfu_\e\over \partial t} \rb^2  
 ( \tau)
dx  
<+\infty,
 \end{aligned}   
\label{dtuborne}
\end{equation}

\ni then 

\begin{equation}
\begin{aligned}
  & {\bfu_0}  \in W^{1, \infty}(0,T; L^2(\OY;\RR^3))  ,
  \\
 &   {\partial \bfu_\e \over \partial t} \ \dto \ {\partial \bfu_0 \over \partial t}, \qquad  \bfu_\e(\tau ) \ \dto\ \bfu_0(\tau) \
\forall \ \tau \in [0,T].
\label{u,t-u0,t}
\end{aligned}
\end{equation}

 \ni Moreover, 
 
 \ni (i) {\color{black} If $\vartheta>0$,}
 then 
 
 \begin{equation} 
\begin{aligned}
&\bfu_0(x,t,y)= \bfv(x,t) \quad \hbox{ in } \ \Omega\times (0,T)\times B.
\end{aligned}   
\label{u0=vinB}
\end{equation}

 \ni (ii)   {\color{black} If $\vartheta=0$,}
  then

\begin{equation} 
\begin{aligned}
&  {\bfu'_0}(x,t,y)=\bfv' (x,t)    \hbox{   in }    \OT \times   \Sigma
 \quad \hbox{ and } \quad 
v_3(.) = \int_{\Sigma} u_{03}(.,y) d\H^2(y).
 \end{aligned}   
\label{u0=vonSigma}
\end{equation}

\ni If, in addition, the estimate \eqref{eusurrborne} is satisfied,
%
%
  then

 \begin{equation} 
\begin{aligned}
&\bfu_0(x,t,y)= \bfv(x,t) \quad \hbox{ on } \ \Omega\times (0,T)\times \Sigma.
\end{aligned}   
\label{u0=vonSigma2}
\end{equation}

\end{lemma}

\begin{proof} The convergences  (\ref{uemetov})   are deduced from  (\ref{euborne22})  in  Lemma \ref{lemident1}.
  Under   (\ref{hypu0v}), by Lemma \ref{lemtwoscale} (ii), the sequence $(\bfu_\e)$ (resp. $(\e\bfe(\bfu_\e))$)
two-scale converges, up to a subsequence, to   some $ \bfu_0 \in L^\infty(0,T; L^2(\OY ;\RR^3))$ (resp. 
 $\bfXi^m \in L^\infty(0,T; L^2(\OY ;\SS^3))$). Choosing   $\bfPsi\in \D(\OT; C^\infty_\sharp(Y;\SS^3))$ and 
  passing to the limit as $\e\to 0$   in the equation

\begin{equation}
\begin{aligned}
&\int_\OT \e    \bfe(\bfu_\e):    \bfPsi\lp x,t, \xe\rp dx dt
= 
 \\
& \quad - \e \int_\OT \bfu_\e\cdot   \bfdiv_x\bfPsi\lp x,t,\xe\rp dxdt
-\int_\OT \bfu_\e\cdot   \bfdiv_y\bfPsi\lp x,t,\xe\rp dxdt,
\end{aligned}
\label{IPee}
\end{equation}

\ni 
we infer $\int_\OTY \bfXi^m\!:\!\bfPsi dxdtdy= -\int_\OTY \bfu_0\cdot   \bfdiv_y \bfPsi dxdtdy$ and  deduce,  by the arbitrary choice of $\bfPsi$, that  
$ \bfu_0 \in L^\infty(0,T; L^2(\O; H^1_\sharp(Y; \RR^3)))$ and $ \bfe_y(\bfu_0)= \bfXi^m$.
 Assertion (\ref{uedtou0}) is proved. 
Under (\ref{dtuborne}),  the convergences       (\ref{u,t-u0,t}) are a straightforward consequence of  Lemma \ref{lemtwoscale} (ii).


If $\vartheta>0$, by {\color{black} \eqref{euborne22} } and  (\ref{hypu0v}), 
 the sequence 
$(\e\bfe(\bfu_\e)\mathds{1}_{B_\e})$ strongly converges   to $0$ in $L^2(\Omega$ $\times(0,T);\RR^3)$.
On the other hand, by \eqref{1Besdto}, \eqref{uedtou0} and Lemma \ref{lemtwoscale}, $(\e\bfe(\bfu_\e)\mathds{1}_{B_\e})$ two-scale converges to $\bfXi^m \mathds{1}_B$.
We deduce that $\bfXi^m=0$ in $\Omega\times(0,T)\times B$.
Let us fix   $\bfPsi \in \D(\OT; \D_\sharp(B;\SS^3))$.
Then for $\e$ small enough, the support of $\bfPsi\lp x,t,\xe\rp$ is included in $B_\e\times(0,T)$. 
 Passing to the limit as $\e\to 0$   in 
  (\ref{IPee}),  we find
$
0= -\int_\OTB \bfu_0\cdot   \bfdiv_y \bfPsi dxdtdy 
$
and infer  from the arbitrariness of $\bfPsi$  that $  \bfe_y(\bfu_0)=0$   in $\OTB$. Hence, for a. e. $(x,t)\in \OT$, the restriction of  $\bfu_0(x,t,.)$  to $B$ is  a rigid displacement. 
Since $\bfu_0$ is  $Y$-periodic, we  deduce that 

\begin{equation}
\bfu_0=\bfa    \ \hbox{ in } \OTB,
\label{rigidab1}
\end{equation}

\ni  for some    $\bfa \!\in \!L^\infty(0,T; L^2(\O;\RR^3)) $.
By  \eqref{1Besdto},  \eqref{uedtou0}  and    Lemma \ref{lemtwoscale} (i),
 the sequence $\lp \bfu_\e  \mathds{1}_{B_\e} \rp $ two-scale converges to 
$\bfu_0(x,t,y)\mathds{1}_B(y)$. 
Fixing  $\bfvarphi\!\in \!\D(\OT;$ $\RR^3)$, taking
 (\ref{uemetov}), (\ref{hypu0v}) and  (\ref{rigidab1})    into account, and   noticing  that
 $\frac{\e}{r_\e}\to\frac{1}{|B|}$,
 we deduce

\begin{equation}
\begin{aligned}
\int_\OT \bfv\cdot  \bfvarphi dxdt&= \lime \int_\OT \bfu_\e\cdot  \bfvarphi dm_\e dt= \lime \frac{\e}{r_\e}\int_\OT \bfu_\e\cdot  \bfvarphi \mathds{1}_{B_\e}(x) dx dt
 \\
&  = {1\over |B|}\int_\OTB \bfu_0\cdot  \bfvarphi(x,t) 1_B(y) dxdtdy
= \int_\OT \bfa\cdot  \bfvarphi dxdt,
\end{aligned}
\nonumber
\end{equation}

\ni and infer, from   the arbitrary choice of $\bfvarphi$,  that $\bfv=\bfa$.    
Assertion   (\ref{u0=vinB}) is proved.

 
Let us assume now  that $\vartheta=0$ (i.e. that $r_\e\ll\e$).  Since the stiff layers are periodicaly distributed, by  (\ref{defBeper}) 
the field   $\ov\bfv_\e$ defined by  (\ref{defovve}) takes the form 

 \begin{equation}
\label{ovve1}
\begin{aligned}
&
\ov\bfv_\e(x ,t):= \sum_{i\in   Z_\e} \bfu_\e(x_1,x_2,\e i ,t) \mathds{1}_{\lp \e i-\frac{\e}{2}, \e i +\frac{\e}{2}\rc }(x_3),
\end{aligned}
\end{equation}

\ni and      coincides  in $B_\e$ with the field  $\hat\bfv_\e$  given by   (\ref{defhatve}).
Therefore, by  (\ref{estimhatve}),

 \begin{equation}
\label{estimovve1}
\begin{aligned}
\int  |\ov\bfv_\e-\bfu_\e|^2  (\tau)dm_\e  
\le  C \frac{r_\e}{\e}  \int_\Omega \e^2 |\bfe(\bfu_\e)|^2(\tau)dx\quad \forall \tau\in [0,T].
\end{aligned}
\end{equation}

\ni   
Since $r_\e\ll\e$, we deduce from (\ref{hypu0v}) and (\ref{estimovve1})  that 
$
\int  |\ov\bfv_\e |^2 (\tau) dm_\e
\le C$.
 On the other hand, taking  (\ref{defme}), (\ref{defBeper})  and (\ref{ovve1}) into account, it is easy to check that 
$
\int  |\ov\bfv_\e |^2 (\tau) dm_\e= \int_\Omega |\ov\bfv_\e |^2(\tau) dx$, therefore the sequence $(\ov\bfv_\e)$ 
is bounded in $L^\infty(0,T; L^2(\Omega;\RR^3))$. It then follows from  Lemma \ref{lemtwoscale} (ii) that

\begin{equation}
\label{ovvedtoovv0}
\begin{aligned}
\ov \bfv_\e \dto \ov \bfv_0, 
\end{aligned}
\end{equation}

\ni   up to  a subsequence, for  some $\ov\bfv_0\in L^\infty(0,T; L^2(\Omega\times Y ;\RR^3))$. We establish  below that 

\begin{equation}
\label{u0=v01}
\begin{aligned}
& \frac{\partial \ov \bfv_0}{\partial y_3}=0  \  \hbox{ a.e. in } \Omega\times (0,T)\times Y ,
\qquad 
\bfu_0=\ov \bfv_0 \ \hbox{ on } \  \Omega\times (0,T)\times \Sigma,
\end{aligned}
\end{equation}

\ni and that (see (\ref{exprim}))

\begin{equation}
\label{v0=v1}
\begin{aligned}
&   \bfv'(x,t)=\ov \bfv'_0(x,t,y) \quad \hbox{ a.e. in } \Omega\times (0,T)\times Y, \quad 
\\& v_3(x,t)= \int_{\lp-\frac{1}{2}, \frac{1}{2}\rp^2} \ov v_{03}(x,t,s_1,s_2, y_3) ds_1ds_2 ,
\end{aligned}
\end{equation}

\ni yielding   (\ref{u0=vonSigma}). 
The  next equation (proved  below)

 \begin{equation}
\begin{aligned}
& \frac{\partial \ov v_{03}}{\partial y_\alpha} =0  \quad \forall \a\in \{1,2\},
 \quad \hbox{ if } \  (\ref{eusurrborne}) \hbox{ holds true}, 
\end{aligned}
\label{v03=v3}
\end{equation}

\ni   joined with  (\ref{u0=vonSigma}), yields (\ref{u0=vonSigma2}).
It remains to prove \eqref{u0=v01}, \eqref{v0=v1} and  \eqref{v03=v3}.


\ni {\bf Proof of (\ref{u0=v01}).} 
  Let us  fix $\bfpsi\in \D(\OT; \D_\sharp (Y;\RR^3))$. By (\ref{ovve1}) we have

\begin{equation}
\nonumber
\begin{aligned}
&\int_{\OT }  \ov \bfv_\e \cdot  \frac{\partial \bfpsi}{\partial y_3}\lp x,t,\xe\rp dxdt 
\\&\  = \sum_{i\in Z_\e} \int_{\Omega'\times (0,T)}  \hskip-0,3cm \bfu_\e(x_1,x_2, \e i,t)\cdot  
\lp \int_{\lp \e i-\frac{\e}{2}, \e i+\frac{\e}{2}\rp} \frac{\partial \bfpsi}{\partial y_3}\lp x,t,\xpe, \frac{x_3}{\e}\rp dx_3 \rp dx'dt.
\end{aligned}
\end{equation}

\ni  Since $\bfpsi(x,t,.)\in \D_\sharp(Y;\RR^3)$, the following  holds 

\begin{equation}
\nonumber
\begin{aligned}
&  \int_{\lp \e i-\frac{\e}{2}, \e i+\frac{\e}{2}\rp} \frac{\partial \bfpsi}{\partial y_3}\lp x,t,\xpe, \frac{x_3}{\e}\rp dx_3  
= \frac{1}{\e}\int_{\lp-\frac{1}{2},\frac{1}{2}\rp} \frac{\partial \bfpsi}{\partial y_3}\lp x,t,\xpe, y_3\rp dy_3=0,
\end{aligned}
\end{equation}

\ni therefore 
$ \int_{\OT }  \ov \bfv_\e\cdot   \frac{\partial \bfpsi}{\partial y_3}\lp x,t,\xe\rp dxdt 
= 0$. By passing to the limit as $\e\to 0$, we infer $ \int_{\OTY }  \ov \bfv_0\cdot   \frac{\partial \bfpsi}{\partial y_3}\lp x,t,y\rp dxdt dy
= 0,$  and deduce from the arbitrary choice of   $\bfpsi $ that  
  
\begin{equation}
\label{v0,3=0}
\begin{aligned}
\frac{\partial \ov v_0}{\partial y_3}=0 \quad \hbox{ in } \ \ \OT\times Y.
\end{aligned}
\end{equation}

\ni We set $Y^+:= \lp-{1\over 2}, {1\over2}\rp^2 \times\lp 0, {1\over2}\rp$ and fix  $ 
\bfPsi \in \D(\OT; \D_\sharp(Y;\SS^3))$.
 Then, for each $i\in Z_\e$ (defined by  (\ref{defZe})), the field 
$\bfPsi\lp x, t, \xe\rp$ vanishes on  $\partial\lp\Omega'\times\lp \e i-\frac{\e}{2}, \e i +\frac{\e}{2} \rp \rp\times(0,T) $, and, for $\e$ small enough, the support of $\bfPsi\lp x, t, \xe\rp$ is included in $\bigcup_{i\in Z_\e}  \Omega'\times\lc \e i-\frac{\e}{2}, \e i +\frac{\e}{2}   \rc\times(0,T) $. Hence, 
 by  integration by parts,    we get 
 
%
 \begin{equation}
\label{ipp}
\begin{aligned}
& \int_\OT\hskip-0,9cm \e  \bfe(\bfu_\e): \bfPsi\lp x, t, \xe\rp  \mathds{1}_{Y^+_\sharp}\lp\xe\rp dxdt 
 =  \sum_{i\in Z_\e}  \int_{\Omega'\times\lp \e i, \e i +\frac{\e}{2}\rp \times(0,T) }\hskip-1,9cm \e\bfe(\bfu_\e): \bfPsi\lp x, t, \xe\rp  dxdt 
\\& = -\sum_{i\in Z_\e} \e \int_{\Omega'\times \{\e i\}\times(0,T)}\!\!\!\!\!\!\!\! \bfu_\e\cdot   \bfPsi\lp x, t, \xpe,0\rp \bfe_3 d\H^2(x)dt 
\\&\ \  - \int_\OT \!\!\! \!\!\!\e  \bfu_\e \bfdiv_x\bfPsi\lp x, t, \xe\rp \mathds{1}_{Y_\sharp^+}\lp\xe\rp + \bfu_\e \bfdiv_y\bfPsi\lp x, t, \xe\rp \mathds{1}_{Y_\sharp^+}\lp\xe\rp dxdt.
\end{aligned}
\end{equation}

\ni We set
$
\ov\bfPsi_\e\lp x, t, \xpe,0\rp:= \sum_{i\in   Z_\e}\bfPsi\lp x_1,x_2,\e i, t, \xpe,0\rp \mathds{1}_{\lp \e i-\frac{\e}{2}, \e i +\frac{\e}{2}\rp }(x_3).
$ Notice that 

 \begin{equation}
\label{ovPsiPsi}
\begin{aligned}
&\lb \ov\bfPsi_\e\lp x, t, \xpe,0\rp-\bfPsi \lp x, t, \xpe,0\rp \rb \le C\e,
\\&\lb \frac{\partial}{\partial y_\a}\ov \bfPsi_\e\lp x,t,\xpe,0\rp-\frac{\partial}{\partial y_\a} \bfPsi \lp x,t,\xpe,0\rp \rb \le C\e \quad \hbox{ for }\ \a\in \{1,2\}. 
\end{aligned}
\end{equation}

  \ni By 
  the definitions      $\ov\bfPsi_\e$  and $\ov\bfv_\e$   (see   (\ref{ovve1})),  there holds 
 \begin{equation}
\label{uvsv}
\begin{aligned}
-\!\sum_{i\in Z_\e}\! \e\! \int_{\Omega'\times \{\e i\}\times(0,T)}\hskip-1,7cm \bfu_\e\cdot   \bfPsi( x, t, \xpe,0) \bfe_3 d\H^2(x)&dt \!
 = \!-\!\!\int_\OT \hskip-0,9cm \ov\bfv_\e \cdot  \ov\bfPsi_\e\lp x, t, \xpe,0\rp\bfe_3 dxdt.
\end{aligned}
\end{equation}

\ni  
Taking  (\ref{ovvedtoovv0})
   and   (\ref{ovPsiPsi}) into account, and noticing that by (\ref{v0,3=0})  there holds $\ov\bfv_0(x,t,y)= \ov\bfv_0(x,t,y',0)$ in $\OT\times Y$, we obtain  

 \begin{equation}
\label{barvlim}
\begin{aligned}
\lime-\int_\OT \hskip-0,9cm \ov\bfv_\e \cdot    \ov\bfPsi_\e\lp x, t, \xpe,0\rp\bfe_3 dxdt
  &=-\int_{\OT \times Y} \hskip-0,9cm\ov\bfv_0 \cdot    \bfPsi \lp x, t, y',0\rp\bfe_3 dxdt dy 
 \\&= -\int_{\OT \times\Sigma} \ov\bfv_0\cdot    \bfPsi  \bfe_3 dxdt d\H^2(y).
\end{aligned}
\end{equation}

\ni
By passing to the limit as $\e\to 0$ in (\ref{ipp}), 
  applying  Lemma \ref{lemtwoscale}   (i)  with  $h_\e= \mathds{1}_{Y^+}\lp\xe\rp$ and taking  (\ref{uedtou0}),  (\ref{uvsv}) and  (\ref{barvlim}) into account,
 we get
 
 \begin{equation}
\nonumber
\begin{aligned}
\int_{\OT\times Y^+} \hskip-1,3cm\bfe_y(\bfu_0): \bfPsi dxdtdy= -&\int_{\OT\times \Sigma}\hskip-1,3cm \ov\bfv_0\cdot   \bfPsi\bfe_3 dxdtd\H^2(y)
 -\int_{\OT\times Y^+}\hskip-1,3cm\bfu_0\cdot  \bfdiv_y\bfPsi dxdtdy.
\end{aligned}
\end{equation}

\ni By integration by parts,  we have

 \begin{equation}
\nonumber
\begin{aligned}
- \int_{\OT\times Y^+}\hskip-1,4cm  \bfu_0\cdot  \bfdiv_y\bfPsi dxdtdy=  & \int_{\OT\times\Sigma} \hskip-1,2cm \bfu_0\cdot   \bfPsi\bfe_3 dxdtd\H^2(y)
 +\int_{\OT\times Y^+}\hskip-1,4cm \bfe_y(\bfu_0):\bfPsi dxdtdy.
\end{aligned}
\end{equation}

\ni  Joining the last two  equations,  we infer  that 
$ \int_{\OT\times \Sigma} \bfu_0\cdot   \bfPsi\bfe_3 dxdtd\H^2(y)= $

\ni  $\int_{\OT\times \Sigma} \ov\bfv_0\cdot   \bfPsi\bfe_3 dxdtd\H^2(y).
$
By the arbitrary choice of $\bfPsi$  
(and by (\ref{v0,3=0})),  
we deduce that  (\ref{u0=v01}) holds.


\ni{\bf Proof of (\ref{v0=v1}).}  Let us fix $\bfpsi\in \D(\OT;\RR^3)$ and   set 

 \begin{equation}
\label{ovpsi}
\begin{aligned}
&
\ov\bfpsi_\e(x ,t):= \sum_{i\in   Z_\e} \bfpsi (x_1,x_2,\e i ,t) \mathds{1}_{\lp \e i-\frac{\e}{2}, \e i +\frac{\e}{2}\rc }(x_3).
\end{aligned}
\end{equation}

\ni By 
(\ref{defme}) and 
(\ref{ovve1})  
we  have  

 \begin{equation}
\label{v0ri20}
\begin{aligned}
&   \int_{\OT}  \ov \bfv_\e\cdot  \ov\bfpsi_\e dm_\e dt
 = 
\frac{\e}{r_\e} \sum_{i\in Z_\e} \int_{\Omega'\times \lp \e i-\frac{r_\e}{2}, \e i +\frac{r_\e}{2}\rp\times(0,T)}
\hskip-1,5cm \bfu_\e(x',\e i ,t)\cdot  \bfpsi(x',\e i ,t) dxdt
\\&\quad= 
 \sum_{i\in Z_\e} \int_{\Omega'\times \lp \e i-\frac{ \e}{2}, \e i +\frac{ \e}{2}\rp\times(0,T)}
\hskip-1cm \bfu_\e(x',\e i ,t)\cdot  \bfpsi(x',\e i ,t) dxdt
 =  \int_\OT  \ov \bfv_\e\cdot   \ov\bfpsi_\e dx dt.
\end{aligned}
\end{equation}

\ni We infer from  (\ref{ovvedtoovv0})   and from  the estimate 

\begin{equation}
|\bfpsi-\ov\bfpsi_\e|_{L^\infty({\OT};\RR^3)}\le C\e,
\label{psiovpsi}
\end{equation}

\ni  that 

 \begin{equation}
\label{limovveovpsie}
\begin{aligned}
&\lime
 \int_\OT  \ov \bfv_\e\cdot   \ov\bfpsi_\e dx dt=\int_{\OT} \lp\int_Y  \ov\bfv_0(x,t,y)dy \rp  \bfpsi (x,t)dxdt.
 \end{aligned}
\end{equation}

\ni By (\ref{estimovve1}) and (\ref{psiovpsi}), the following  holds 

 \begin{equation}
\label{uvp}
\begin{aligned}
\lime \lb  \int_{\OT}   \bfu_\e\cdot   \bfpsi  dm_\e dt-
 \int_{\OT} \ov \bfv_\e\cdot  \ov\bfpsi_\e dm_\e dt\rb=0.
\end{aligned}
\end{equation}

\ni The weak* convergence of $(\bfu_\e m_\e)$ to $\bfv$ and 
 (\ref{v0ri20}),  (\ref{limovveovpsie}),   (\ref{uvp}), imply 

\begin{equation}
\nonumber
\begin{aligned}
  \int_\OT    \hskip-0,5cm  \bfv \cdot   \bfpsi dx dt  & = \lime \int_{\OT}  \hskip-0,5cm\bfu_\e\cdot  \bfpsi dm_\e dt 
 =  \int_{\OT}  \lp\int_Y  \ov\bfv_0(x,t,y)dy \rp \cdot  \bfpsi (x,t)dxdt  ,
\end{aligned}
\end{equation}

\ni yielding, by  the arbitrary choice of $\bfpsi$,  

\begin{equation}
\label{v=intovv0}
\begin{aligned}
 &\bfv(x,t)= \int_Y  \ov\bfv_0(x,t,y)dy\quad \hbox{ in } \ \ \OT .
\end{aligned}
\end{equation}

\ni By (\ref{v0,3=0}) and (\ref{v=intovv0}), the proof of (\ref{v0=v1}) is achieved provided  that we establish that 

\begin{equation}
\begin{aligned}
& \frac{\partial \ov v_{0\a}}{\partial y_\b} =0  \quad \forall \a, \b \in \{1,2\}.\end{aligned}
\label{v0a,b=0}
\end{equation}

\ni To that aim, let us   fix   $\bfPsi \in \D\lp\OT; C^\infty_\sharp\lp\lp-\frac{1}{2},\frac{1}{2}\rp^2;\SS^3\rp\rp$.
Since $\bfu_\e$ vanishes on $ \partial \Omega\times (0,T)$, by integrating by parts with respect to $x_1$ and $x_2$, we get 
 (see (\ref{exprim}))
 \begin{equation}
\nonumber
\begin{aligned}
\int_\OT \bfe_{x'}( \bfu'_\e): \bfPsi\lp x,t,\xpe\rp dm_\e dt=&
- \int_\OT  \bfu_\e'\cdot   \bfdiv'_{x'} \bfPsi\lp x,t,\xpe\rp dm_\e dt
\\&-  {1\over \e}\int_\OT   \bfu_\e'\cdot   \bfdiv'_{y'} \bfPsi\lp x,t,\xpe\rp dm_\e dt.
\end{aligned}
\end{equation}

\ni 
By (\ref{hypu0v}),  the term of the  left hand side  and the first term of the right hand side of the above equation are bounded, therefore
%
 \begin{equation}
\label{v0ri}
\begin{aligned}
\lime \int_\OT   \bfu_\e'\cdot   \bfdiv'_{y'} \bfPsi\lp x,t,\xpe\rp dm_\e dt
=0. 
\end{aligned}
\end{equation}

\ni  On the other hand, by  (\ref{estimovve1}) and   (\ref{ovPsiPsi}), there holds 

 \begin{equation}
\label{v0ri1}
\begin{aligned}
 \lime \lb\int_\OT   \hskip-0,8cm \bfu_\e'\cdot   \bfdiv'_{y'} \bfPsi\lp x,t,\xpe\rp  dm_\e dt
- \int_\OT   \hskip-0,8cm \ov \bfv_\e'\cdot   \bfdiv'_{y'} \ov \bfPsi_\e\lp x,t,\xpe\rp dm_\e dt\rb=0
. 
\end{aligned}
\end{equation}
 
\ni  A computation analogous to  (\ref{v0ri20}) yields
 \begin{equation}
\label{v0ri2}
\begin{aligned}
& \int_\OT  \hskip-0,4cm\ov \bfv_\e'\cdot   \bfdiv'_{y'} \ov \bfPsi_\e\lp x,t,\xpe\rp dm_\e dt 
=  \int_\OT \hskip-0,4cm \ov \bfv_\e'\cdot   \bfdiv'_{y'} \ov \bfPsi_\e\lp x,t,\xpe\rp dx dt.
\end{aligned}
\end{equation}

\ni By (\ref{ovPsiPsi}) and by the two-scale convergence of $\ov \bfv_\e$ to $\ov \bfv_0$ (see \eqref{ovvedtoovv0}), there holds

 \begin{equation}
\label{v0ri3}
\begin{aligned}
\lime \int_\OT   \hskip-0,4cm \ov \bfv_\e'\cdot   \bfdiv'_{y'}  \ov \bfPsi_\e\lp x,t,\xpe\rp dx dt
 &=  \int_\OTY  \hskip-0,8cm \ov \bfv_0'\cdot   \bfdiv'_{y'}  \bfPsi\lp x,t,y'\rp dx dt dy. 
\end{aligned}
\end{equation}

\ni Joining (\ref{v0ri}),  (\ref{v0ri1}),  (\ref{v0ri2}), and  (\ref{v0ri3}), we  get

 \begin{equation}
\label{v0ri4}
\begin{aligned}
  \int_\OTY    \ov \bfv_0'\cdot   \bfdiv'_{y'}  \bfPsi\lp x,t,y'\rp dx dt dy=0,
\end{aligned}
\end{equation}

\ni 
hence $\bfe_{y'}(\ov \bfv'_0)=0$, in the sense of distributions. We deduce that  $y'\to\ov  \bfv_0'(x,t,y',y_3)$  is a rigid displacement. By integrating (\ref{v0ri4}) by parts, we infer
 $$
\int_{\OT \times \partial \lp-\frac{1}{2}, \frac{1}{2}\rp^2} \ov \bfv_0'(x,t,y)  \cdot     \bfPsi \lp x, t, y' \rp\bfnu  dxdt d\H^1(y')=0,
$$
and infer  from the arbitrary choice of $\bfPsi \in \D\lp\OT; C^\infty_\sharp \lp\lp-\frac{1}{2}, \frac{1}{2}\rp^2;\SS^3\rp\rp$,  that $\ov \bfv_0' \in L^2\lp\OT, H^1_\sharp\lp\lp-\frac{1}{2}, \frac{1}{2}\rp^2;\RR^3\rp\rp$. The periodicity of $\ov \bfv'_0$  with respect to $y'$ 
and  the fact   that $y'\to\ov  \bfv_0'(x,t,y',y_3)$  is a rigid displacement imply  that $y'\to \ov\bfv_0'(x,t,y')$
is a constant field.  Assertion 
(\ref{v0a,b=0}) is proved. The proof of (\ref{v0=v1}) is achieved.


\ni {\bf Proof of (\ref{v03=v3}).} \ We assume   (\ref{eusurrborne}),   fix $ \psi\in \D\lp\OT; \D_\sharp\lp \lp -\frac{1}{2}, \frac{1}{2}\rp^2\rp\rp$, $\eta \in \D(I)$,
and  $\a\in \{1,2\}$. Noticing that the mapping $x\to  \psi\lp x, t,\frac{x'}{\e}\rp \eta\lp \frac{ y_\e(x_3)}{r_\e}\rp $ is compactly supported in $B_\e$, by integration by parts we obtain

\begin{equation}
\begin{aligned}
& \int_{\OT} \lp\frac{\partial u_{\e\a }}{\partial x_3}+\frac{\partial u_{\e3}}{\partial x_\a }\rp \psi\lp x, t,\frac{x'}{\e}\rp \eta\lp \frac{ y_\e(x_3)}{r_\e}\rp dm_\e dt
\\&  = - \int_{\OT} \lp  u_{\e\a } \frac{\partial \psi}{\partial x_3} \lp x, t,\frac{x'}{\e}\rp  + u_{\e3} \frac{\partial \psi}{\partial x_\a } \lp x,t, \frac{x'}{\e}\rp
 \rp\eta\lp \frac{ y_\e(x_3)}{r_\e}\rp dm_\e dt
\\& \hskip3cm - \int_{\OT}  \frac{u_{\e\a }}{r_\e}   \psi\lp x, t,\frac{x'}{\e}\rp \frac{\partial  \eta}{\partial x_3}\lp \frac{ y_\e(x_3)}{r_\e}\rp dm_\e dt
\\& \hskip3,5cm  -\frac{1}{\e} \int_{\OT}   u_{\e3}  \frac{\partial   \psi}{\partial y_\a }\lp x,t, \frac{x'}{\e}\rp  \eta\lp \frac{ y_\e(x_3)}{r_\e}\rp dm_\e dt.
 \end{aligned}
\label{ip03}
\end{equation}

\ni By (\ref{key}) we have

\begin{equation}
\begin{aligned}
& \int 
 \lb\frac{\partial u_{\e\a }}{\partial x_3}+\frac{\partial u_{\e3}}{\partial x_\a }\rb^2  + 
\lb  \frac{u_{\e\a }}{r_\e}  \rb^2 + | u_{\e3}|^2 (\tau)dm_\e
 \le C \int \frac{1}{r_\e^2} |\bfe(\bfu_\e)|^2 (\tau) dm_\e,
 \end{aligned}
\nonumber
\end{equation}

\ni hence, by  (\ref{eusurrborne}), all terms of the three first lines of (\ref{ip03}) are bounded. We infer 

\begin{equation}
\begin{aligned}
&\lime \int_{\OT}   u_{\e3}  \frac{\partial   \psi}{\partial y_\a }\lp x, t, \frac{x'}{\e}\rp  \eta\lp \frac{ y_\e(x_3)}{r_\e}\rp dm_\e dt=0,
 \end{aligned}
\nonumber
\end{equation}

\ni and then deduce from (\ref{estimovve1})  and  from an estimate analogous to (\ref{ovPsiPsi}) that

\begin{equation}
\begin{aligned}
&\lime \int_{\OT}   \ov v_{\e3}  \frac{\partial   \ov \psi_\e}{\partial y_\a }\lp x, t,\frac{x'}{\e}\rp  \eta\lp \frac{ y_\e(x_3)}{r_\e}\rp dm_\e dt=0,
 \end{aligned}
\label{j1}
\end{equation}

\ni where $\ov \psi_\e$ is defined by (\ref{ovpsi}).
 Taking 
 (\ref{ovve1}) and  (\ref{ovpsi}) into account, and noticing that 
$\frac{1}{r_\e}\int_{\lp \e i-\frac{r_\e}{2}, \e i +\frac{r_\e}{2}\rp}\eta\lp \frac{ y_\e(x_3)}{r_\e}\rp dx_3=\int_I \eta(y)dy$, we get

\begin{equation}
\begin{aligned}
&  \int_{\OT}   \ov v_{\e3}  \frac{\partial   \ov \psi_\e}{\partial y_\a }\lp x, t,\frac{x'}{\e}\rp  \eta\lp \frac{ y_\e(x_3)}{r_\e}\rp dm_\e dt
\\&=\frac{\e}{r_\e} \!\sum_{i\in Z_\e}\! \int_{\Omega'\times \lp \e i-\frac{r_\e}{2}, \e i +\frac{r_\e}{2}\rp\times(0,T)}\hskip-1cm 
 u_{\e3}(x',\e i ,t) \frac{\partial \psi}{\partial y_\a }\!\lp x',\e i ,t,\xpe\rp\! \eta\lp\! \frac{ y_\e(x_3)}{r_\e}\rp \!dx'dx_3dt  
 \\&=\e  \sum_{i\in Z_\e} \int_{\Omega'  \times(0,T)}\hskip-0cm 
 u_{\e3}(x',\e i ,t) \frac{\partial \psi}{\partial y_\a }\lp x',\e i ,t,\xpe\rp dx' dt \lp \int_I \eta(y)dy\rp
\\&= \lp \int_{\OT}   \ov v_{\e3}  \frac{\partial   \ov \psi_\e}{\partial y_\a }\lp x, t,\frac{x'}{\e}\rp dxdt \rp
\lp\int_I \eta(y)dy\rp.
 \end{aligned}
\label{j2}
\end{equation}

\ni  
On the other hand, by (\ref{ovvedtoovv0}) and an estimate analogous to  (\ref{ovPsiPsi}), there holds 

\begin{equation}
\begin{aligned}
& \lime \int_{\OT}   \ov v_{\e3}  \frac{\partial   \ov \psi_\e}{\partial y_\a }\lp x, t,\frac{x'}{\e}\rp dxdt 
= \int_{\OTY} \bfv_{03} \frac{\partial    \psi }{\partial y_\a } dxdtdy.
 \end{aligned}
\label{j3}
\end{equation}

 \ni Joining (\ref{j1}), (\ref{j2}), (\ref{j3}), and choosing $\eta$ such that $\int_I \eta dy_3\not=0$, we infer  that
$  \int_{\OTY} v_{03} \frac{\partial    \psi }{\partial y_\a } dxdtdy=0$. 
By the arbitraryness of $\psi$,    Assertion (\ref{v03=v3}) is proved. 
\end{proof}




  \ni {\color{black}   
  {\color{black}   
 In the next  Corollary, we derive from  Proposition  \ref{propapriori} and lemmas   \ref{lemident1},  \ref{lemident2},  \ref{lemidentu0v},
    a series of convergences  and identification relations  for  various sequences associated with  the solution  to \eqref{Pe}.}
  
  \begin{corollary}\label{corapriori} 
 Let $\bfu_\e$ be the solution to \eqref{Pe}. 

 \ni (i)  Up to a subsequence,
the convergences (\ref{uemetov})  hold and 
{\color{black}

\begin{equation}
  nv_1=nv_2=0  \quad   \hbox{  if \quad $k=+\infty$, }  \qquad 
   n\bfv=0   \quad \hbox{  if \quad
 $\kappa=+\infty$.
 }
\label{v1v2=0}
\end{equation}
 }

\ni Under (\ref{H1}) or (\ref{H2}), $n\bfu=n\bfv$.

\ni  (ii) If   $\kappa>0$,  
 the   relation 
 (\ref{flexion0})  and convergences
 (\ref{cvflexion})  hold.
 
 \ni (iii)
 In the periodic case, that is under (\ref{mu0lambda0}) and  (\ref{defBeper}),
  the convergences and relations  (\ref{uedtou0}), (\ref{u,t-u0,t}) hold.
If $\vartheta>0$
 (resp.   $\vartheta=0$),
  the relations (\ref{u0=vinB}) (resp. (\ref{u0=vonSigma}))
  are verified.
If  in addition $\kappa>0$, then 
  (\ref{u0=vonSigma2}) holds.

\end{corollary}

\begin{proof}
  Noticing that by   (\ref{kkappa})
   and    (\ref{apriori}),   
 the estimate (\ref{euborne22}) holds, Assertion (i) follows from Lemma \ref{lemident1} 
 ( Assertion \eqref{v1v2=0} is a consequence of  \eqref{kkappa}, \eqref{apriori},  and \eqref{v=0}).
   If   $\kappa>0$,   by     (\ref{kkappa}) and (\ref{apriori})
    the estimate (\ref{eusurrborne}) holds, and Assertion (ii) follows from 
Lemma \ref{lemident2}.
 In the periodic case, by   (\ref{mu0lambda0}),  (\ref{defBeper}),
and  (\ref{apriori}),  $\bfu_\e$  satisfies  
(\ref{hypu0v}) and (\ref{dtuborne}), hence (iii) results  from   Lemma \ref{lemidentu0v}.
\end{proof}
 }

 
   \section{Proof of theorems \ref{thstiff}, \ref{thinter},  \ref{th}}\label{secprooftheoremth}  \quad  
      {\color{black}
    In the spirit of Tartar \cite{Ta}, we will multiply \eqref{Pe} by an appropriate  test field $\bfphi_\e$, integrate by parts,     and,  passing  to the limit as $\e\to 0$ by means of 
 the   convergences derived  in Corollary \ref{corapriori}, obtain a variational problem equivalent to 
  the announced limit problem, and also to 
 \eqref{Pvar2} for some  suitable  $H,V,a, h, \xi_0, \xi_1$. Theorem \ref{thDautrayLions} will yield existence, uniqueness, and regularity of the effective displacement. Uniqueness implies  that    the convergences   obtained  in Corollary \ref{corapriori}   for   subsequences     hold  for the  complete   sequences.


  } 
   
      \subsection{Proof of Theorem \ref{th}}\label{secproofth}  \quad  
We set
  
\begin{equation}
\begin{aligned}
&
\lpt \begin{aligned}
 &H:=\la   
 \begin{aligned} &(\bfw_0, \bfpsi )\in   L^2(\O\times Y;\RR^3) \times   L^2(\O ;\RR^3)  , \ 
 \\& \bfw_0 = \bfpsi     \  \hbox{ in }   \O \times {B} \end{aligned}
\ra ,
\\ &
((\bfw_0, \bfpsi ), (\widetilde\bfw_0, \widetilde\bfpsi  ))_{ H} := \int_{\OY}  \rho \bfw_0\cdot \widetilde\bfw_0 dxdy,  
\end{aligned} \ra& & \hbox{if } \ \vartheta>0,
\\&
\lpt \begin{aligned}
&H\!:=   L^2(\OY; \RR^3) \!\times  L^2(\Omega ;\RR^3),
\\&
((\bfw_0, \bfpsi ), (\widetilde\bfw_0, \widetilde\bfpsi ) )_{ H} := \int_{\OY}   \rho    \bfw_0\cdot \widetilde\bfw_0 dxdy + \int_\Omega \ru \bfpsi\cdot \widetilde\bfpsi dx
,
\end{aligned} \ra& & \hbox{if } \ \vartheta=0.
\end{aligned}
\label{H}
\end{equation}

\ni   We easily  deduce from the positiveness of  $\rho$ and $\ov\rho_1$ (see (\ref{defrhoe}))  
 that   $H$ is a Hilbert space.  
We fix  a couple $(\bfw_0   ,
 \bfpsi   )\in  L^2(0,T;H) $ 
satisfying    (see (\ref{defA}))

   \begin{align}
  	& \bfw_0  \in  C^\infty ([0, T];\D(\O;   C^\infty_\sharp (Y;\RR^3))),  \quad \bfpsi \in  C^\infty ([0, T];\D(\O;    \RR^3)),
\label{W0s}
	 \\ &   \bfw_0(T)={\partial \bfw_0\over \partial t}(T)= \bfpsi(T)={\partial \bfpsi\over \partial t}(T) =0,
  \label{suphyp}
  \\&
\begin{aligned} &\psi_1=\psi_2= 0 & & \hbox{if } \quad k=+\infty,
\qquad  \bfpsi=0 & &  \hbox{if } \quad \kappa=+\infty,
\end{aligned}
\label{f0}
\\& \bfw_0  (x,t,y)=  \bfpsi (x,t) \quad    \hbox{ in } \OT\times A.\label{w0=psi0}
\end{align}

%
%
%

  \ni  We choose  a sequence $(\alpha_\e)$  of positive reals  such that 
\begin{equation}
 \e r_\e <\!\!< \alpha_\e<\!\!<1, 
\label{alphapetit}
\end{equation}
and  set

\begin{equation}  
{C_\e} := \la x\in \Omega,   \ {\rm dist} (x,B_\e) <\alpha_\e r_\e  \ra. 
\label{defCe}
\end{equation}

\ni It is usefull to notice that 

\begin{equation}  
 \L^3({C_\e}\setminus {B_\e})\le C\frac{\a_\e r_\e}{\e},
 \label{LCe}
\end{equation}

\ni and that, by (\ref{w0=psi0}), the following estimate holds for  $m\in \{1,2\}$:

\begin{equation}
\begin{aligned}
 &\lb \bfpsi(x,\!t)\!- \!\bfw_0\! \lp\! x,t,\!\xe\rp\!\rb  \!+\! \lb \frac{\partial^m\bfpsi}{\partial t^m}(x,\!t)\!- \!\frac{\partial^m\bfw_0}{\partial t^m}  \lp\! x,\!t,\!\xe\rp\!\rb  \!
  \le \!C  \frac{\alpha_\e r_\e}{\e}  \  
   \hbox{ in } C_\e\!\times \!(0,\!T).
  \end{aligned} 
\label{estimpsiw}
\end{equation}

 \ni  By \eqref{defCe}, we can fix  a sequence    $(\eta_\e)$      in   $C^\infty(\ov \Omega)$    satisfying 

\begin{equation}
0\le \eta_\e \le 1, \qquad
 \eta_\e=1\ \ \hbox{ in } \ \ {B_\e} ,  \qquad \eta_\e = 0 \ \hbox{ in } \ \ \Omega\setminus {C_\e}, \qquad |\bfnabla \eta_\e |<{C\over
r_\e \alpha_\e}. 
\label{defeta}
\end{equation}

\ni   The sequence of test fields  $(\bfphi_\e)$ mentioned above  will    be defined  by

\begin{equation}
  \bfphi_{\e} (x,t) :=
   \eta_\e(x) \,\wideparen\bfpsi_\e\lp x,t \rp +  (1- \eta_\e(x) )\bfw_0 \lp x,t,\xe\rp,
\label {defphie} 
\end{equation}

\ni  
where  $\wideparen\bfpsi_\e$ is  described   in Section  \ref{appendix}.  
 As  $\bfphi_{\e}(x,t) = \bfw_0 \lp x,t,\xe\rp$
in $\Omega\setminus C_\e \times(0,T)$,   we deduce from (\ref{w0=psi0}), (\ref{estimpsiw}), (\ref{defeta}),  (\ref{defphie}), and (\ref{estimpsi}) that  the following estimates hold  in 
$ \Omega\times (0,T)$ for $m\in \{1,2\}$:
 
\begin{equation}
\begin{aligned}
  \lb  \bfphi_{\e}(x,t) \!- \! \bfw_0 \lp x,t,\xe\rp\rb\!  + \! \lb  \frac{\partial^m\bfphi_{\e}}{\partial t^m}(x,t) \!- \! \frac{\partial^m\bfw_0}{\partial t^m} \lp x,t,\xe\rp\rb\!  
  \le\! C\lp r_\e\!+ \!  \frac{\a_\e r_\e}{\e}\rp .
  \end{aligned} 
\label{estimphi}
\end{equation}

\ni It is also interesting to notice that by (\ref{defeta}), (\ref{defphie}), and (\ref{estimpsi}),  

\begin{equation}
\begin{aligned}
  \lb  \bfphi_{\e}(x,t) - \bfpsi(x,t)\rb  +   \lb  \frac{\partial^m \bfphi_{\e}}{\partial t^m}(x,t) - \frac{\partial^m  \bfpsi}{\partial t^m}(x,t)\rb      \le C r_\e  \quad \hbox{ in  } \ B_\e\times(0,T).
  \end{aligned} 
\label{estimphiBe}
\end{equation}

\ni   

\ni By  (\ref{Pe}) and  (\ref{lmu}) we have $|\bfsigma_\e( \bfphi_{\e} ) |  \le C \mu_{0\e} |\bfnabla\bfphi_\e| $ in $\Omega\setminus B_\e\times(0,T)$, therefore by (\ref{defeta}), (\ref{estimphi}), (\ref{estimpsi}), the next estimates are satisfied in  $C_\e\setminus B_\e\times(0,T)$

\begin{equation}
\begin{aligned}
 |\bfsigma_\e(\! \bfphi_{\e}\!) | &
 \le \!C \mu_{0\e} \!\lp 
  |\nabla \eta_\e|  \! \lb  \bfphi_{\e}(x,\!t)\! -\!  \bfw_0 \lp\! x,t,\!\xe\rp\!\rb  \!+ \!
  \lb \bfnabla  \wideparen\bfpsi_\e( x,\!t)\rb \! +\!  \lb \bfnabla\! \lp \!\bfw_0 \lp\! x,\!t,\!\xe\rp\! \rp\!\rb\rp
 \\& \le C \mu_{0\e} \lp 
\frac{1}{\a_\e r_\e}  \lp r_\e+   \frac{\a_\e r_\e}{\e}\rp  + 
\frac{C}{\e} \rp \le C \mu_{0\e} \lp \frac{1}{\e}+\frac{1}{\a_\e} \rp ,
  \end{aligned} 
\label{estimsigmaC-B}
\end{equation}

\ni yielding 

\begin{equation}
\begin{aligned}
 \bfsigma_\e( \bfphi_{\e})  
 & \le  C \mu_{0\e} \lp \frac{1}{\e}+\frac{1}{\a_\e} \rp 
  \hskip 0,5cm \hbox{ in } (C_\e\setminus B_\e)\times (0,T).
  \end{aligned} 
\label{estimsigmasoft}
\end{equation}

\ni 
Applying  (\ref{lem1})   to  $\chi_\e =\rho_\e \mathds{1}_{\Omega\setminus B_\e}$, $h_0\in \{
\bfw_0, {\partial^2\bfw_0 \over
\partial t^2},\bfw_0( 0 ), $ ${\partial^2\bfw_0 \over
\partial t^2} ( 0 )\}$, we deduce from  (\ref{defrhoe}), (\ref{1Besdto}), (\ref{alphapetit}), and  (\ref{estimphi}),
that
  the following   convergences hold  for $  m\in \{1,2\}$ 

\begin{equation}
\begin{aligned} &  \rho_\e  \bfphi_{\e}\mathds{1}_{\Omega\setminus B_\e} \ \sdto \   \rho \mathds{1}_{Y\setminus A}(y) \bfw_0,\quad & & \rho_\e {\partial^m  \bfphi_{\e}\over \partial t^m}  \mathds{1}_{\Omega\setminus B_\e} \ \sdto \  
\rho\mathds{1}_{Y\setminus A}(y){\partial^m \bfw_0 \over \partial t^m}, 
\\ &    \rho_\e
 \bfphi_{\e}( 0 )\mathds{1}_{\Omega\setminus B_\e}  \sdto    \rho\mathds{1}_{Y\setminus A}(y) \bfw_0 (  0  ),\  
& & \rho_\e{\partial
 \bfphi_{\e}\over
\partial t }( 0)\mathds{1}_{\Omega\setminus B_\e} \sdto  \rho\mathds{1}_{Y\setminus A}(y){\partial  \bfw_0 \over
\partial t }( 0 ).
\end{aligned}
\label{Phidtothin0}
\end{equation}

\ni 
  By  multiplying     (\ref{Pe}) by   
$ \bfphi_{\e} $,    after integrations by parts   we obtain (see (\ref{suphyp})) 
\begin{equation}
\begin{aligned}
&  \int_\OT  \re \bfu_\e\cdot  {\partial^2  \bfphi_{\e} \over \partial t^2} dxdt    + \int_\O  \re \bfa_0 \cdot {\partial
 \bfphi_{\e} \over \partial t}(0)dx-\int_\O \re \bfb_0 \cdot  \bfphi_{\e}( 0) dx
\\ & \hskip3cm +\int_\OT     \bfe(\bfu_\e):  \bfsigma_\e( \bfphi_{\e} ) dxdt 
 =\int_\OT \re \bff \cdot  \bfphi_{\e} dxdt. \label{IP}
 \end{aligned}
\end{equation}

\ni By (\ref{defme}) and (\ref{defrhoe}), we have

\begin{equation}
\begin{aligned}
   \int_\OT  \re \bfu_\e\cdot  {\partial^2  \bfphi_{\e} \over \partial t^2} dxdt  
= \int_\OT & \rho \mathds{1}_{\Omega\setminus B_\e} \bfu_\e\cdot  {\partial^2  \bfphi_{\e} \over \partial t^2} dxdt
\\&+ \int_\OT  \frac{r_\e}{\e}\rho_{1\e}  \bfu_\e\cdot  {\partial^2  \bfphi_{\e} \over \partial t^2} dm_\e(x) dt.
 \end{aligned}
\label{splitOB}
\end{equation}
 
\ni
We deduce from   (\ref{uemetov}),  
  (\ref{uedtou0})  (see Corollary \ref{corapriori}), and  (\ref{Phidtothin0})  that

\begin{equation}
\begin{aligned}
 & \lime \int_\OT  \re \mathds{1}_{\Omega\setminus B_\e} \bfu_\e\cdot  {\partial^2  \bfphi_{\e} \over \partial t^2} dxdt
 = \int_\OTY \rho \mathds{1}_{Y\setminus A}(y) \bfu_0\cdot  {\partial^2  \bfw_0 \over \partial t^2} dxdt dy. 
 \end{aligned}
 \nonumber
\end{equation}

\ni By (\ref{defrhoe}),  (\ref{uemetov}),
  and (\ref{estimphiBe}), we have 
$$
 \lime  \int_\OT  \frac{r_\e}{\e}\rho_{1\e}  \bfu_\e\cdot  {\partial^2  \bfphi_{\e} \over \partial t^2} dm_\e(x) dt
 = \int_{\Omega\times (0,T)} \ov\rho_1 \bfv \cdot  {\partial^2  \bfpsi  \over \partial t^2} dxdt.
 $$

\ni The last two equations imply

\begin{equation}
\begin{aligned}
\lime    \int_\OT  \hskip-0,9cm \re \bfu_\e\cdot  {\partial^2  \bfphi_{\e} \over \partial t^2} dxdt  
= \int_{\OT\times Y\setminus A} \hskip-0,8cm \rho  \bfu_0\cdot  &{\partial^2  \bfw_0 \over \partial t^2} dxdt dy
 + \int_{\Omega\times (0,T)} \hskip-0,6cm \ov\rho_1 \bfv \cdot  {\partial^2  \bfpsi  \over \partial t^2} dxdt.
 \end{aligned}
\label{lim0}
\end{equation}

\ni  
As, by (\ref{Pe}),   $\bfa_0$, $\bfb_0$, and $\bff$ are    continuous, 
 we obtain by the same argument

\begin{equation}
\begin{aligned}
& \lime  \int_\O  \re \bfa_0 \cdot {\partial
 \bfphi_{\e} \over \partial t}(0)dx=\int_{\Omega\times Y\setminus A} \hskip-0,5cm\rho   \bfa_0\cdot  {\partial  \bfw_0 \over \partial t} dx dy
+ \int_\Omega \ov\rho_1 \bfa_0 \cdot  {\partial^2  \bfpsi  \over \partial t^2} dxdt,
 \\& \lime  \int_\O \re \bfb_0 \cdot  \bfphi_{\e}( 0) dx =\int_{\Omega\times Y\setminus A} \rho  \bfb_0\cdot   \bfw_0   dx dy
+ \int_\Omega \ov\rho_1 \bfb_0 \cdot    \bfpsi   dxdt,
\\ & \lime \int_\OT \re \bff \cdot  \bfphi_{\e} dxdt =\int_{\OT\times Y\setminus A} \hskip-0,5cm\rho  \bff\cdot    \bfw_0   dxdt dy
+ \int_{\Omega\times (0,T)}\hskip-0,8cm \ov\rho_1 \bff \cdot  {\partial^2  \bfpsi  \over \partial t^2} dx dt
. 
 \end{aligned}
 \label{lim1}
\end{equation}


\ni We split  the $4^{th}$ term of the  left hand member of   (\ref{IP}) into the sum of three terms: 

\begin{equation}
\begin{aligned}
 &  \int_\OT   \hskip-0,8cm    \bfe(\bfu_\e) :  \bfsigma_\e( \bfphi_{\e}  ) dxdt  = I_{1\e} + I_{2\e} + I_{3\e} ; 
\quad I_{1\e}:=
\int_{\O\setminus {C_\e}\times (0,T)}  \hskip-1,2cm 
  \e\bfe(\bfu_\e):  \frac{1}{\e}\bfsigma_\e( \bfphi_{\e} ) dxdt,
\\ &   I_{2\e}:= \int_{{C_\e}\setminus {B_\e} \times (0,T)}   \hskip-1cm   \e \bfe(\bfu_\e): \frac{1}{\e} \bfsigma_\e( \bfphi_{\e} ) dxdt,\quad 
I_{3\e}:= \int_{  {B_\e} \times (0,T)}    \hskip-1cm   \bfe(\bfu_\e): \bfsigma_\e( \bfphi_{\e} ) dxdt.
\end{aligned}
\label{split}
\end{equation}  
 
\ni By    (\ref{defeta}) and  (\ref{defphie}), we have  $\bfphi_{\e} \mathds{1}_{\O\setminus {C_\e} }= \bfw_0 \lp x,t,{x \over \e} \rp \mathds{1}_{\O\setminus {C_\e} }$.
Taking  (\ref{Pe}), (\ref{lmu}),  and (\ref{mu0lambda0})  into account, we deduce that 

\begin{equation}
\label{estimsigmaO/C}
\begin{aligned}
\lb  {1\over \e}\bfsigma_\e\lp\bfphi_\e \rp \! - \bfsigma_{0y}( \bfw_0)\!\lp  x,t,{x \over \e}  \rp \rb \mathds{1}_{\O\setminus {C_\e} } 
\le C \e,
\end{aligned}
\end{equation}

\ni where the operator $\bfsigma_{0y}$ is defined by (\ref{eysigmayg}). The following convergence

\begin{equation}
\label{1Cesdto}
\begin{aligned}
\mathds{1}_{\Omega\setminus C_\e} \ \sdto \ \mathds{1}_{Y\setminus A},
\end{aligned}
\end{equation}

\ni follows from (\ref{1Besdto}) and from the strong convergence of $\mathds{1}_{C_\e\setminus B_\e}$ to $0$ in $L^2(\Omega)$, which results  from  (\ref{alphapetit}) and  (\ref{LCe}). 
By   applying Assertion  (\ref{lem1})  of Lemma \ref{lemtwoscale}   to  $h_0 :=\bfsigma_{0y}( \bfw_0)   $ and $
\chi_\e := 
\mathds{1}_{\O\setminus {C_\e} } $, taking (\ref{estimsigmaO/C}) into account, we infer

\begin{equation}
     {1\over \e}\bfsigma_\e\lp  \bfphi_{\e} \rp  \mathds{1}_{\O\setminus {C_\e} }\  \sdto \ \bfsigma_{0y}(\bfw_0 )   \mathds{1}_{ Y\setminus A }(y) .
 \label{sdto2}
\end{equation}

\ni We deduce from  (\ref{uedtou0}), (\ref{split}),   and (\ref{sdto2}) that 

\begin{equation}
\label{I1}
\begin{aligned}
\lim_{\e\to 0} I_{1\e}&
 = \int_{\OT\times Y\setminus A}
  \bfe_y(\bfu_0): \bfsigma_{0y}(\bfw_0 ) dxdtdy.
\end{aligned}
\end{equation}

\ni By    (\ref{mu0lambda0}),  (\ref{estimsigmasoft}) and (\ref{LCe}), we have 
$\int_{C_\e\setminus B_\e\times (0,T)} \lb\frac{1}{\e} \bfsigma_\e(\bfphi_\e) \rb^2 dx  dt 
\le C \lp \frac{\a_\e r_\e}{\e} + \frac{\e r_\e}{\a_\e}\rp$,
\ni therefore, by (\ref{alphapetit}), the sequence $\lp\frac{1}{\e} \bfsigma_\e(\bfphi_\e) \mathds{1}_{C_\e\setminus B_\e}\rp$ strongly converges to $0$ in $L^2(C_\e\setminus B_\e\times(0,T);\SS^3)$. 
Accordingly, we  infer  from (\ref{hypu0v}) and  (\ref{split}) that

\begin{equation}
\lim_{\e\to 0} I_{2\e}=0.
\label{I2}
\end{equation}

 \ni  Finally, the limit of the sequence $(I_{3\e})$ defined by (\ref{split}) is   computed in Lemma \ref{lemI3} in terms of  $k$ and $\kappa$.
 Passing   to the limit as $\e \to 0$ in (\ref{IP}), collecting  (\ref{f0}),  (\ref{lim0}), (\ref{lim1}),  (\ref{split}),  (\ref{I1}), (\ref{I2}), and  (\ref{I32}), 
   we  obtain the variational formulation 
 given, according to the order of magnitude of $k$ and $\kappa$, by (\ref{Elims}), (\ref{Elim22}), (\ref{Elim3}), or (\ref{Elim0}). 
 We    distinguish $4$ cases:

\ni {\bf  Case $0<k<+\infty$.} \quad    We find
   
\begin{equation}
\begin{aligned} &   \int_{\OT\times Y\setminus A}\hskip-0,5cm \rho {\bfu_0} \cdot {\partial^2 \bfw_0 \over \partial t^2}  dxdtdy    + \int_{\OT\times Y\setminus A} \hskip-0,5cm  \rho \bfa_0 \cdot {\partial \bfw_0
\over \partial t}( 0  )-\rho \bfb_0 \cdot  \bfw_0 ( 0 ) dxdy
\\& +\int_\OT     \ru  {\bfv } \cdot {\partial^2 \bfpsi \over \partial t^2}  dxdt    + \int_\O    \ru  \bfa_0 \cdot {\partial \bfpsi
\over \partial t}( 0  )- \ru  \bfb_0 \cdot  \bfpsi ( 0 ) dx  
\\&
  +\int_{\OT\times Y\setminus A}  \bfe_y({\bfu_0}):\bfsigma_{0y}(\bfw_0 ) dxdt dy  
 + {k }   \int_\OT  \! \bfe_{x'}(\bfv'): \bfsigma_{x'}(\bfpsi' )
dxdt 
 \\&=\int_{\OT\times Y\setminus A} \rho \bff \cdot \bfw_0  dxdtdy+\int_\OT   \ru  \bff \cdot \bfpsi dxdt,
\end{aligned}
  \label{Elims}
\end{equation}

\ni for all $(\bfw_0,\bfpsi)\in L^2(0,T;H)$ satisfying (\ref{W0s}), (\ref{suphyp}), (\ref{w0=psi0}). We set (see (\ref{H}))

\begin{equation}
\begin{aligned}
 &   \xi=(\bfu_0,\bfv ),\quad\xi_0= (\bfa_0,\bfa_0 ),\quad \xi_1=(\bfb_0,\bfb_0 ), \quad h= (\bff, \bff),
\\ & 
   V:= \la (\bfw_0, \bfpsi)\in H\lb\begin{aligned} &\bfw_0 \in   L^2(\O; H^1_\sharp(Y;\RR^3)) 
  \\& \psi_1, \psi_2 \in  L^2(0,L; H^1_0(\Omega' ))\end{aligned} \rpt \ra \hskip 1,3cm  \hbox{if } r_\e=e \e, 
 \\ &V:= \la  (\bfw_0, \bfpsi)\in H  \lb \  \begin{aligned}&   \bfw_0 \in   L^2(\O; H^1_\sharp(Y;\RR^3)) , 
  \\& \psi_1, \psi_2 \in  L^2(0,L; H^1_0(\Omega' ))  \\& \bfw_0' (x,y)= \bfpsi'(x) \ \hbox{ on } \ \Omega\times \Sigma
  \\&  \int_{\Sigma} w_{03}(.,y) d\H^2(y)=\psi_3
  \end{aligned}\rpt
 \ra \qquad \hbox{if } r_\e<\!\!<\e, 
\\ &    \ov a(\bfv,\bfpsi
):=  k \int_\Omega \bfe_{x'}(\bfv'): \bfsigma_{x'}(\bfpsi') dx,
\\ &  a((\bfu_0,  \bfv  ), (\bfw_0, \bfpsi )):= 
 \int_{\Omega\times Y\setminus A}   \bfe_{ y}(\bfu_0): \bfsigma_{ y}(\bfw_0)dxdy +\ov a( \bfv  ,   \bfpsi ),
\\ &   (\!(\!(\!\bfu_0,\! \bfv ),\!(\bfw_0,\! \bfpsi \!)\!)\!)_{ V} \!\!:=\!
  (\!(\!\bfu_0, \!\bfv\! )\!,\!(\!\bfw_0,\! \bfpsi \!)\!)_{ H} \!+\! \ov a(   \! \bfv , \!\bfpsi )
+ \int_{\Omega\times Y\setminus A}\hskip-0,8cm \bfnabla_y\bfu_0\!\cdot\! \bfnabla_y\bfw_0dxdy.
\end{aligned}
 \label{data1}
\end{equation}

\ni
By (\ref{n=1per}),  (\ref{uemetov}), (\ref{uedtou0}),  and  (\ref{u,t-u0,t}) we have  $\xi=(\bfu_0,\bfv )\!\in \!L^2(0,T;V) , 
 {\partial \xi\over \partial t}\! \in \!L^2(0,T;H)$, thus  by a density argument     the   variational
formulation (\ref{Elims}) is equivalent to  (\ref{Pvar2}). By (\ref{H}), (\ref{data1}),  and  the next  Korn's inequality 
 (see \cite{OlYoSh},   p. 14), 
 $$
 \int_{Y\setminus A} |\bfw|^2 +|\bfnabla(\bfw)|^2 dy\le C
 \int_{Y\setminus A} |\bfw|^2 +|\bfe(\bfw)|^2 dy\quad \forall \bfw\in  H^1(Y\setminus A;\RR^3), 
 $$
   for   all $\widetilde\xi=(\bfw_0, \bfpsi) \in V$,  the following holds
\begin{equation}
\begin{aligned}
 ||\widetilde\xi||^2_V   
  & =  |\widetilde\xi|^2_H+  \ov a(\psi,\psi)
  +  || \bfnabla_y(\bfw_0)||^2_{L^2(\Omega\times Y\setminus A;\RR^3)} \!
  \\&   \le C|\widetilde\xi|^2_H+ \!C|| \bfe_y(\bfw_0)||^2_{L^2(\Omega\times Y\setminus A;\RR^3)} \!+ \ov a(\psi,\psi)
 \le C|\widetilde\xi|^2_H+ C a(\widetilde\xi,\widetilde\xi),
 \end{aligned}
 \label{VHa}
\end{equation} 
yielding (\ref{hypC}).  We deduce from   Theorem \ref{thDautrayLions}    that  
$  \xi=(\bfu_0,\bfv ) $ is the unique solution to (\ref{Elims}). By (\ref{Pvar}), (\ref{xireg}),
(\ref{data1}),  
the following holds \begin{equation} \begin{aligned} 
   \xi \in C([0,T]; V)\cap C^1([0,T]; H),
\ 
 \xi(0)=   (\bfa_0,\bfa_0 ), \   {\partial \xi\over \partial t}(0) =  (\bfb_0, \bfb_0 ).
 \end{aligned} \label{bdu0}
\end{equation}
It follows  from (\ref{bdu0}), from  the next  inequalities (deduced from (\ref{H}), (\ref{data1}))   
\begin{equation}
\begin{aligned} 
  &  ||\bfw_0||_{L^2(\O; H^1_\sharp(Y;\rr^3))}    +||\bfpsi||_{L^2(\O;\rr^3)} + || \psi_1||_{ L^2(0,L; H^1_0(\Omega' ))}
\\ &  \hskip 1,3cm+ || \psi_2||_{ L^2(0,L; H^1_0(\Omega' ))}\le C
|| (\bfw_0,\bfpsi)||_{V}  \quad \forall \ (\bfw_0,\bfpsi) \in V,
\\ &  || \bfw_0||_{L^2(\OY; \rr^3)} +||\bfpsi||_{L^2(\O;\rr^3)}  \le C | (\bfw_0,\bfpsi )|_{H}  
 \quad \forall \ (\bfw_0,\bfpsi) \in H,
 \end{aligned}
 \label{norm}
\end{equation}  
and  the next 
 implication, holding for all  couple  $(E_1,E_2)$ of Banach  spaces 
\begin{equation}
\lpt\begin{aligned}
 &
  A\!\in\! \L(E_1,E_2)
  \\&  B \!\in\! C^k([0,T];E_1) 
  \end{aligned}\ra 
    \  \Rightarrow 
\lc   A\circ B \!\in \! C^k([0,T];E_2);\  
\  {d^s\over
dt^s}( A\circ
B)  = A\circ {d^s\over dt^s}B
\ \forall  s\le k 
 \rc,\nonumber
\end{equation}

\ni  applied  with $B=\xi=(\bfu_0, \bfv )$, $E_1\in \{H,V\}$,
$(A(\xi), E_2) \in \Big\{ \lp \bfu_0, L^2(\O; H^1_\sharp(Y;\RR^3))\rp ,$ $ \lp\bfv,  L^2(\O;\RR^3)\rp, 
\lp v_\a,  L^2(0,L;  H^1_0(\Omega' ))\rp ,$ $ \lp \bfu_0,  L^2(\OY; \RR^3)\rp,  \lp \bfv, L^2(\O )\rp \Big\}$, 
%
  that
\begin{equation}
\begin{aligned}
&    
\bfu_0 \!\in\! C([0,T];\! L^2(\O; H^1_\sharp(Y;\RR^3) ) )\!\cap\! C^1([0,T]; L^2(\OY;\RR^3)), 
\\ & \bfu_0(0)= \bfa_0, \, {\partial \bfu_0 \over \partial t}(0)=
\bfb_0,
\\ &   \bfv\in   C^1([0,T]; L^2(\O; \RR^3 )), \ \bfv(0)= \bfa_0, \ {\partial \bfv \over \partial t}(0)= \bfb_0,
\\ &    v_1, v_2\! \in \!C([0,T]; L^2(0,L; H^1_0(\Omega' )))\!\cap \!C^1([0,T]; L^2(\O )).
 \end{aligned}
 \label{bdeff}
\end{equation}
\ni Next we   prove that the variational problem (\ref{Elims}) is equivalent to (\ref{Phom1}). 
Setting  $\bfpsi=0$ in (\ref{Elims}), 
noticing that  $\bfe_y( {\bfu_0}) : \bfsigma_{0y}(\bfw_0) =\bfsigma_{0y}({\bfu_0}) : \bfnabla_{ y} (\bfw_0)$,
we   get 

\begin{equation}
\begin{aligned}
      \int_{\Omega\times(0,T)\times  Y\setminus A}& \hskip-0,7cm  \rho {\bfu_0} \cdot {\partial^2 \bfw_0  \over \partial t^2}  dxdtdy   +
\int_{\Omega \times  Y\setminus A}\hskip-0,7cm 
\rho
\bfa_0 \cdot {\partial \bfw_0\over
\partial t}( 0  )dxdy
 -\int_{\Omega \times  Y\setminus A}\hskip-0,7cm   \rho \bfb_0 \cdot  \bfw_0( 0 ) dxdy
  \\& +\int_{\OT\times Y\setminus A}\hskip-0,9cm
\bfsigma_{0y}({\bfu_0}) : \bfnabla_{ y} (\bfw_0) dxdt dy 
   =  
\int_{{\Omega\times(0,T)\times  Y\setminus A}} \hskip-0,9cm \rho \bff\cdot \bfw_0 dxdtdy,
 \end{aligned}
 \label{Elim0}
\end{equation} 
and,  letting   $\bfw_0$ vary  over  $\D(\O\! \times \! (0,T)\! \times\!  Y\setminus A;\RR^3)$,  
deduce 
\begin{equation}
\rho{\partial^2 \bfu_0\over \partial t^2}-\bfdiv_y(\bfsigma_{0y}(\bfu_0))=\rho \bff  \qquad \qquad \ \hbox{ in }\  {\Omega\times(0,T)\times  Y\setminus A}.
\label{Equ0}
\end{equation} 

\ni 
By integrating (\ref{Elim0}) by parts with respect to $(t,y)$
 for an arbitrary   $\bfw_0 $ satisfying (\ref{W0s}), (\ref{suphyp}), (\ref{w0=psi0}),   we infer from (\ref{Equ0})  that
 $
\int_{\OT\times \partial Y} \bfsigma_{0y}(\bfu_0)  \bfnu\cdot \bfw_0 dxdtd\H^2(y)$ $=0 
 $
($\bfnu:= $  outward pointing normal to $\partial Y$). 
Noticing that   $\bfsigma_{0y}(\bfu_0)\bfnu=0$ $\H^2$ a. e. on    $\partial Y \cap \ov   A $ 
(because if $\vartheta>0$, then $A=B$ and, by (\ref{u0=vinB}),  $\bfsigma_{0y}(\bfu_0)=0$ in $B$, whereas  if $r_\e<\!\!<\e$, then $A=\Sigma$ and    $\H^2(\partial Y \cap \ov \Sigma)=0$), we deduce 
\begin{equation}
\bfsigma_{0y}(\bfu_0) \bfnu(x,t,y)= -\bfsigma_{0y}(\bfu_0)\bfnu(x,t,-y)\quad   \hbox{ on } \quad \OT\times \partial Y. \qquad
\label{OTdY}
\end{equation}


\ni Fixing    $(\bfw_0, \bfpsi )\in L^2(0,T;H)$    satisfying (\ref{W0s}), (\ref{suphyp}), we infer from  the $Y$-periodicity of $\bfw_0$, {\color{black}  \eqref{w0=psi0}, 
and  (\ref{OTdY}),     that (see (\ref{eysigmayg}))
\begin{equation}
\begin{aligned}
   - \int_{\OT \times \partial (Y\setminus {A})} \hskip-1,5 cm  \bfsigma_{0y} ({\bfu_0}) \bfnu\cdot  \bfw_0  dxdtd\H^2(y)
  &= -\int_{\OT \times \partial (Y\setminus A)\cap \ov A} \hskip-1,5 cm  \bfsigma_{0y}(\bfu_0)  \bfnu_{Y\setminus A}\cdot \bfpsi   dxdt d\H^2(y)
 \\  &    =  \int_{\OT  }   \bfg( \bfu_0 )\cdot    \bfpsi  
 dxdt 
.
\end{aligned}
 \label{bord}
\end{equation} 
}

\ni 
By multiplying      (\ref{Equ0}) by    $ \bfw_0$    and by integrating it by parts over ${\Omega\times(0,T)\times  Y\setminus A}$, 
thanks to  (\ref{bdeff}),  (\ref{OTdY}), (\ref{bord}) we obtain  

\begin{equation}
\begin{aligned}
 &   \int_{\Omega\times(0,T)\times  Y\setminus A}   \hskip-1,1  cm   \rho\, {\bfu_0}\cdot  {\partial^2 \bfw_0 \over \partial t^2}  dxdtdy     +  \int_{\Omega \times  Y\setminus A}   \hskip-0,8  cm  \rho\, \bfa_0  \cdot  {\partial \bfw_0\over
\partial t}( 0  )dxdy 
  - \int_{\Omega \times  Y\setminus A} \hskip-0,7  cm   \rho \,\bfb_0   \cdot  \bfw_0( 0 ) dxdy
 \\&  \ +\int_{\Omega\times(0,T)\times  Y\setminus A}   \hskip-1,6 cm     \bfe_y({\bfu_0}):\bfsigma_{0y}(\bfw_0) dxdt dy  
 +  \int_{\OT  }    \hskip-0,8 cm \bfg( \bfu_0 )\cdot    \bfpsi 
 dxdt \!
  = \int_{\Omega\times(0,T)\times  Y\setminus A}   \hskip-1  cm \rho \bff \cdot  \bfw_0 dxdtdy.
 \end{aligned}
 \label{partiel}
\end{equation}

\ni By subtracting  (\ref{partiel}) from (\ref{Elims}),  we find

\begin{equation}
\begin{aligned}
   \int_\OT &  \hskip-0,5 cm  \ru   \bfv 
   \cdot    {\partial^2 \bfpsi\over \partial t^2} dxdt  
   - \int_{\OT}   \hskip-0,5 cm   \bfg(\bfu_0) \cdot     \bfpsi    dxdt 
  +  {k }   \int_\OT   \hskip-0,8 cm \bfe_{x'}(\bfv'): \bfsigma_{x'}(\bfpsi' )dxdt 
 \\& + \int_\O     {\overline \rho_1} \bfa_0 \cdot  {\partial
\bfpsi\over \partial t}( 0  )dx  
-\int_\O     {\overline \rho_1} \bfb_0 \cdot   \bfpsi( 0 ) dx
   =\!\int_\OT  {\overline \rho_1} \bff \cdot  \bfpsi dxdt\!
    .
  \end{aligned}
\label{sss}
\end{equation} 

\ni Making $\bfpsi$ vary in $\D(\OT;\RR^3)$, 
we infer 
\begin{equation}
\begin{aligned}
 &  
{\overline \rho_1}\! {\partial^2
\bfv
\over \partial t^2}  \! 
-k \bfdiv\bfsigma_{x'}(\bfv')=
   \!\ru  \bff +   \bfg(\bfu_0)   \quad \hbox{ in  } \ \ \OT.
  \end{aligned}
 \label{Ev}
\end{equation}
By   (\ref{bdeff}), (\ref{Equ0}), (\ref{OTdY}),  (\ref{Ev}), and Lemma \ref{lemidentu0v},  the couple  $(\bfu_0, \bfv )$ is a solution to (\ref{Phom1}), (\ref{k0}).
Conversely, any solution to (\ref{Phom1}), (\ref{k0}) satisfies (\ref{Elims}).


\ni  {\bf  Case   $k=+\infty, \ \kappa=0$.} \  \quad   
We obtain 
  
\begin{equation}
\begin{aligned}
&   \int_{\OT\times Y\setminus A} \rho {\bfu_0} \cdot  {\partial^2 \bfw_0 \over \partial t^2}  dxdtdy   
 + \int_{\Omega\times Y\setminus A} \hskip-0,6cm  \rho \bfa_0 \cdot  {\partial \bfw_0
\over \partial t}( 0  )-\rho \bfb_0 \cdot   \bfw_0 ( 0 ) dxdy
\\& +\int_\OT     \ru  {v_3 } {\partial^2 \psi_3\over \partial t^2}  dxdt    + \int_\O    \ru  a_{03}{\partial \psi_3
\over \partial t}( 0  )- \ru  b_{03} \psi_{03} ( 0 ) dx  
 \\  &   +\int_{\OT\times (Y\setminus  A) }  \hskip-1,5 cm \bfe_y({\bfu_0}) : \bfsigma_{0y}(\bfw_0  ) dxdt dy   
 =\int_{\OT\times Y\setminus A}  \hskip-1,5 cm \rho \bff \cdot  \bfw_0  dxdtdy+\int_\OT  \hskip-0,5 cm   \ru  f_3\psi_3  dxdt. 
   \end{aligned}\label{Elim22}
\end{equation}

\ni  This variational formulation is   satisfied  for all $(\bfw_0,\bfpsi )\in L^2(0,T;H)$ verifying (\ref{W0s}), (\ref{suphyp}), and (\ref{f0}). 
 We set ($H$ and $V$ being  given by  (\ref{H}), (\ref{data1}))

\begin{equation}
\begin{aligned}
 &\xi = (\bfu_0, \bfv), \quad 
H^{(2)}:= \la (\bfw_0,\bfpsi )\in H, \ \psi_1=\psi_2 =0\ra,  
  \\  &  (.,.)_{H^{(2)}}:=  (.,.)_{H },     
 \quad  V^{(2)}:= V\cap H^{(2)}, \quad 
  ((.,.))_{V^{(2)}}:=  ((.,.))_{V },
    \\  &  h^{(2)}\!:= \lp   \bff \mathds{1}_{Y\setminus A}\!
+   f_3\bfe_3  \mathds{1}_{A} ,  f_3\bfe_3
   \rp, 
 \\  & a^{(2)}((\bfu_0,  \bfv  ), (\bfw_0, \bfpsi )):= 
 \int_{\Omega \times  Y\setminus A}   \bfe_{ y}(\bfu_0): \bfsigma_{ y}(\bfw_0)dxdy ,
 \\  & \xi_0^{(2)}:=  (\bfa_0\mathds{1}_{Y\setminus A}\! +  a_{03}\bfe_3 \mathds{1}_{A},  a_{03}\bfe_3 ), \quad 
    \xi_1^{(2)}\!:= (\bfb_0\mathds{1}_{Y\setminus A}\! +  b_{03}\bfe_3 \mathds{1}_{A},  b_{03}\bfe_3 ). 
  \end{aligned}\label{data2}
\end{equation}

\ni   By  (\ref{n=1per}), 
 (\ref{u,t-u0,t}) and 
 (\ref{v1v2=0}),    we have   $\xi  \in L^2(0,T;V^{(2)})$  and $\xi'\in L^2(0,T;H^{(2)})$. 
Therefore, by a density argument,  the variational problem (\ref{Elim22}) is equivalent to  (\ref{Pvar2}).
By (\ref{VHa}), (\ref{data2}), 
 the estimate  (\ref{hypC})  is  satisfied. We  deduce  from    Theorem
\ref{thDautrayLions}  that  $\xi= (\bfu_0,\bfv
)$   is the unique solution to (\ref{Elim22}) and that 
$ \xi  \in  C ([0, T];  V^{(2)} ) \cap   C^1 ([0, T]; H^{(2)} ), \ \xi(0) =    \xi^{(2)}_0, \   {\partial \xi \over \partial t}(0)  =   \xi^{(2)}_1$.
Then, repeating the argument employed to prove (\ref{bdeff}), 
we infer from   (\ref{norm}) and  (\ref{data2})  that  the initial-boundary conditions and regularity properties stated in (\ref{Phommatrix}), (\ref{infty0}) are satisfied. Setting $\psi_3=0$ in (\ref{Elim22}), we get (\ref{Elim0})  and deduce (\ref{Equ0}), (\ref{OTdY}), (\ref{bord}), (\ref{partiel}).
Then, substracting (\ref{partiel})  from (\ref{Elim22}), taking (\ref{f0}) into account,  we find  

\begin{equation}
\begin{aligned}
 &  \int_\OT   \ru   v_3
   {\partial^2  \psi_3 \over \partial t^2} dxdt  
 - \int_{\OT}    (\bfg(\bfu_0))_3 \psi_3   dxdt 
+ \int_\O     {\overline \rho_1} a_{03} {\partial
\psi_3 \over \partial t}( 0  )dx  
\\& \hskip4,5cm
-\int_\O     {\overline \rho_1} b_{03} \psi_3( 0 ) dx
  =\!\int_\OT  {\overline \rho_1}f_3\psi_3 dxdt\!   .
  \end{aligned}
\nonumber
 \end{equation}

\ni Making  $\psi_3$ vary in $\D(\OT)$, we deduce that
${\overline \rho_1} {\partial^2v_3\over \partial t^2}     = \ru f_3+  (\bfg(\bfu_0))_3$ in  $ \OT$ and infer that  $(\bfu_0, \bfv)$ is solution to (\ref{Phom1}), (\ref{infty0}). 
    

\ni{\bf  Case   $0<\kappa<+\infty$.} 
 \quad Passing to the limit as $\e\to 0$ in (\ref{IP}),   we obtain  
  \begin{equation}
\begin{aligned}
&   \int_{\OT\times Y\setminus A} \rho {\bfu_0} \cdot  {\partial^2 \bfw_0 \over \partial t^2}  dxdtdy   
 + \int_{\O \times Y\setminus A} \hskip-0,6cm  \rho \bfa_0 \cdot  {\partial \bfw_0
\over \partial t}( 0  )-\rho \bfb_0 \cdot   \bfw_0 ( 0 ) dxdy
\\& +\int_\OT     \ru  {v_3 } {\partial^2 \psi_3\over \partial t^2}  dxdt    + \int_\O    \ru  a_{03}{\partial \psi_3
\over \partial t}( 0  )- \ru  b_{03} \psi_{03} ( 0 ) dx  
 \\  &  +\int_{\OT\times (Y\setminus  A) } \hskip-1cm \!\bfe_y({\bfu_0}) : \bfsigma_{0y}(\bfw_0  ) dxdt dy  
  +{\kappa  \over6}\!\int_{\OT } \bfH(v_3): \!
\bfH^{\bfsigma}(\psi_3 ) dxdt
 \\&=\int_{\OT\times Y\setminus A} \rho \bff \cdot \bfw_0  dxdtdy+\int_\OT   \ru  f_3\psi_3  dxdt,
   \end{aligned}
 \label{Elim3} 
\end{equation}

\ni   for all 
 $\!(\bfw_0,\!\bfpsi )\!\!\in \!L^2(0 ,\!T\!;\!H)\!$ verifying    (\ref{W0s}), (\ref{suphyp}), (\ref{f0}).
We set (see  (\ref{data1}), (\ref{data2}), \eqref{sigmab0})
\begin{equation}
\begin{aligned}
 & H^{(3)}:=H^{(2)},\quad
 \\&V^{(3)}:=  \la  (\bfw_0, \psi_3\bfe_3)\in V^{(2)}   \lb \  \begin{aligned}&    \psi_3 \in L^2(0,L; H^2_0(\Omega' ))
   \\& \bfw_0  (x,y)=  \psi_3(x)\bfe_3 \ \hbox{ on } \ \Omega\times \Sigma
     \end{aligned}\rpt
 \ra,
\\  &
  ( ( ( \bfu_0, v_3\bfe_3), (\bfw_0, \psi_3\bfe_3)  ) )_{V^{(3)}} :=  
 ( ( ( \bfu_0, v_3\bfe_3), (\bfw_0, \psi_3\bfe_3 )  ) )_{V }
 \\  & \hskip 1cm+ \int_\O \lp {\partial^2 v_3\over \partial x_1^2}{\partial^2 \psi_3\over \partial x_1^2}+{\partial^2 v_3\over \partial x_2^2}{\partial^2
\psi_3\over\partial x_2^2} +{\partial^2 v_3\over \partial x_1^2}{\partial^2
\psi_3\over\partial x_2^2}+{\partial^2 v_3\over \partial x_2^2}{\partial^2
\psi_3\over\partial x_1^2}\rp dx  ,
 \end{aligned}
 \label{data3} 
\end{equation}
\begin{equation}
\begin{aligned}
 &
\ov a^{(3)}( v_3\bfe_3,   \psi_3\bfe_3)  ):=  
{\kappa  \over6}\!\int_{\O  } \!\! \bfH(v_3): \!
\bfH^{\bfsigma}(\psi_3 ) dxdt,
\\  & a^{(3)}( ( \bfu_0, \bfv ), (\bfw_0, \bfpsi )  ):= \int_{\OYs} \bfe_y(\bfu_0):\bfsigma_y(\bfw_0)dxdy
+ \ov a^{(3)} (  \bfv ,   \bfpsi   ),
    \\  & \xi_0^{(3)}:= \xi_0^{(2)}, \quad 
  \xi_1^{(3)}\!:= \xi_1^{(2)}, 
\quad  h^{(3)}\!:= h^{(2)}.
 \end{aligned}
 \label{data3bis} 
\end{equation}

\ni
By   Corollary \ref{corapriori} (ii) 
 and assertions  (\ref{n=1per}),  (\ref{u,t-u0,t})   and (\ref{v1v2=0}),  
 there holds $\xi=(\bfu_0,\bfv )\in L^2(0,T;V^{(3)})$ and  $\xi'\in L^2(0,T;H^{(3)})$ hence, by  a density argument,  
the variational formulation (\ref{Elim3}) is equivalent to   (\ref{Pvar2}).
By    (\ref{data1}),  (\ref{VHa}),   (\ref{data2}),   (\ref{data3}), 
 (\ref{data3bis}),
   and (\ref{sigmab0}),
  for   all $\widetilde\xi=(\bfw_0, \bfpsi) \in V^{(3)}$, we have 

\begin{equation}
\begin{aligned}
||\widetilde\xi||_{V^{(3)}}^2&  \le ||\widetilde\xi||_{V }^2+ C  \ov a^{(3)}( \bfpsi ,  \bfpsi )
   \le C (|\widetilde \xi|_H+  a(\widetilde \xi, \widetilde \xi)+ \ov a^{(3)}( \bfpsi ,  \bfpsi))
 \\ &  \le C (|\widetilde \xi|_{H^{(3)}}+  a^{(3)}(\widetilde \xi, \widetilde \xi) ), 
\end{aligned}
\nonumber
\end{equation}
\ni
hence Assumption   (\ref{hypC}) is satisfied.
We deduce from  Theorem
\ref{thDautrayLions} that    $\xi\!= \! (\bfu_0,\bfv )$   is the unique solution to (\ref{Elim3}) and  that
$ 
   \xi \in C([0,T]; V^{(3)}) \cap   C^1([0,T]; H^{(3)} ) $,
  $\xi(0)=  \xi_0^{(3)}$,   ${\partial \xi \over \partial t}(0)  =   \xi_1^{(3)}$,
yielding, by   the inequality   (\ref{norm}) joined with    

$$
 ||\psi_3||_{L^2(0,L; H^2_0(\Omega' ))} \! 
 \le  \!C 
|| ( \bfw_0,\bfpsi  ) ||_{V^{(3)}},    \forall\    ( \bfw_0,\bfpsi  )\! \in \!V^{(3)},
$$

\ni the initial-boundary conditions and regularity properties stated in (\ref{Phommatrix}), (\ref{inftykappa}).
  Repeating   the   argument of the case $0<k<+\infty$, we set $\psi_3=0 $ in (\ref{Elim3}),
  obtain 
(\ref{Elim0}), deduce  (\ref{Equ0}), (\ref{OTdY}), (\ref{bord}), (\ref{partiel}), substract  (\ref{partiel}) from  (\ref{Elim3}), and get

\begin{equation}
\begin{aligned}
  & \int_\OT   \!\!\!\!\!\!\!  {\overline \rho_1} v_3  {\partial^2  \psi_3 \over \partial t^2}
     dxdt 
  +{\kappa   \over6}\!\int_{\OT } \!\!\bfH(v_3): \!
\bfH^{\bfsigma}(\psi_3 ) dxdt
\\  & \hskip0,5cm  -\int_{\OT  }    (\bfg( \bfu_0 ))_3  \bfpsi_3
 dxdt    -\! \int_\O  
  {\overline \rho_1} ( \bfb_0)_3   \psi_3(x,0 )
   dx 
 =  \!\int_\OT\! \! \! \! \! {\overline
\rho_1}f_3 \psi_3
 dxdt.
  \end{aligned}
  \label{bbb}
  \end{equation}

\ni By (\ref{sigmab0}),  the following equation  holds in the sense of distributions in $\D'(\OT)$

\begin{equation}
\begin{aligned}
&{\kappa   \over6}\!\left \langle   \bfH(v_3): \!
\bfH^{\bfsigma}(\psi_3 )\right \rangle_{\D' , \D }
 = \frac{\kappa  }{3}\frac{l+1}{l+2}\left \langle  
  \sum_{\a,\b=1}^2\!\! {\partial^4  v_3 \over \partial x_\a^2\partial x_\b^2} 
 , \psi_3 \right  \rangle_{\D' , \D }.
 \end{aligned}
\nonumber
  \end{equation}

\ni
Making  $\psi_3$ vary in $\D(\OT)$ in (\ref{bbb}),  we infer

\begin{equation}
\begin{aligned} 
& {\overline \rho_1} {\partial^2
 v_3
\over \partial t^2}   +  \frac{\kappa  }{3}\frac{l+1}{l+2}
\sum_{\a,\b=1}^2\!\! {\partial^4  v_3 \over \partial x_\a^2\partial x_\b^2} 
 =  \ru f_3+   (\bfg(\bfu_0))_3, \ \hbox{  in }\ \OT,  
 \end{aligned}
\nonumber
  \end{equation}

\ni and deduce that  $(\bfu_0,\bfv)$
satisfies (\ref{Phom1}), (\ref{inftykappa}).  

\ni{\it Case  $\kappa=+\infty$.} By  (\ref{I32}) 
 we have   $I_{3\e}=0$.
By passing to the limit as $\e\to 0$ in (\ref{IP}),  we  obtain  (\ref{Elim0}) and, taking (\ref{v1v2=0}) into account,  deduce in a similar manner   that  $(\bfu_0,\bfv )$
satisfies  (\ref{Phom1}), (\ref{inftyinfty}). The proof of Theorem \ref{th}  is achieved.



\subsection{  Proofs  of theorems \ref{thstiff} and \ref{thinter}}  
Under the assumptions of Theorem \ref{thstiff}, by (\ref{mulambda}) and 
(\ref{apriori}), 
the sequence $(\bfu_\e)$ (resp.  $\lp \frac{\partial \bfu_\e}{\partial t}\rp$) is bounded in 
$L^\infty(0,T; H^1_0(\Omega;\RR^3))$ (resp. $L^\infty(0,T; L^2(\Omega;$ $\RR^3))$), therefore 
by the Aubin-Lions-Simon lemma (see \cite[Corollary 6]{Si}),  
$(\bfu_\e)$  strongly converges in $L^2(0,T; L^2(\Omega;\RR^3))$ and weakly* converges in $L^\infty(0,T; H^1_0(\Omega;\RR^3))$, up to a subsequence, to some  $\bfu\in L^\infty(0,T; H^1_0(\Omega;\RR^3))$. 
In particular, assumption (\ref{H2}) of Lemma \ref{lemident1} is satisfied,  hence 
 $n\bfu=n\bfv$.
 
 Under the assumptions of Theorem \ref{thinter}, by   the apriori estimates  (\ref{apriori}),
  the sequence $(\bfu_\e)$ is bounded in 
$L^\infty(0,T; L^2(\Omega;\RR^3))$, hence   weakly* converges in $L^\infty(0,T; $ $ L^2(\Omega;\RR^3))$, up to a subsequence, to some  $\bfu\in L^\infty(0,T; L^2(\Omega;\RR^3))$. 
By (\ref{mulambdaintermediaire}) and (\ref{netonstrong}), Assumption (\ref{H1}) of Lemma \ref{lemident1} is satisfied, thus we also get   $n\bfu=n\bfv$.   
  Applying  Corollary  \ref{corapriori}, we  deduce in both cases  from  (\ref{uemetov}), \eqref{flexion0},      \eqref{v1v2=0}, and \eqref{I32},  that  
 
 \begin{equation}
\begin{aligned}
 &   
\bfu_\e m_\e  \buildrel \star \over \rightharpoonup  n\bfu   
\hskip 3cm   & &\hbox{weakly* in } \ L^\infty(0,T; \M (\ov\Omega;\RR^3) ),
\\&u_1, u_2 \in L^\infty(0,T; L_n^2(0,L; H^1_0(\Omega'))),
 \\& \bfe_{x'}(\bfu_\e')m_\e  \buildrel \star \over \rightharpoonup 
n \bfe_{x'} (\bfu')
\hskip 1cm & & \hbox{weakly* in } \ L^\infty(0,T; \M (\ov\Omega; \SS^3) ),
\\& 
  nu_1=nu_2=0,
    \quad   & &\hbox{if \quad $k=+\infty$, } 
\\
 &u_3\in L^\infty(0,T; L_n^2(0,L; H_0^2(\Omega' ))) \qquad& & \hbox{if } \quad \kappa>0,
\\&  
   n\bfu=0   \quad & &\hbox{if \quad
 $\kappa=+\infty$,}
\\& \I_{n,k,\kappa}(\bfv, \bfpsi)  = 
 \I_{n,k,\kappa}(\bfu, \bfpsi). 
 \end{aligned}
\label{totstiff}
\end{equation} 
 

\ni Let us check  that 

\begin{equation}
\begin{aligned}
\mathds{1}_{\Omega\setminus B_\e} \buildrel\star \over \rightharpoonup 1-\vartheta n \ \hbox{ weakly* in } \ L^\infty(\O),
\end{aligned} 
\label{1Omega-BetonA}
\end{equation}

\ni where $\vartheta$ is defined by \eqref{defvartheta}. 
 If   $\vartheta =0$,  
  (\ref{1Omega-BetonA})   follows from the fact that $|B_\e|\to 0$.
Otherwise, if   $\vartheta>0$,   
then 
 the sequence $(\frac{\e}{r_\e}\mathds{1}_{B_\e})$ is bounded in 
$L^\infty(\Omega)$ and, by (\ref{meton}), weakly* converges in $L^\infty(\Omega)$
to $n$. It then follows from \eqref{defvartheta} that 
 $\lp   \mathds{1}_{B_\e}\rp$ 
  weakly* converges  in $L^\infty(\O)$  to $\vartheta  n$,  yielding 
(\ref{1Omega-BetonA}).
  Next, we  check that

\begin{equation}
\bfu_\e \mathds{1}_{\Omega\setminus B_\e} \rightharpoonup \bfu (1-\vartheta n) \quad \hbox{weakly in } \ L^2(\Omega\times (0,T);\RR^3). 
\label{cvuO-B}
\end{equation}

\ni If 
  $\vartheta =0$,
    Assertion  (\ref{cvuO-B}) follows from the weak convergence of $(\bfu_\e)$ 
to $\bfu$ in $L^2(\Omega\times(0,T))$ and the   convergence of $\L^3(B_\e)$ to $0$. Otherwise, $\vartheta>0$, 
then $\lp \frac{\e}{r_\e} \bfu_\e \mathds{1}_{B_\e} \rp$ is bounded in $L^2(\Omega\times(0,T))$,
and weakly converges, by \eqref{totstiff}, to $n\bfu$.  Hence, by \eqref{defvartheta}, 
 $\lp  \bfu_\e \mathds{1}_{B_\e} \rp$ weakly converges to $ n \vartheta \bfu$, yielding (\ref{cvuO-B}).

\ni{\color{black} We fix a field $\bfpsi$ verifying \eqref{W0s}, \eqref{suphyp}, and 

  \begin{equation}
\begin{aligned} &n\psi_1=n\psi_2= 0\quad  \hbox{if } \quad k=+\infty; 
 \quad n \bfpsi=0 \quad  \hbox{if } \quad \kappa=+\infty.
\end{aligned}
\label{f0n}
\end{equation}

\ni   The} sequence of  test fields $(\bfphi_\e)$ defined by substituting $\bfpsi$ for $\bfw_0$ in (\ref{defphie}), {\color{black} that is}

 \begin{equation}
  \bfphi_{\e} (x,t) :=
   \eta_\e(x) \,\wideparen\bfpsi_\e\lp x,t \rp +  (1- \eta_\e(x) )\bfpsi \lp x,t \rp,
\label{defphiestiff} 
\end{equation}

\ni  where $\wideparen\bfpsi_\e\lp x,t \rp$ is described in Section \ref{appendix}, and $\eta_\e$ satisfies  (\ref{defeta}), now  with respect to  the non-periodic sets $B_\e$, $C_\e$  given  by (\ref{defBe}), (\ref{defCe}). We    assume  that (see Remark \ref{remI22})

\begin{equation}
\begin{aligned}
&\frac{r_\e}{\e} <\!\!< \alpha_\e<\!\!<1 & &\hbox{under the assumptions of Theorem \ref{thstiff}}, 
\\& \mu_{0\e}<\!\!<\alpha_\e<\!\!<1 & &\hbox{under the assumptions of Theorem \ref{thinter}}. 
\label{alpha2}
\end{aligned}\end{equation}

\ni By  (\ref{defeta}),  (\ref{defphiestiff}), and (\ref{estimpsi}),   the following estimates hold  in 
$ \Omega\times (0,T)$ for $m\in \{1,2\}$:
 
\begin{equation}
\begin{aligned}
  \lb  \bfphi_{\e}(x,t) -  \bfpsi \lp x,t,\xe\rp\rb  +  \lb  \frac{\partial^m\bfphi_{\e}}{\partial t^m}(x,t) -  \frac{\partial^m\bfpsi}{\partial t^m} \lp x,t,\xe\rp\rb  
  \le C r_\e.
  \end{aligned} 
\label{estimphistiff}
\end{equation}

\ni We deduce from  (\ref{LCe}) and from the estimate  
$  \bfsigma_\e( \bfphi_{\e}(x,t) )   
  \le C \frac{\mu_{0\e} }{\a_\e}
  \hbox{ in } C_\e\setminus B_\e\times (0,T),
 $
   obtained in a similar manner as  (\ref{estimsigmaC-B}), that 

\begin{equation}
\begin{aligned}
\int_{C_\e\setminus B_\e\times(0,T)}  \lb \bfsigma_\e( \bfphi_{\e}(x,t) ) \rb^2 dxdt \le    C \frac{\mu_{\e0}^2 r_\e}{\alpha_\e \e}.
 \end{aligned} 
\label{sigphiC2}
\end{equation}

\ni  

 \ni We multiply Equation  (\ref{Pe}) by $\bfphi_\e$ and integrate it  by parts to get 
 (see (\ref{IP}),  \eqref{split}) 
 
 \begin{equation}
\begin{aligned}
&  \int_\OT  \re \bfu_\e\cdot  {\partial^2  \bfphi_{\e} \over \partial t^2} dxdt    + \int_\O  \re \bfa_0 \cdot {\partial
 \bfphi_{\e} \over \partial t}(0)dx-\int_\O \re \bfb_0 \cdot  \bfphi_{\e}( 0) dx
\\ & \hskip3cm +I_{1\e} + I_{2\e} + I_{3\e} 
 =\int_\OT \re \bff \cdot  \bfphi_{\e} dxdt. \label{IPstiffinter}
 \end{aligned}
\end{equation}

\ni  By the same argument as the one used  to get \eqref{lim0}, \eqref{lim1}, splitting each term as in \eqref{splitOB} and 
taking into account  (\ref{defrhoe}), (\ref{meton}), (\ref{1Omega-BetonA}), \eqref{cvuO-B}, (\ref{estimphistiff}), 
    we obtain 

\begin{equation}
\begin{aligned}
&\lime    \int_\OT  \hskip-0,5cm \re \bfu_\e\cdot  {\partial^2  \bfphi_{\e} \over \partial t^2} dxdt  
= \int_\OT \hskip-1cm  \rho (1-\vartheta n)   \bfu\cdot   {\partial^2  \bfpsi \over \partial t^2} dxdt  
 + \int_{\Omega\times (0,T)} \hskip-1cm \ov\rho_1 \bfu\cdot  {\partial^2  \bfpsi  \over \partial t^2} n  dxdt ,
 \\& \lime  \int_\O \hskip-0,2cm \re \bfa_0 \cdot {\partial
 \bfphi_{\e} \over \partial t}(0)dx=\int_{\Omega }  \hskip-0,2cm\rho  (1-\vartheta n)    \bfa_0\cdot  {\partial  \bfpsi \over \partial t} (0)dx  
+ \hskip-0,1cm \int_\Omega \hskip-0,2cm \ov\rho_1 \bfa_0 \cdot  {\partial^2  \bfpsi  \over \partial t^2} (0)ndxdt,
 \\& \lime  \int_\O \re \bfb_0 \cdot  \bfphi_{\e}(0) dx =\int_{\Omega }  \rho  (1-\vartheta n)  
 \bfb_0\cdot   \bfpsi (0)   dx 
+ \int_\Omega \ov\rho_1 \bfb_0 \cdot    \bfpsi  (0) n dxdt,
\\ & \lime \int_\OT  \hskip-0,5cm\re \bff \cdot  \bfphi_{\e} dxdt =\int_{\OT} \hskip-0,5cm\rho  (1-\vartheta n)   \bff\cdot    \bfpsi    dxdt  
+ \int_{\Omega\times (0,T)}\hskip-0,5 cm \ov\rho_1 \bff \cdot  {\partial^2  \bfpsi  \over \partial t^2} n dx dt
. 
 \end{aligned}
 \label{lim12}
\end{equation}

\ni 
Under the assumptions of Theorem \ref{thstiff}, by  (\ref{mulambda}),  (\ref{defsigma}) and (\ref{defphiestiff})  we have $\bfsigma_\e(\bfphi_\e)  = \bfsigma(\bfpsi)$ in $\Omega\setminus C_\e\times (0,T)$, and 
by (\ref{repetit}) and (\ref{defCe}),   $\lime |C_\e|= 0$, therefore the sequence $(\bfsigma_\e(\bfphi_\e) \mathds{1}_{\O\setminus {C_\e}})$ strongly converges to $\bfsigma(\bfpsi)$ in 
$L^2( \Omega\times (0,T);\SS^3))$. 
We deduce from 
  the weak* convergence of $(\bfu_\e)$ to $\bfu$ in   $L^\infty(0,T; H^1_0(\Omega;\RR^3))$ that
\begin{equation}
\label{I12}
\begin{aligned}
\lim_{\e\to 0} I_{1\e}&
= \int_{\OT } \bfe (\bfu ): \bfsigma (\bfpsi) dxdt .
\end{aligned}
\end{equation}

\ni Under the assumptions of Theorem \ref{thinter}, noticing that $|\bfsigma_\e(\bfphi_\e)\mathds{1}_{\Omega\setminus C_\e}|=|\bfsigma_\e(\bfpsi)\mathds{1}_{\Omega\setminus C_\e}|\le C\mu_{0\e}$ and  taking  (\ref{mulambdaintermediaire}), (\ref{apriori}), (\ref{split}) into account, we  get 

\begin{equation}
\begin{aligned}
 \limsup_{\e\to0} I_{1\e} \le  \limsup_{\e\to0}  C \mu_{0\e}^{\frac{1}{2}}\lp  \int_{\OT } \mu_{0\e}|\bfe(\bfu_\e)(\tau)|^2 dxdt\rp^{\frac{1}{2}}=0.
\end{aligned}
\label{I13}
\end{equation}
 
 \ni 
By  (\ref{apriori}),  and (\ref{sigphiC2}), we have

$$
I_{2\e} \le \lp \int_{\OT} |\bfe(\bfu_\e)|^2 dxdt\rp^{\frac{1}{2}}\lp \int_{C_\e\setminus B_\e\times(0,T)}  \lb \bfsigma_\e( \bfphi_{\e}(x,t) ) \rb^2 dxdt \rp^{\frac{1}{2}}
\le  C\lp  \frac{\mu_{0\e}}{\alpha_\e } \frac{r_\e}{\e}\rp^{\frac{1}{2}}.
$$

\ni Under the assumptions of Theorem \ref{thstiff} (resp. Theorem \ref{thinter}), we deduce from  (\ref{mulambda}) and  (\ref{alpha2})    that

\begin{equation}
\begin{aligned}
 \lim_{\e\to0} I_{2\e}  =0.
\end{aligned}
\label{limI2=0stiffinter}
\end{equation}

\ni   
Collecting
   (\ref{lim12}),     (\ref{I12}),  \eqref{I13}, \eqref{limI2=0stiffinter},  and  (\ref{I32}),   by passing    to the limit as $\e \to 0$ in (\ref{IP}),     we obtain, under the assumptions of Theorem \ref{thstiff},

\begin{equation}
\begin{aligned} &   \int_{\OT } (\rho+\ru n) \bfu   \cdot {\partial^2 \bfpsi \over \partial t^2}  dxdt    + \int_{\OT} \hskip-0,5cm (\rho+\ru n)\lp  \bfa_0 \cdot {\partial \bfpsi
\over \partial t}( 0  )-  \bfb_0 \cdot  \bfpsi( 0 )\rp dx 
\\&    +\int_{\OT}  \bfe ({\bfu }):\bfsigma (\bfpsi) dxdt  
 +\I_{n,k,\kappa}(\bfu, \bfpsi)   =\int_{\OT }(\rho+\ru n) \bff \cdot \bfu dxdt .
\end{aligned}
   \label{Elimstiff}
\end{equation}

\ni and, under the assumptions of Theorem \ref{thinter},

\begin{equation}
\begin{aligned} &   \int_{\OT } (\rho(1-\vartheta n)+\ru n) \bfu   \cdot {\partial^2 \bfpsi \over \partial t^2}  dxdt  
\\&\qquad  + \int_{\OT} \hskip-0,5cm (\rho(1-\vartheta n)+\ru n)\lp  \bfa_0 \cdot {\partial \bfpsi
\over \partial t}( 0  )-  \bfb_0 \cdot  \bfpsi( 0 )\rp dx 
 +\I_{n,k,\kappa}(\bfu, \bfpsi)
 \\&   \hskip5cm   =\int_{\OT }(\rho(1-\vartheta n)+\ru n) \bff \cdot \bfu dxdt .
\end{aligned}
   \label{Eliminter}
\end{equation}
{\color{black}
\ni 
The variational formulation   \eqref{Eliminter}, joined with \eqref{totstiff},  is equivalent to \eqref{Pvar2}, where

\begin{equation}
\begin{aligned} &  H_{n,k,\kappa}\!=\!\la \bfpsi\! \in\! L^2(\O;\RR^3)\lb
\begin{aligned} &n\psi_1=n\psi_2=0 & &    \hbox{if } k=+\infty;\  
\\& n\bfpsi=0 & & \hbox{if } \kappa=+\infty 
\end{aligned} \rpt
\ra\!; \  
\\&V_{n,k, \kappa}\!=\!\la \bfpsi\!\in\! 
H_{n,k,\kappa}\lb \begin{aligned}&\psi_1,\psi_2\! \in \! L_n^2(0,L; H_0^1(\Omega' )) & &\hbox{if } 0<k  
\\&  \psi_3\in  L_n^2(0,L; H_0^2(\Omega' )) & &\hbox{if } 0<\kappa,
\end{aligned} \rpt\ra  ; 
\\& (\bfu,\bfpsi)_{H_{n,k, \kappa}}=
  \int_\O (\rho(1-\vartheta n)+\ru n) \bfu   \cdot \bfpsi  dx ; \ 
 \\& 
   ((\bfu,\bfpsi))_{V_{n,k, \kappa}}= (\bfu,\bfpsi)_{H_{n,k, \kappa}}+ \I_{n,k, \kappa}(\bfu,\bfpsi); \quad a_{n,k, \kappa}(\bfu,\bfpsi)= \I_{n,k, \kappa}(\bfu,\bfpsi);
\end{aligned}
\label{datainter33}
\end{equation}
\begin{equation}
\begin{aligned} &  h_{n,k,\kappa}= \H_{n,k,\kappa}(\bff),\ \xi_{0,n,k,\kappa}= \H_{n,k,\kappa}(\bfa_0),\ \xi_{1,n,k,\kappa}= \H_{n,k,\kappa}(\bfb_0),\ 
   \\&  \H_{n,k,\kappa}(\bfg):=  \la\begin{aligned}& \bfg & & \hbox{if } 0<k<+\infty,
   \\&(g_1 \mathds{1}_{\{n=0\}}, g_2\mathds{1}_{\{n=0\}},g_3), & & \hbox{if } k=+\infty, \ \kappa<+\infty,
   \\& \bfg \mathds{1}_{\{n=0\}} & & \hbox{if } \kappa=+\infty.
   \end{aligned}\rpt
\end{aligned}
\label{datainter33bis}
\end{equation}

\ni The  variational formulation \eqref{Elimstiff}, joined with \eqref{totstiff},  is  equivalent to \eqref{Pvar2}, with data deduced from \eqref{datainter33},
\eqref{datainter33bis}
 by substituting 
$\widetilde V_{n,k,\kappa}$ and $\widetilde a_{n,k,\kappa}$ for $V_{n,k,\kappa}$ and $a_{n,k,\kappa}$, where 

\begin{equation}
\begin{aligned} &  
\widetilde V_{n,k,\kappa}:=   V_{n,k,\kappa}\cap H^1_0(\Omega;\RR^3);\   ((\bfu,\bfpsi))_{\widetilde V_{n,k, \kappa}}\hskip-0,2cm =  ((\bfu,\bfpsi))_{V_{n,k, \kappa}}+ \int_\Omega \bfnabla \bfu\cdot \bfnabla\bfpsi dx,
\\& \widetilde a_{n,k,\kappa}(\bfu,\bfpsi):= a_{n,k,\kappa} (\bfu,\bfpsi)+ \int_\O  \bfe ({\bfu }):\bfsigma (\bfpsi) dx.
\end{aligned}
\label{tildeVa}
\end{equation}

\ni  The assumptions of Theorem \ref{thDautrayLions} are satisfied in both cases,  guaranteeing existence,  uniqueness and  regularity properties 
 of the solution.
  Finally, by integrations by parts, it  is easy to check that the  variational problems \eqref{Elimstiff}, \eqref{Eliminter}, associated with  \eqref{totstiff},  are equivalent 
 to the problems announced in theorems \ref{thstiff}, \ref{thinter}. 

 \begin{remark}\label{remI22} The assumption stated in the first line of (\ref{alpha2}) is   employed  to derive (\ref{I2}) and  requires (\ref{repetit}).
 The case $\mu_{0\e}=\mu>0$, $\vartheta>0$ $k=+\infty$ is  open.
\end{remark}

\begin{remark}[Multiphase case]\label{remmultistiffinter}
Theorems \ref{thstiff}, \ref{thinter}  can be extended to the case of  $m$ distributions $B_\e^{[s]} $ ($s\in \{1,..,m\}$)  of parallel disjoint homothetical layers 
of   thickness $r_\e^{[s]}$,  Lam\'e coefficients $\l_{1\e}^{[s]}, \mu_{1\e}^{[s]}$,  and mass density $\frac{\e}{r_\e^{[s]}} \ov\rho_1^{[s]}$, 
 defined in terms of a finite subset $\omega_\e^{[s]}$ of $(0,L)$ and   $r_\e^{[s]}$ by a formula like \eqref{defBe}.   
The sets $\omega_\e^{[s]}$ are disjoint and their union 
 $\omega_\e:= \bigcup_{s=1}^m \omega_\e^{[s]}$
  satisfies \eqref{condomegae}, which implies  that  the minimal distance between  two distincts points  of $\omega_\e$ is equal to $\e$. 
  We suppose    that $\e >r_\e^{[s]}(1+\delta),\  \forall s\in \{1,..,m\} $ for some $\delta>0$ and set $\vartheta^{[s]} := \lim_{\e\to 0} \frac{r_\e^{[s]}}{\e}$.
The Lam\'e coefficients in $\Omega \setminus \bigcup_{s=1}^m B_\e^{[s]}$ are assume to be constant  and denoted by $\l_{0\e}$, $\mu_{0\e}$. 

When $\l_{0\e}$, $\mu_{0\e}$ satisfy \eqref{mulambdaintermediaire} and  each sequence $(n_\e^{[s]})$ strongly converges to $n^{[s]}$ in $L^1(\O)$,   the solution    to \eqref{Pe}
weakly*  converges  in $L^\infty(0,T; L^2(\O;\RR^3))$
to the unique solution to the   problem \eqref{Pvar2}, where the  data are deduced from \eqref{datainter33} as follows:

\begin{equation}
\label{datamulti}
\begin{aligned}
&H:= \bigcap_{s=1}^m H_{n^{[s]},k^{[s]},\kappa^{[s]}}; \quad 
 (\bfu,\bfpsi)_H=
  \int_\O (\rho(1-\sum_{s=1}^m \vartheta^{[s]} n^{[s]})+\ru^{[s]} n^{[s]}) \bfu   \cdot \bfpsi  dx
\\&V:= \bigcap_{s=1}^m V_{n^{[s]},k^{[s]},\kappa^{[s]}};\quad   ((\bfu,\bfpsi))_{V }= (\bfu,\bfpsi)_H+ \sum_{s=1}^m \I_{n^{[s]},k^{[s]}, \kappa^{[s]}}(\bfu,\bfpsi),
\\&a (\bfu,\bfpsi)=\sum_{s=1}^m a_{n^{[s]},k^{[s]},\kappa^{[s]}}(\bfu,\bfpsi), 
\\& h = \H (\bff),\ \xi_{0 }= \H (\bfa_0),\ \xi_{1 }= \H (\bfb_0),\ 
  \\&(\H (\bfg)(x))_\a =  \la\begin{aligned}&0 && \hbox{if } \exists s\in \{1,..,m\},\, n^{[s]}(x)>0\, \hbox{ and } k^{[s]}=+\infty,
  \\& g_\a & &\hbox{otherwise},  \qquad (\a\in \{1,2\}),
  \end{aligned}\rpt
  \\&(\H (\bfg)(x))_3=  
    \la\begin{aligned}&0 && \hbox{if } \exists s\in \{1,..,m\},\, n^{[s]}(x)>0\, \hbox{ and } \kappa^{[s]}=+\infty,
  \\& g_3(x) & &\hbox{otherwise.}
   \end{aligned}\rpt
\end{aligned}
\end{equation}

  When $\l_{0\e}$, $\mu_{0\e}$ satisfy  \eqref{mulambda}, and when $\vartheta^{[s]}=0$ for each $s\in \{1,..,m\}$, 
    the
solution  to  (\ref{Pe}) weakly* converges   in $L^\infty(0,T; $ $H^1_0(\O;\RR^3))$ 
 to the unique
solution to \eqref{Pvar2}, with data $\widetilde H$, $\widetilde V$, $\widetilde a$...  deduced from $H$, $V$, $a$...  defined in \eqref{datamulti} as follows:

\begin{equation}
\label{datamultistiff}
\begin{aligned}
&\widetilde H:=  H; \quad \widetilde V = V \cap H^1_0(\O;\RR^3); \quad ((\bfu,\bfpsi))_{\widetilde V}:= ((\bfu,\bfpsi))_V + \int_\O \bfnabla\bfu\cdot\bfnabla \bfpsi dx;
\\& \widetilde a(\bfu,\bfpsi)= a(\bfu,\bfpsi)+ \int_\O \bfe(\bfu):\bfsigma(\bfpsi) dx; \quad (\widetilde h, \widetilde \xi_0, \widetilde \xi_1):= (h,\xi_0, \xi_1). 
\end{aligned}
\end{equation}
   
  \end{remark}

 \begin{remark}[Elliptic case]\label{remequilibriumstiffinter} When $\l_{0\e}$, $\mu_{0\e}$ satisfy  \eqref{mulambda}, and when $\vartheta^{[s]}=0$ for each $s\in \{1,..,m\}$, 
 the solution $\bfu_\e$ to the equilibrium problem

 \begin{equation} 
  - \bfdiv(\bfsigma_\e( \bfu_\e)) =
  \bff \quad \hbox{ in } \quad \O ,
\ \quad \bfu_\e\in   H_0^1(\O, \RR^3) ,  \quad  \bff \in  L^2(\O, \RR^3), \label {Pelliptique2} 
\end{equation}

\ni  is bounded in $H^1_0(\O;\RR^3)$ and weakly converges to the unique field $  \bfu  \in \widetilde V$ satisfying  $ a(\bfu , \bfpsi)= (\bff, \bfpsi)_{\widetilde H}, \  \forall \bfpsi \in \widetilde V$, where $\widetilde V $ is the Hilbert space    and $\widetilde a(.,.)$   the continuous coercive bilinear form on $\widetilde V$ given by \eqref{datamultistiff}. 
 
If  $\l_{0\e}$, $\mu_{0\e}$ satisfy \eqref{mulambdaintermediaire},    each sequence $(n_\e^{[s]})$ strongly converges to $n^{[s]}$ in $L^2(\O)$, 
and   $\bfu_\e$ is bounded in $L^2(\Omega;\RR^3)$,  then $\bfu_\e$  weakly converges, up to a subsequence,  to   some $\bfu\in V$ verifying  
$a(\bfu , \bfpsi)= (\bff, \bfpsi)_{ H}\  \forall \bfpsi \in  V$,
 with  $H,V,a(.,.)$  defined by \eqref{datamulti}. In this case, 
 the non-negative bilinear form $a(.,.)$  may fail to be coercive on $L^2(\O;\RR^3)$ and the sequence $\bfu_\e$  to be bounded in $L^2(\O;\RR^3)$. 
These coercivity   
and
 boundedness
 are guaranteed 
by the existence of  $s\in \{1,..,m\}$ 
 and $c>0$ such that 
 $\kappa^{[s]}>0$ and  $n^{[s]}_\e\ge c$ a.e. in $\O_\e:=\bigcup_{i\in Z^{[s]}_\e} \lp \e i-\frac{\e}{2}, \e i+\frac{\e}{2}\rc$ (see \eqref{defZe}). 
(Notice that if  the second assumption in \eqref{condomegae} is  replaced by 
 $\min_{j,j'\in J_\e, j\not= j'} |\omega_\e^{j}-\omega_\e^{j'}|= \eta \e$ for some arbitrarily  fixed $\eta\in \lp 0, \frac{1}{2}\rp$, our proofs are unchanged and   $n_\e^{[s]}\ge c
 \mathds{1}_{\O_\e} $ 
  does not imply that 
 $B_\e^{[s]}$ is $\e$-periodic).

 \ni{\it Sketch of the proof. } Let $s$ be such that $\kappa^{[s]}>0$. 
 The bilinear form associated with \eqref{Pelliptique2}, namely $a_\e(\bfvarphi, \bfpsi) = \int_\Omega  \bfe(\bfvarphi):\bfsigma_\e(\bfpsi) dx\  \forall (\bfvarphi, \bfpsi)\in \lp H^1_0(\O;\RR^3)\rp^2$,  satisfies,
by  \eqref{Pe}, \eqref{kkappa}, and \eqref{mulambdaintermediaire}  

 \begin{equation}
\label{estimae}
\begin{aligned}
a_\e(\bfvarphi, \bfvarphi)&\ge C  \int_\Omega  \e^2 \lb \bfe(\bfvarphi)\rb^2  dx +  C \int_\O \lb  \frac{1}{  r_\e^{[s]} } \bfe(\bfvarphi)\rb^2dm_\e^{[s]}. 
\end{aligned}
\end{equation}

 \ni Let $\bfu_\e$ be a sequence in $H^1_0(\O;\RR^3)$,  and let  $m_\e^{[s]}$, $\hat\bfv_\e^{[s]}$,  $\ov\bfv_\e^{[s]}$ be defined by substituting $\omega_\e^{[s]}$ for $\omega_\e$ in \eqref{defme}, \eqref{defhatve}, \eqref{defovve}. 
We have, since $n_\e^{[s]}\ge c \mathds{1}_{\O_\e} $, 
 
  \begin{equation}
\nonumber
\begin{aligned}
  \int_{\O} \lb  \bfu_\e\rb^2dx & \le     \int_{\O\setminus \O_\e} \lb  \bfu_\e\rb^2dx +  C\int_{\O_\e} \lb n_\e^{[s]} \bfu_\e-\ov\bfv_\e^{[s]}\rb^2dx + C\int_\O \lb   \ov\bfv_\e^{[s]}\rb^2dx.
\end{aligned}
\end{equation}

\ni Looking back at \eqref{neue-ovve}, and using the fact that $\bfu_\e$ vanishes on $\partial \O$, we obtain

\begin{equation}
\nonumber
\begin{aligned}
  \int_{\O\setminus \O_\e} &\lb  \bfu_\e\rb^2dx +  C\int_{\O_\e} \lb n_\e^{[s]} \bfu_\e-\ov\bfv_\e^{[s]}\rb^2dx 
\\&  \le      C\e^2  \int_{\Omega\setminus \O_\e} 
 \lb \frac{\partial \bfu_\e}{\partial x_3} \rb^2(\tau) dx+
   C\e^2 \sum_{i\in Z^{[s]}_\e}    \int_{\Omega'\times\lp \e i-\frac{\e}{2}, \e i+\frac{\e}{2} \rc} 
 \lb \frac{\partial \bfu_\e}{\partial x_3} \rb^2(\tau) dx
 \\&\le C\e^2 \int_\O |\bfnabla \bfu_\e|^2dx \le C\e^2 \int_\O |\bfe( \bfu_\e)|^2dx , 
\end{aligned}
\end{equation}
 
\ni yielding $\int_{\O} \lb  \bfu_\e\rb^2dx   \le C\e^2 \int_\O |\bfe(\bfu_\e)|^2 dx+ C\int_\O \lb   \ov\bfv_\e^{[s]}\rb^2dx $.  On the other hand, by  \eqref{ovvdx=hatvdm}, \eqref{estimhatve}, \eqref{key}, 
 
 \begin{equation}
\nonumber
\begin{aligned}
 \int_\O \lb   \ov\bfv_\e^{[s]}\rb^2dx &=  \int_\O \lb   \hat\bfv_\e^{[s]}\rb^2dm_\e^{[s]}\le C\int_\O \lb \bfu_\e-  \hat\bfv_\e^{[s]}\rb^2dm_\e^{[s]}+ C\int_\O \lb   \bfu_\e \rb^2dm_\e^{[s]}
    \\& \le  C\e r_\e^{[s]}\int_\O \lb \bfe(\bfu_\e)\rb^2dm_\e^{[s]}+ C\int_\O \lb  \frac{1}{  r_\e^{[s]} } \bfe(\bfu_\e)\rb^2dm_\e^{[s]},\end{aligned}
\end{equation}

\ni therefore, for all $\bfu_\e \in H^1_0(\O;\RR^3)$, $ \int_{\O } \lb  \bfu_\e\rb^2dx   \le   C  a_\e(\bfu_\e,\bfu_\e)$.
%
%
%
In the particular case when  $\bfu_\e$ is the solution to \eqref{Pelliptique2}, we infer $\int_\O \lb   \bfu_\e\rb^2dx  \le  C  \int_\O  \bff\cdot\bfu_\e dx \le C \lp \int_\O \lb   \bfu_\e\rb^2dx\rp^{\frac{1}{2}}, $
%
   hence $(\bfu_\e)$ is bounded in $L^2(\O;\RR^3)$. 
 We choose a smooth field $\bfpsi\in V$ and consider the associated   sequence of test field  $\bfphi_\e$ used for the proof of the multiphase case, whose construction is similar to  
 \eqref{defphiestiff}. 
  Repeating the argument of \cite[p. 40, $(iii)\Rightarrow (i)$]{BeSiam}, we find that $ a(\bfpsi, \bfpsi)=\lim_{\e\to0}  a_\e(\bfphi_\e,\bfphi_\e)
  \ge c\lime  \int_\O \lb  \bfphi_\e\rb^2dx=\int_\O \lb \psi \rb^2dx$. 
  By a density argument, we deduce 
 $\int_\O \lb \psi \rb^2dx\le Ca(\bfpsi, \bfpsi) \ \forall \bfpsi\in V$.

 \end{remark}
 
 }

 \section{Appendix}\label{appendix}
  
  \ni A common step in   the proofs of theorems \ref{thstiff}, \ref{thinter}, and   \ref{th} 
lies in  the computation of the limit of the sequence  $(I_{3\e})$  defined by  (see (\ref{defphie}), (\ref{split}))
  
   \begin{equation}
\begin{aligned}
I_{3\e}:= \int_{  {B_\e} \times (0,T)}     \bfe(\bfu_\e): \bfsigma_\e( \wideparen\bfpsi_{\e} ) dxdt,
\end{aligned}
\label{defI3}
\end{equation}

\ni where $\bfu_\e$ is  the solution to (\ref{Pe})  and the oscillating test fields  $\wideparen\bfpsi_{\e} $  is defined bellow, 
in terms of    $\bfpsi \in  C^\infty ([0, T];\D(\O;    \RR^3))$      satisfying (\ref{f0}),  of  $\delta$     given by  (\ref{disjoint}),
and of the order of magnitude  of   the parameters $k$ and $\kappa$.
We   introduce the field 
$\ov  \bfpsi_\e $ given by 

\begin{equation}
\begin{aligned}
  &\ov  \bfpsi_\e (x,t  ):= \sum_{j\in J_\e} \lp \intb_{ \lp \omega_\e^{j}-\frac{r_\e}{2}, \omega_\e^{j}+\frac{r_\e}{2}\rp}
\hskip-1,5cm  \bfpsi(x_1,x_2,s_3 ,t)ds_3\rp \mathds{1}_{ \lp \omega_\e^{j}-\frac{r_\e(1+\delta)}{2}, \omega_\e^{j}+\frac{r_\e(1+\delta)}{2}\rp}
(x_3).  
\end{aligned}
\label{ovpsie0}
\end{equation} 

\ni

\ni (i) If $0<k\le+\infty$ and $\kappa=0$,  we set 

\begin{equation}
\begin{aligned}
&\wideparen\bfpsi_\e(x,t)\! :=
 \ov \bfpsi_\e \lp x,t \rp
  +r_\e \bfw_{1\e}\lp x,t, \frac{y_\e(x_3)}{r_\e}\rp,  
  \\&\bfw_{1\e}(x,t,y_3):=
\begin{pmatrix}-{\partial \ov\psi_{\e3}\over \partial x_1} y_3\cr -{\partial \ov\psi_{\e3}\over \partial x_2} y_3
\cr {-l_\e \over l_\e+2}\lp  {\partial \ov\psi_{\e1}\over \partial x_1}+{\partial \ov\psi_{\e2}\over \partial x_2} \rp y_3\end{pmatrix} .
\end{aligned}
\label{defparenpsikfini}
\end{equation}

\ni where  the function $y_\e(.)$  is defined by    (\ref{yepsnonper}). 
     

 %

\ni (ii) If $0<\kappa\le +\infty$,
we   set

\begin{equation}
\begin{aligned}
  &
\wideparen\bfpsi_\e(x,t ) :=
  \ov\bfpsi_\e \lp x,t \rp   + r_\e\bfw_{1\e}\lp x,t,\frac{y_\e(x_3)}{r_\e}\rp+ r_\e^2 \bfw_{2\e}\lp x,t,\frac{y_\e(x_3)}{r_\e}\rp
 ,
 \\&\bfw_{1\e}\lp x,t,y_3 \rp := - {\partial  \ov\psi_{\e  3} \over \partial x_1 }y_3 \bfe_1  - {\partial  \ov\psi_{\e  3} \over \partial x_2 } y_3\bfe_2 ,
\\  &
\bfw_{2\e} \lp x,t,y_3 \rp\!:=  {  l_\e \over 2(l_\e+2)} \lp \frac{ \partial^2 \ov\psi_{\e3}}{\partial x_1^2}+\frac{\partial^2\ov\psi_{\e3}}{\partial x_2^2}\rp y_3^2 \bfe_3.
\end{aligned}
\nonumber
\end{equation}

\ni It is usefull to notice that  $ \wideparen\bfpsi_\e$ is continuously differentiable in $C_\e\times (0,T)$ (see (\ref{defCe})), that 
 $ \wideparen\bfpsi_\e=0$ if  $\kappa=+\infty$ and that  for $m\in \{1,2\}$,

\begin{equation}
\begin{aligned}
 &  \lp  \lb  \wideparen\bfpsi_{\e}(x,t)   -\bfpsi(x,t)\rb+  \lb \frac{\partial^m\wideparen\bfpsi_{\e}}{\partial t^m}(x,t)   -
    \frac{\partial^m \bfpsi}{\partial t^m}(x,t)
   \rb\rp   \mathds{1}_{C_\e}(x)
 \le Cr_\e,
 \\& \lb \bfnabla  \wideparen\bfpsi_\e( x,t)\rb \mathds{1}_{C_\e} \le C .
  \end{aligned} 
\label{estimpsi}
\end{equation}






\begin{lemma}\label{lemI3}
Let $\bfu_\e$ be the solution to (\ref{Pe}).
Let $I_{3\e}$ be defined by (\ref{defI3}) in terms of 
$ \wideparen\bfpsi_{\e}$ described above.
 Then, 

\begin{equation}
\label{I32}
\begin{aligned}
\lim_{\e\to 0} I_{3\e}& \!=\I_{n,k,\kappa}(\bfv,\bfpsi) \!:= \!
%
%
 \la  \begin{aligned}
 &k \int_{\OT}  \hskip-0,5cm  \bfe_{x'}(\bfv'): \bfsigma_{x'}(\bfpsi') n dx dt \ & & \hbox{if } 0<k<+\infty,
 \\&  {\kappa  \over6}\!\int_{\OT }\hskip-0,5cm  \bfH(v_3): \!
\bfH^{\bfsigma}(\psi_3 ) n  dxdt 
 & &  \hbox{if } 0<\kappa<+\infty,
\\& 0 & & \hskip-3cm \hbox{if } (k,\kappa) = (+\infty ,0)    \hbox{ or   } \kappa=+\infty ,
  \end{aligned}\rpt
  \end{aligned}
\end{equation}

\ni  the operators $\bfH$,  $\bfH^{\bfsigma}$
being defined by 

\begin{equation}
\begin{aligned} 
&\bfH (\psi  ):=\begin{pmatrix}{\partial^2
\psi\over \partial x_1^2} 
&  {\partial^2
 \psi\over \partial x_1\partial x_2} &0 
\cr  {\partial^2
\psi\over \partial x_1\partial x_2  }\!\!\!\!\!\!&   {\partial^2
\psi \over \partial x_2^2}  &0
\\0&0&0
\end{pmatrix}; 
\\&
 \bfH^{\bfsigma}(\psi ):=  \begin{pmatrix}2{ l +1\over l +2}{\partial^2
 \psi \over \partial x_1^2}+{ l \over l +2} {\partial^2
 \psi \over \partial x_2^2} 
&  {\partial^2
 \psi \over \partial x_1\partial x_2} & 0
\cr  {\partial^2
 \psi \over \partial x_1\partial x_2  }&  { l \over l +2} {\partial^2
 \psi \over \partial x_1^2}+2{ l +1\over l +2}{\partial^2
 \psi \over \partial x_2^2}& 0\cr 0&0&0
\end{pmatrix}. 
 \end{aligned}
 \label{sigmab0}
\end{equation}

\end{lemma}


\begin{proof}
{\bf Case $0<k<+\infty$.} 
We easily check that 
\begin{equation}
\begin{aligned}
&\ln    \bfpsi\! -\! \ov\bfpsi_\e   \rn_{L^\infty(B_\e\times(0,T))}  + \ln {\partial (\bfpsi-\ov\bfpsi_\e)\over \partial x_\a}  \rn_{ L^\infty(B_\e\times(0,T))} \le
 Cr_\e  \quad (\a\in \{1,2\}),
 \\&  \ln {\partial^2 (\bfpsi-\ov\bfpsi_\e)\over \partial x_\a\partial x_\beta}  \rn_{ L^\infty(B_\e\times(0,T))} \le
 Cr_\e  \quad (\a, \beta\in \{1,2\}).
 \end{aligned}
\label{estimpsiovpsi0}
\end{equation}

\ni  A straightforward computation yields (see (\ref{Pe}), (\ref{lmu}))

 \begin{equation}
   \label{sigmaephie=0}
\begin{aligned}
&\bfsigma_\e (\wideparen\bfpsi_\e)  \mathds{1}_{B_\e} \! \! =  
 \\&  \mu_{1\e}  \!\begin{pmatrix}2 { \partial \ov\psi_{  \e1} \over \partial x_1}+ {2l_\e\over l_\e+2}\lp { \partial \ov\psi_{  \e1} \over \partial x_1}+{ \partial \ov\psi_{\e2} \over \partial x_2}\rp&\!\!\!\!\!\!\!\!  { \partial \ov\psi_{  \e1} \over \partial x_2}+{ \partial \ov\psi_{  \e2} \over \partial x_1} & 0
\cr \!\!\!\!\!\!\!\!\! { \partial \ov\psi_{  \e1} \over \partial x_2}+{ \partial \ov\psi_{  \e2} \over \partial x_1} & \!\!\!\!\!\!\!2 { \partial \ov\psi_{  \e2} \over \partial x_2}+ {2l_\e\over l_\e+2}\lp { \partial \ov\psi_{  \e1} \over \partial x_1}+{ \partial \ov\psi_{\e2} \over \partial x_2}\rp&0
\cr 0 &  0& 0
\end{pmatrix}  \mathds{1}_{B_\e}
\\&+r_\e \mu_{1\e}O(1)    .
\end{aligned}
\end{equation}

\ni
Since $\frac{r_\e}{\e} \mu_{1\e} \to k \in (0,+\infty)$  and  $l_\e\to l\in (0,\infty)$ (see 
 (\ref{lmu})),     we infer  from (\ref{estimpsiovpsi0})  that 
   
   \begin{equation} 
\begin{aligned}
 &\lime  \lb  \frac{r_\e}{\e} \bfsigma_\e( \wideparen\bfpsi_\e) 
- k \bfsigma_{x'}(\bfpsi')(x,t)\rb_{L^\infty(B_\e\times(0,T))} =0,
\end{aligned}
\nonumber
 \label{estsigmakfini0}
\end{equation} 

\ni where $ \bfsigma_{x'}(\bfpsi')$ is given  by   (\ref{exprim}).
By Corollary   \ref{corapriori} (i), the convergences     (\ref{uemetov}) are verified, thus  the sequence  $ ( \bfe_{x'}(\bfu'_\e)m_\e )$ 
 weakly* converges  in $L^\infty(0,T; \M (\ov\Omega; \SS^3) )$    to $n\bfe_{x'}(\bfv')$.
 Taking (\ref{defme}) and  (\ref{defI3}) into account,
we deduce that 

$$
\lim_{\e\to 0} I_{3\e}= \lime  \int_{  \Omega \times (0,T)}    k \bfe_{x'}(\bfu_\e): \bfsigma_{x'}( \psi' ) dm_\e dt
= k \int_{\OT}  \hskip-0,5cm  \bfe_{x'}(\bfv'): \bfsigma_{x'}(\bfpsi') n dx dt.
$$

\ni {\bf Case $(k,\kappa)=(+\infty,0)$.}  By    (\ref{f0}),    (\ref{sigmaephie=0}),  we have     
$ 
|\bfsigma_\e( \wideparen\bfpsi_{\e})\mathds{1}_{{B_\e}}| 
\le\!  C\mu_{1\e}  r_\e$, thus,   by  (\ref{kkappa}),  the second line of (\ref{apriori}),  and (\ref{defI3}),  there holds  

\begin{equation}
\begin{aligned}
I_{3\e}  & \le C  \mu_{1\e} r_\e   \int_{{B_\e}\times
(0,T)}  |\bfe(\bfu_\e)|  dxdt=  C  \mu_{1\e}\frac{ r_\e^2}{\e}   \int_{{B_\e}\times
(0,T)}  |\bfe(\bfu_\e)|  dm_\e dt
 \\ &\le C  \mu_{1\e}\frac{ r_\e^2}{\e}   \sqrt{  \int_{{B_\e}\times
(0,T)}  |\bfe(\bfu_\e)|^2 dm_\e dt  } 
 \le C  \mu_{1\e} \frac{ r_\e^2}{\e}   \sqrt{   \frac{\e}{r_\e\mu_{\e1}} }
 \le C \sqrt{\frac{r_\e^3}{\e} \mu_{1\e}} =   o(1).
  \end{aligned}
\nonumber
\end{equation}

\ni {\bf Case $0<\kappa<+\infty$.}   
A straightforward computation  gives
 
  \begin{equation}
\begin{aligned} 
&\frac{r_\e^2}{\e} \bfsigma_\e(\wideparen\bfpsi_\e   )\mathds{1}_{{B_\e}}   =
\\&-2\frac{r_\e^3}{\e}  \mu_{1\e}   \! \frac{y_\e(x_3)}{r_\e}
\!\! \begin{pmatrix}
 \!{2(l_\e+1)\over l_\e+2} {\partial^2\ov\psi_{\e3}\over \partial x_1^2} \!+\!{l_\e\over l_\e+2}  {\partial^2\ov \psi_{\e3}\over \partial x_2^2} & {\partial^2\ov\psi_{\e3}\over \partial x_1\partial x_2} & 0
 \\
  {\partial^2\ov\psi_{\e3}\over \partial x_1\partial x_2} & \hskip-0,5cm {l_\e\over l_\e+2}  {\partial^2\ov\psi_{\e3}\over \partial x_1^2}\!+\!{2(l_\e+1)\over l_\e+2}{\partial^2\ov\psi_{\e3}\over \partial x_2^2}  & 0 \cr 0&0&0
\end{pmatrix}\!\!\mathds{1}_{B_\e}
\\&\hskip9cm + r_\e O\lp\frac{r_\e^3}{\e}\mu_\e\rp
.
  \end{aligned}
  \nonumber
\end{equation}

\ni We deduce from (\ref{lmu}),   (\ref{kkappa}),  (\ref{sigmab0}), and (\ref{estimpsiovpsi0}),  that 
 
 \begin{equation}
\begin{aligned} 
&\lime \lb \frac{r_\e^2}{\e}\bfsigma_\e(\wideparen\bfpsi_\e   ) 
+2\kappa   \bfH^{\bfsigma}(\psi_3 )\frac{y_\e(x_3)}{r_\e}
 \rb_{L^\infty({B_\e}\times(0,T); \SS^3)} =0.
  \end{aligned}
 \label{estsigmakappafini}
\end{equation}

\ni By (\ref{defme}) and (\ref{defI3}), we have

   \begin{equation}
 \label{estsigmakappafini2}
\begin{aligned}
  I_{3\e}  &=  \frac{r_\e}{\e}  \int_\OT \hskip-0,8cm \bfe (\bfu_\e): \bfsigma_\e (\wideparen\bfpsi_\e )  d m_\e dt 
  =  \int_\OT \frac{1}{r_\e}\bfe (\bfu_\e): \frac{r_\e^2}{\e}\bfsigma_\e (\wideparen\bfpsi_\e )  d m_\e dt  .
\end{aligned}
\end{equation}

  \ni  Taking  Corollary   \ref{corapriori} (ii) into account,  we infer from (\ref{eusurrborne}), 
    (\ref{cvflexion}),     (\ref{estsigmakappafini}),    (\ref{estsigmakappafini2}),   that
  
   \begin{equation}  
\begin{aligned}
 & \lime  I_{3\e}    =  \lime  -  \int_\OT \frac{1}{r_\e}\bfe_{x'} (\bfu_\e): 2\kappa   \bfH^{\bfsigma}(\psi_3 )\frac{y_\e(x_3)}{r_\e}  d m_\e dt  
 \\& =\!\!-2\kappa \!\!\!\sum_{\a,\b=1}^2  \!\! \int_{\OT\times I} \hskip-1,3cm (\bfH^{\bfsigma}(\psi_3 ))_{\a\b}
\lp \! {1\over 2}  \!\lp\! {\partial \xi_\a\over \partial x_\b}\!
 + \!{\partial \xi_\b\over \partial x_\a} \!\rp\!\!(x,t)\!
 -  \!{\partial^2 v_3\over \partial x_\a \partial x_\b}(x,t)y_3\!\rp y_3 n dxdtdy_3  
\\& = {\kappa  \over6}\!\int_{\OT } \bfH(v_3) :
\bfH^{\bfsigma}(\psi_3 ) n dxdt .
\end{aligned}\nonumber
\end{equation}
\end{proof}


\end{document}